\documentclass[hidelinks,onefignum,onetabnum]{siamart251216}

\usepackage{braket,amsfonts}
\usepackage{multirow}
\usepackage{array,amssymb}
\usepackage{subfigure}
\usepackage{pgfplots}
\usepackage{booktabs}
\usepackage{algorithmic}
\usepackage{graphicx,epstopdf}
\usepackage{amsopn,enumerate}
\usepackage{xspace}
\usepackage{bold-extra}
\usepackage[most]{tcolorbox}
\usepackage{math}
\usepackage[sort]{cite}
\usepackage{mathtools}
\usepackage{enumitem}
\usepackage[format=plain,font={small,it},skip=2pt]{caption}
\usepackage{xargs}
\usepackage[colorinlistoftodos,prependcaption,textsize=tiny]{todonotes}
\usepackage{math}
\newcommand{\val}{\mathrm{val}}
\allowdisplaybreaks[1]

\AddToHook{env/figure/begin}{%
  \setlength{\subfigtopskip}{0pt}%
  \setlength{\subfigcapskip}{-3pt}% Raise the panel labels within each row.
  \setlength{\subfigbottomskip}{6pt}}
\newcommand{\performanceprofile}[1]{%
  \includegraphics[trim=0 0 0 45bp,clip,width=0.45\textwidth]{Figures/#1}}

\makeatletter
\g@addto@macro\normalsize{%
  \abovedisplayskip=5pt plus 2pt minus 1pt
  \belowdisplayskip=5pt plus 2pt minus 1pt
  \abovedisplayshortskip=0pt plus 2pt
  \belowdisplayshortskip=3pt plus 2pt minus 1pt}
\makeatother
\normalsize
\renewenvironment{equation*}{\csname gather*\endcsname}{\csname endgather*\endcsname}
\AddToHook{env/equation/begin}{\ifhmode\unskip\fi}

\newsiamremark{assumption}{Assumption}
\newsiamremark{remark}{Remark}

\title{A Trust-Region Interior-Point Stochastic Sequential Quadratic Programming Method
}

\author{Yuchen Fang\thanks{Department of Mathematics, University of California, Berkeley (\email{yc\_fang@berkeley.edu}).}
\and Jihun Kim\thanks{Department of IEOR, University of California, Berkeley (\email{jihun.kim@berkeley.edu}).}
\and Sen Na\thanks{School of Industrial and Systems Engineering, Georgia Tech (\email{senna@gatech.edu}).}
\and James Demmel\thanks{Department of EECS, University of California, Berkeley (\email{demmel@berkeley.edu}).}
\and Javad Lavaei\thanks{Department of IEOR, University of California, Berkeley (\email{lavaei@berkeley.edu}).}}

\ifpdf
\hypersetup{
pdftitle={Trust-Region Interior-Point SSQP},
pdfauthor={Y. Fang, J. Kim, S. Na, J. Demmel, J. Lavaei}
}
\fi

\renewcommand{\proofbox}{\hfill\ensuremath{\blacksquare}}

\begin{document}

\maketitle

\begin{abstract}
In this paper, we propose a trust-region interior-point stochastic sequential quadratic programming (TR-IP-SSQP) method for solving optimization problems with a stochastic objective and deterministic nonlinear equality and inequality constraints. In this setting, exact evaluations of the objective function and its gradient are unavailable, but their stochastic estimates can be constructed. In particular, at each iteration our method builds stochastic oracles whose estimates of the objective value and gradient satisfy proper adaptive accuracy conditions with a fixed probability. To handle inequality constraints, we adopt an interior-point method (IPM), {\red in which the barrier parameter is taken from a prescribed decreasing sequence and is reduced once a stochastic stationarity test is passed. Under standard assumptions, we show that, almost surely, the KKT residual vanishes along the iterations at which the barrier parameter is reduced, and every limit point of these iterates is a KKT point.} We implement the method on a subset of problems from the CUTEst test set, as well as on constrained logistic regression problems, to demonstrate its practical performance.
\end{abstract}


% REQUIRED
%\begin{MSCcodes}

%\end{MSCcodes}

\section{Introduction}\label{sec:1}

We consider solving optimization problems of the form
\begin{equation}\label{Intro_StoProb}
\min_{\bx\in\mR^d}\;f(\bx)=\mE_{\P}[F(\bx;\xi)],\quad\;\text{s.t.}\;\; c(\bx)=\b0,\quad h(\bx)\leq \b0,
\end{equation}
where $f:\mR^d\to\mR$ is continuously differentiable and defined as the expectation of a stochastic realization $F(\bx;\xi)$. The functions $c:\mR^d\to\mR^m$ and $h:\mR^d\to\mR^n$ represent continuously differentiable deterministic equality and inequality constraints, respectively. The random variable $\xi$ follows the distribution $\P$. We assume that neither the objective value $f(\bx)$ nor its gradient $\nabla f(\bx)$ can be evaluated exactly, but must instead be estimated via sampling. Problem \eqref{Intro_StoProb} arises in a wide range of applications, including optimal control, constrained machine learning, and safe reinforcement learning \citep{Betts2010Practical, Chen2018Constraint, Achiam2017Constrained, Cuomo2022Scientific}.

Interior-point methods (IPMs) constitute one of the most effective frameworks for solving inequality-constrained optimization problems. They have demonstrated remarkable success in deterministic optimization and form the foundation of widely used solvers such as Ipopt \citep{Waechter2005Implementation}, Knitro \citep{Byrd2006Knitro}, and LOQO \citep{Vanderbei1999Loqo}. Numerous studies have integrated interior-point techniques within deterministic optimization frameworks \citep{Byrd1999Interior, Byrd2000Trust, Curtis2010Interior, Waechter2005Implementation}, as well as in deterministic settings with noisy function evaluations \citep{Curtis2025Interior}.

Extensions to stochastic optimization have also emerged in recent years. For example, \cite{badenbroek2022complexity, narayanan2016randomized} developed IPMs for minimizing linear functions over convex sets. Subsequent work has extended the setting to nonconvex stochastic optimization. In particular, \cite{Curtis2023Stochastic} investigated bound-constrained problems, and their method was later extended in \cite{curtis2026single} to incorporate linear equality and nonlinear inequality constraints. Under suitable decay conditions on the barrier parameter, these methods establish liminf-type almost-sure convergence guarantees and demonstrate favorable empirical performance.

Despite these advances, existing approaches \cite{Curtis2023Stochastic, curtis2026single} still leave room for further refinement. As gradient-based methods, they may face challenges in effectively handling nonlinear constraints; and their analyses also rely on unbiased gradient estimators with bounded variance, which restricts the range of admissible sampling mechanisms. Moreover, these methods have strict conditions on the decay rate of parameter sequences, potentially increasing the effort of practical tuning; and they also enforce strict feasibility at every iteration, thereby necessitating auxiliary procedures to obtain a feasible initial point. Such requirements may introduce additional complexity in both theoretical analysis and implementation, as noted in deterministic optimization settings \citep{Wright1997Primal}.

To address nonlinear constrained optimization, Sequential Quadratic Programming (SQP) and its variants are among the most effective techniques in the deterministic setting \citep{Boggs1995Sequential, Byrd1987Trust, Heinkenschloss2014Matrix}. By performing local linearizations of the constraints at each iteration, SQP methods exhibit a strong capability for handling nonlinear constraints. Motivated by the large scale and intrinsic uncertainty of modern applications, there has been substantial interest in Stochastic SQP (SSQP) methods. In particular, recent adaptive-sampling SSQP methods \cite{Na2022adaptive, Na2021Inequality, Qiu2023Sequential, Fang2024Trust, Scheinberg2025Stochastic, Berahas2025Sequential, Fang2025High} eliminate the need for unbiased objective estimates. Instead, they construct \textit{stochastic oracles} that require objective estimates to satisfy adaptive accuracy conditions with a fixed probability. Such oracles encompass a broader range of sampling mechanisms and problem settings.

Despite significant progress for equality-constrained problems, SSQP methods for nonlinear inequality-constrained optimization are still relatively limited. Existing approaches are predominantly built on line-search methods \citep{Curtis2023Sequential, Na2021Inequality, Qiu2023Sequential, Berahas2025Retrospective}, while the designs built on trust-region methods remain largely open. Compared with line-search methods, trust-region methods offer several advantages: the step direction and step length are computed simultaneously, often enhancing robustness \citep[Chapter 4]{Nocedal2006Numerical}. Moreover, trust-region methods naturally accommodate indefinite Hessian approximations, enabling the direct utilization of curvature information. In contrast, line-search methods typically require positive definiteness of the Hessian approximation, thereby necessitating explicit Hessian modifications.

In this work, we develop a trust-region interior-point SSQP (TR-IP-SSQP) method with adaptive sampling to solve nonlinear inequality-constrained optimization problems. Our contributions are summarized as follows:

\vspace{0.15cm}
\begin{enumerate}[leftmargin=1.8em]

\item We extend the trust-region SSQP method in \citep{Fang2024Trust} to nonlinear inequality-constrained optimization. This extension is nontrivial. Interior-point methods require the slack variables to remain \textit{deterministically} positive, while their updates are \textit{stochastic}. To address this issue, we modify the step computation to explicitly incorporate the update of slack variables and introduce a fraction-to-boundary condition to guarantee the positivity requirements. While such a condition is standard in deterministic interior-point methods \citep{Byrd1999Interior, Byrd2000Trust}, to the best of our knowledge, this is its first incorporation within stochastic optimization.

\item Compared with existing stochastic interior-point methods for nonlinear optimization \citep{curtis2026single, Curtis2023Stochastic}, our method adopts an adaptive sampling strategy, thereby permitting biased objective estimates and allowing gradient estimates with unbounded variance. Furthermore, our method is developed within a \emph{relaxed-feasibility} interior-point framework: strict feasibility of the constraints is not enforced at every iteration, and thus no auxiliary procedure is required to compute a feasible initial point. {\red In addition, the barrier parameter is reduced only when a validation check on the estimated stationarity of the current barrier problem is passed.} As a result, the method eliminates the conditions on the decay rate of the barrier parameter, while the approaches in \citep{curtis2026single, Curtis2023Stochastic} require carefully controlled barrier-parameter decay to guarantee convergence.

\item Our algorithm is built upon a trust-region SSQP framework rather than line-search-based SSQP methods \citep{Na2021Inequality, Qiu2023Sequential} or gradient-based approaches \citep{curtis2026single, Curtis2023Stochastic}. The integration of trust-region and SSQP techniques enables the direct use of second-order (Hessian) information without explicit matrix modifications and provides improved handling of nonlinear objective and constraint structures. The availability of second-order information can also lead to enhanced practical performance, particularly in nonconvex optimization settings.
\end{enumerate}

\vspace{0.15cm}

Under the assumption that the objective estimates satisfy proper adaptive accuracy conditions with a high but fixed probability, we establish \textit{global almost-sure convergence}, showing that, {\red almost surely, the iterates at which the barrier parameter is reduced satisfy the KKT conditions of Problem \eqref{Intro_StoProb} asymptotically, and every limit point of these iterates is a KKT point.} This almost-sure convergence property aligns with existing guarantees for inequality-constrained stochastic optimization \cite{Na2021Inequality,curtis2026single, Curtis2023Stochastic}. We implement our algorithm on a subset of problems from the CUTEst test set \citep{Gould2014CUTEst} and on constrained {\red logistic regression} problems to demonstrate its empirical performance.

\vspace{0.15cm}

\noindent\textbf{Notation.}
We use $\|\cdot\|$ to denote the $\ell_2$ norm for vectors and the operator norm for matrices. $I$ denotes the identity matrix and $\b0$ denotes the zero matrix/vector, whose dimensions are clear from the context.  For constraints $c(\bx): \mR^d\rightarrow\mR^m$ and $h(\bx): \mR^d\rightarrow\mR^n$, we let $G(\bx)\coloneqq \nabla c(\bx)^T\in\mR^{m\times d}$ and $J(\bx)\coloneqq \nabla h(\bx)^T\in\mR^{n\times d}$ be their Jacobian matrices, respectively. We use $\bx^{(i)}$ to denote the $i$-th entry of a vector $\bx$; $c^{(i)}(\bx)$ and $h^{(i)}(\bx)$ to denote the $i$-th equality and inequality constraint, respectively; and $(\cdot;\cdot)$ to denote the vertical concatenation of vectors. 
We let $\bar{(\cdot)}$ denote stochastic estimates and define $\barg(\bx)\coloneqq \bar{\nabla} f(\bx)$ to be the estimate of $\nabla f(\bx)$. Given positive sequences $\{ a_k, b_k\}$,~$a_k = \mathcal{O} (b_k)$ indicates that there exists $C >0$ such that $a_k \leq  Cb_k$ for all large enough $k$.

\vspace{0.15cm}
\noindent\textbf{Structure of the paper.}
We introduce the interior-point method and step computation in Section \ref{sec:2}. In Section \ref{sec:3}, we introduce the probabilistic oracles and the algorithm design. We establish the global convergence guarantee in Section \ref{sec:4} and numerical experiments are presented in Section \ref{sec:5}. Conclusions are drawn in Section \ref{sec:6}.

\section{Preliminaries}\label{sec:2}

The Lagrangian of Problem \eqref{Intro_StoProb} is given by $\L(\bx,\blambda,\btau)=f(\bx)+\blambda^T c(\bx)+\btau^T h(\bx)$, where $\blambda$ and $\btau$ denote the Lagrange multipliers associated with constraints $c(\bx)$ and $h(\bx)$, respectively. Under standard constraint qualifications, such as linear independence constraint qualification (LICQ), finding a first-order stationary point of Problem \eqref{Intro_StoProb} is equivalent to finding a triple $(\bx_*,\blambda_*,\btau_*)$ satisfying the KKT conditions:
\begin{equation}\label{KKT1}
\begin{aligned}
\nabla f(\bx_*)+G(\bx_*)^T\blambda_*+ J(\bx_*)^T\btau_* & =\b0,\quad
c(\bx_*) = \b0, \\
h(\bx_*)  \leq \b0, \quad
\btau_*  & \geq \b0, \quad
\btau_*^{(i)}h^{(i)}(\bx_*) = 0, \; i = 1,2,\cdots,n.    
\end{aligned}
\end{equation}
Following the interior-point paradigm, we associate Problem \eqref{Intro_StoProb} with a barrier problem by introducing a slack variable $\bs=(\bs^{(1)},\bs^{(2)},\ldots,\bs^{(n)})^T\in\mathbb{R}^n$ and augmenting the objective with a log-barrier term, where $\theta> 0$ is a barrier parameter:
\begin{equation}\label{BarrierProblem}
\begin{aligned}
\min_{\bx\in \mR^d,\bs\in\mR^n} &\quad  f(\bx)-\theta \sum_{i=1}^n \ln \bs^{(i)}\\
\text{s.t. }&\quad c(\bx)=\b0,\quad h(\bx)+\bs = \b0, \quad \bs^{(i)}> 0, \; \forall i = 1,2,\cdots,n.    
\end{aligned}	
\end{equation}
Our algorithm repeatedly solves the barrier problem \eqref{BarrierProblem} approximately, while driving the barrier parameter $\theta$ to zero. The Lagrangian of Problem \eqref{BarrierProblem} is $\L_{\theta}(\bx,\bs,\blambda,\btau)=f(\bx)-\theta \sum_{i=1}^n \ln \bs^{(i)}+\blambda^Tc(\bx)+\btau^T(h(\bx)+\bs)$. For a given barrier parameter $\theta$, a first-order stationary point of \eqref{BarrierProblem} is a quadruple $(\bx_\theta,\bs_\theta,\blambda_\theta,\btau_\theta)$ satisfying
\begin{equation}\label{1st_order_point}
\nabla \L_{\theta}(\bx_{\theta}, \bs_{\theta}, \blambda_{\theta},\btau_{\theta})
= \begin{pmatrix}
\nabla f(\bx_{\theta})+G(\bx_{\theta})^T\blambda_{\theta} + J(\bx_{\theta})^T\btau_{\theta}  \\
-\theta S_{\theta}^{-1}\boldsymbol{e} + \btau_{\theta}   \\
c(\bx_{\theta})  \\
h(\bx_{\theta})+\bs_{\theta} 
\end{pmatrix} = \b0,
\end{equation}
where $\boldsymbol{e}\in\mR^n$ denotes the all-one vector and $S_{\theta} = \text{diag}(\bs_\theta^{(1)},\bs_\theta^{(2)},\cdots,\bs_\theta^{(n)})$.

In deterministic optimization, interior-point methods are typically implemented using a nested-loop structure \citep{Byrd1999Interior,Byrd2000Trust,Potra2000Interior,Wright1997Primal}. The outer loop reduces the barrier parameter, while the inner loop approximately solves \eqref{BarrierProblem} for a fixed $\theta$ until a stationarity condition, such as $\|\nabla \L_{\theta}(\bx,\bs,\blambda,\btau)\|\le \epsilon$, is satisfied. In the stochastic setting, however, exact evaluations of the objective, gradient, and KKT residuals are unavailable, so such termination criteria cannot be verified exactly, which complicates the nested-loop design. To address this issue, {\red we adopt a \textit{probabilistic nested-loop} structure in which the barrier parameter takes values in a prescribed diminishing sequence $\{\beta_j\}\downarrow 0$, and is reduced to the next value once a stochastic estimate of the stationarity measure of the current barrier problem passes a validation check (see Section \ref{sec:3})}.

At iteration $k$, we denote $\nabla f_k=\nabla f(\bx_k)$, $\barg_k=\barg(\bx_k)$ and $\L_{\theta_k}=\L_{\theta_k}(\bx_k,\bs_k,\blambda_k,\btau_k)$, where $\theta_k$ is the current barrier parameter. We define $c_k,h_k,G_k,J_k,S_k$ analogously. Throughout the paper, we assume that $G_k=\nabla c(\bx_k)^T$ has full row rank, as stated in Assumption \ref{assump1}, and define the \emph{true} Lagrange multipliers at iteration $k$ as
\begin{equation}\label{def:true_Lag_para}
\btau_k=\theta_k S_k^{-1}\boldsymbol{e},\qquad
\blambda_k=-[G_kG_k^T]^{-1}G_k\bigl(\nabla f_k+\theta_k J_k^T S_k^{-1}\boldsymbol{e}\bigr).
\end{equation}
In addition, we define the residual measure of Problem \eqref{Intro_StoProb} at the $k$-th iteration as
\begin{equation}\label{def:KKT_measure}
V_k = \begin{pmatrix}
\nabla f_k + G_k^T\blambda_k + J_k^T\btau_k \\ c_k \\
\min\{-h_k,\btau_k\}
\end{pmatrix}.
\end{equation}

\subsection{TR-IP-SSQP subproblem}

At iteration $k$, given the current iterate~$(\bx_k,\bs_k)$, the barrier parameter $\theta_k$, and the trust-region radius $\Delta_k$, we formulate a TR-IP-SSQP subproblem for Problem \eqref{BarrierProblem} as
\begin{equation}\label{def:SQPsubproblem_1}
\begin{aligned}
\min_{\Delta\bx_k\in\mathbb{R}^d,\;\Delta\bs_k\in\mathbb{R}^n}\;\;
& \barg_k^T\Delta\bx_k - \theta_k \boldsymbol{e}^TS_k^{-1}\Delta\bs_k
+ \frac{1}{2}\Delta\bx_k^T\barH_k\Delta\bx_k
+ \frac{1}{2}\theta_k\Delta\bs_k^T S_k^{-2}\Delta\bs_k \\
\text{s.t.}\quad
& c_k + G_k\Delta\bx_k = \b0,\quad
h_k + \bs_k + J_k\Delta\bx_k + \Delta\bs_k = \b0, \\
& \|(\Delta\bx_k; S_k^{-1}\Delta\bs_k)\| \le \Delta_k,\quad
\bs_k + \Delta\bs_k \ge (1-\epsilon_s)\bs_k,
\end{aligned}   
\end{equation}
where $\barH_k$ is a (symmetric) approximation to the Hessian $\nabla^2_{\bx}\L_k$.

The formulation of the subproblem follows that used in deterministic interior-point methods \citep{Byrd1999Interior, Byrd2000Trust, Curtis2010Interior}. In contrast to recent stochastic interior-point methods that impose a user-specified lower-bounding sequence on the slack variables \citep{curtis2026single, Curtis2023Stochastic}, we employ the fraction-to-boundary condition $\bs_k + \Delta\bs_k \ge (1-\epsilon_s)\bs_k$ to maintain positivity of the slack variables. This design reduces the number of tuning parameters and aligns our method with deterministic interior-point methods.
By the development of our algorithm in Section \ref{sec:3}, when iteration $k$ is successful, the next iterate satisfies $\bs_{k+1}=\bs_k+\Delta\bs_k$, so the condition can equivalently be written as $\bs_{k+1}\ge(1-\epsilon_s)\bs_k$; when iteration $k$ is unsuccessful, $\bs_{k+1}=\bs_k$ and this inequality holds trivially. This mechanism prevents the iterates from approaching the boundary of the feasible region too aggressively.
In practice, $\epsilon_s$ is chosen close to one, e.g., $\epsilon_s=0.9$ or $0.99$ as suggested in \citep{Byrd2000Trust}.

\subsection{Computing the trial step}\label{sec: Computing the (rescaled) trial step}

We now describe how to compute the trial step. For notational convenience, we denote the trial step as
\begin{equation*}
{\red \bd_k \coloneqq  ( \Delta\bx_k ; \Delta \bs_k) .}
\end{equation*}
Using this notation, the TR-IP-SSQP subproblem \eqref{def:SQPsubproblem_1} can be equivalently written as
\begin{equation}\label{def:SQPsubproblem2}
\begin{aligned}
\min_{ \bd_k\in\mR^{d+n}} \;\; & \bar{\psi}_k^T \bd_k + \frac{1}{2} \bd_k^T \bar{W}_k \bd_k \\
\text{s.t. } \;\;\; & \vartheta_k+A_k \bd_k=\b0, \; \| ( \Delta\bx_k ; S_k^{-1}\Delta \bs_k)\|\leq\Delta_k,\;
S_k^{-1}\Delta \bs_k \geq -\epsilon_s \boldsymbol{e},
\end{aligned}
\end{equation}
where ${\red  \bar{\psi}_k=(\barg_k ; -\theta_k S_k^{-1}\boldsymbol{e})}, \vartheta_k = (c_k ; h_k+\bs_k)$, $ {\red \bar{W}_k=\begin{pmatrix}
\barH_k & 0 \\ 0 & \theta_k S_k^{-2}
\end{pmatrix}}$, and {\red $A_k = \begin{pmatrix}
G_k & 0 \\
J_k & I
\end{pmatrix}$}.
A potential difficulty in solving \eqref{def:SQPsubproblem2} is infeasibility, namely, the linearized constraints and the trust-region constraint have no intersection. To address this issue, we decompose the trial step into a \emph{normal step} and a \emph{tangential step} as
\begin{equation*}
{\red \bd_k =  \bw_k + \bt_k = \begin{pmatrix}
\bw^x_k \\ \bw^s_k
\end{pmatrix} + \begin{pmatrix}
\bt^x_k \\ \bt^s_k
\end{pmatrix},}
\end{equation*}
where $ \bw_k\in\operatorname{im}(A_k^T)$ is the normal step and $ \bt_k\in\ker(A_k)$ is the tangential step. We use $\bw^x_k$, $\bw^s_k$ and $\bt^x_k$, $\bt^s_k$ to denote the components of the normal and tangential steps corresponding to $\Delta\bx_k$ and $\Delta\bs_k$, respectively.

To enforce the trust-region constraint, we decompose the trust-region radius $\Delta_k$~into two components that separately control the lengths of the (rescaled) normal and tangential steps. {\red Specifically, denoting $R_k = \begin{pmatrix}
I & 0 \\ 0 & S_k^{-1}
\end{pmatrix}$, we have $(\Delta\bx_k;S_k^{-1}\Delta\bs_k) = R_k\bd_k = R_k\bw_k + R_k\bt_k$.} Thus, we impose the constraints
\begin{equation}\label{trust-region decom}
\|R_k\bw_k\| \leq \breve{\Delta}_k \coloneqq \zeta\Delta_k, \quad   {\red \| R_k\bt_k \| \leq \hat{\Delta}_k \coloneqq \Delta_k - \|R_k\bw_k\|},
\end{equation}
where $\zeta\in(0,1)$ is a user-specified parameter. Similarly, we decompose the fraction-to-boundary condition $S_k^{-1}\Delta\bs_k  = S_k^{-1}\bw^s_k +  S_k^{-1} \bt^s_k  \geq -\epsilon_s \boldsymbol{e}$ in \eqref{def:SQPsubproblem2} as
\begin{equation}\label{fraction to boundary decom}
S_k^{-1}\bw^s_k \geq -\zeta\epsilon_s \boldsymbol{e} \quad\quad\text{and}\quad\quad S_k^{-1}\bt^s_k \geq -\epsilon_s \boldsymbol{e} - S_k^{-1}\bw_k^s.
\end{equation}
Here, we use the same parameter $\zeta$ in the above two decompositions for simplicity.

We first compute the normal step $\bw_k$. Since $A_k$ has full row rank (cf. Assumption \ref{assump1}), we define
\begin{equation*}
\bv_k := -A_k^T[A_kA_k^T]^{-1}\vartheta_k.
\end{equation*}
In the absence of trust-region and fraction-to-boundary constraints, one would set $\bw_k=\bv_k$ to satisfy $\vartheta_k + A_k\bw_k=\b0$. However, to enforce trust-region and fraction-to-boundary constraints, we scale $\bv_k$ by a factor $\bargamma_k\in(0,1]$ as $\bw_k = \bargamma_k\bv_k$. An appropriate value of $\bargamma_k$ should satisfy the trust-region constraint \eqref{trust-region decom} and the fraction-to-boundary condition \eqref{fraction to boundary decom}, which yield $\bargamma_k \leq \zeta\Delta_k/\|R_k\bv_k\|$ and  $\bargamma_k \leq \zeta\epsilon_s/\|S_k^{-1}\bv_k^s\|$, separately. We thus choose the scaling factor $\bargamma_k$ as
\begin{equation}\label{eq:Sto_gamma_k}
{\red \bargamma_k = \min\left\{
\frac{\zeta\epsilon_s}{\|S_k^{-1}\bv_k^s\|},
\frac{\zeta\Delta_k}{\|R_k\bv_k\|}, 1 \right\} .}
\end{equation}

Given the normal step, the tangential step $\bt_k$ is obtained by approximately solving
\begin{equation}\label{eq:Sto_tangential_step1}
\begin{aligned}
\min_{\bt_k\in\mR^{d+n}} \;\;& m( \bt_k)=\frac{1}{2}\bt_k^T\bar{W}_k\bt_k+(\bar{\psi}_k+\bargamma_k\bar{W}_k\bv_k)^T\bt_k\\
\text{s.t.}\;\;\; & A_k\bt_k=\b0, \; \|R_k\bt_k\|\leq\hat{\Delta}_k, \; S_k^{-1}\bt_k^s \geq -\epsilon_s\boldsymbol{e} - S_k^{-1}\bw_k^s.  
\end{aligned}
\end{equation}
It is complicated to solve \eqref{eq:Sto_tangential_step1} directly due to the coupled trust-region constraint and fraction-to-boundary condition. Therefore, we derive an auxiliary subproblem with a simpler constraint. {\red First, note that the condition $S_k^{-1}\bt_k^s \ge -\epsilon_s\boldsymbol{e}-S_k^{-1}\bw_k^s$ is satisfied whenever $\|S_k^{-1}\bt_k^s \|\le\epsilon_s-\|S_k^{-1}\bw_k^s\|$. Then, using $\|S_k^{-1}\bt_k^s \|\le \|(\bt_k^x;S_k^{-1}\bt_k^s)\| = \|R_k \bt_k\|$, we observe that the trust-region constraint and the fraction-to-boundary condition in \eqref{eq:Sto_tangential_step1} are satisfied simultaneously whenever $\|R_k \bt_k\| \le \min\{\hat{\Delta}_k,\epsilon_s-\|S_k^{-1}\bw_k^s\|\}$. In addition, since $\|R_k\bt_k\|\leq \|R_k\|\|\bt_k\|$, this in turn holds whenever $\|\bt_k\| \le \min\{\hat{\Delta}_k,\epsilon_s-\|S_k^{-1}\bw_k^s\|\}/\|R_k\|$. This motivates the auxiliary subproblem
\begin{equation}\label{eq:Sto_tangential_step_2}
\begin{aligned}
\min_{\bt_k\in\mR^{d+n}} \;\;& \widehat{m}(\bt_k)=\frac{1}{2}\bt_k^{T}\bar{W}_k\bt_k+(\bar{\psi}_k+\bar{W}_k\bw_k)^{T}\bt_k\\
\text{s.t.}\;\;\;\; & A_k\bt_k=\b0, \; \|\bt_k\| \le \frac{\min\{\hat{\Delta}_k,\epsilon_s-\|S_k^{-1}\bw_k^s\|\}}{\|R_k\|}, 
\end{aligned}		
\end{equation}}
whose objective $\widehat{m}(\cdot)$ coincides with $m(\cdot)$ in \eqref{eq:Sto_tangential_step1}. We retain the separate symbol only to track which feasible set it is being minimized over. By construction, every $\bt_k$ feasible for \eqref{eq:Sto_tangential_step_2} is feasible for \eqref{eq:Sto_tangential_step1}. Consequently, the optimal value of \eqref{eq:Sto_tangential_step_2} is no smaller~than that of \eqref{eq:Sto_tangential_step1}. In this sense, \eqref{eq:Sto_tangential_step_2} is a more conservative surrogate for \eqref{eq:Sto_tangential_step1}.

For the sake of computational efficiency, we do not need to solve \eqref{eq:Sto_tangential_step1} to optimality; instead, we only require the computed step $\bt_k$ to be feasible for \eqref{eq:Sto_tangential_step1} and to achieve at least a fraction $\kappa_{fcd}\in(0,1]$ of the Cauchy reduction of the auxiliary subproblem \eqref{eq:Sto_tangential_step_2} (such a requirement can be met by any reasonable iterative solver, e.g., one step of the projected conjugate gradient method \citep{Byrd1999Interior}):
\begin{equation}\label{ns:equ:1}
m(\bt_k) -m(\b0) \leq \kappa_{fcd}[ \widehat{m}(\bt_k^C) -\widehat{m}(\b0)], 
\end{equation}
where $\bt_k^C$ is the Cauchy point of the subproblem \eqref{eq:Sto_tangential_step_2} (see \cite[(4.11)]{Nocedal2006Numerical}  for the formula).

In what follows, we derive an explicit upper bound on the Cauchy reduction on the right-hand side of \eqref{ns:equ:1} for use in the later analysis, as is standard in existing trust-region literature \cite{Nocedal2006Numerical, Fang2024Fully, Fang2024Trust}. In particular, for any $\bt_k\in\ker(A_k)$, there exists a unique $\bu_k\in\mR^{d-m}$ such that $\bt_k=Z_k\bu_k$, where $Z_k\in\mR^{(d+n)\times (d-m)}$ is an orthonormal basis for $\ker(A_k)$ satisfying $Z_k^{T}Z_k=I$ and $Z_kZ_k^{T}=P_k$, with $P_k = I-A_k^{T}[A_kA_k^{T}]^{-1}A_k$ denoting the orthogonal projection matrix onto $\ker(A_k)$. With this reparameterization, it follows that $\bt_k^C = Z_k\bu_k^C$, where $\bu_k^C$ is the Cauchy point of the problem:
{\red \begin{equation}\label{eq:Sto_tangential_step2}
\begin{aligned}
\min_{\bu_k\in\mR^{d-m}} \;\;& \tilde{m}(\bu_k)=\frac{1}{2}(Z_k\bu_k)^{T}\bar{W}_kZ_k\bu_k +[Z_k^{T}(\bar{\psi}_k+\bar{W}_k\bw_k)]^{T}\bu_k\\
\text{s.t.}\;\;\;\; &  \| \bu_k\| \leq \frac{\min\{\hat{\Delta}_k,\epsilon_s-\|S_k^{-1}\bw_k^s\|\}}{\|R_k\| }.
\end{aligned}
\end{equation}}
The above problem is simply rewritten from \eqref{eq:Sto_tangential_step_2} under the substitution $\bt_k=Z_k\bu_k$, from which we know $\widehat{m}(\bt_k) = \tilde{m}(\bu_k)$. As $\bt_k^C = Z_k\bu_k^C$, we further have
$\widehat{m}(\bt_k^C) = \tilde{m}(\bu_k^C)$. Therefore, combining
\eqref{ns:equ:1} and \cite[Lemma 4.3]{Nocedal2006Numerical} leads to
{\red \begin{align}\label{eq:cauchy1}
m(\bt_k) & -m(\b0) \leq \kappa_{fcd}[\widehat{m}( \bt_k^C)-\widehat{m}(\b0) ] = \kappa_{fcd}[\tilde{m}(\bu_k^C)-\tilde{m}(\b0)] \nonumber \\ 
& \leq -\frac{\kappa_{fcd}}{2}\|Z_k^T(\barpsi_k+\bargamma_k\bar{W}_k\bv_k)\|  \cdot\min\left\{\frac{\hat{\Delta}_k}{\|R_k\|}, \frac{\epsilon_s - \|S_k^{-1}\bw_k^s\|}{\|R_k\|},\frac{\|Z_k^T(\barpsi_k+\bargamma_k\bar{W}_k\bv_k)\|}{\|Z_k^T\bar{W}_kZ_k\|}\right\}.
\end{align}}

We mention that \eqref{eq:Sto_tangential_step2} is introduced solely for theoretical analysis. In practice, the method approximately solves \eqref{eq:Sto_tangential_step1} or \eqref{eq:Sto_tangential_step_2}, and the explicit construction of the null-space basis $Z_k$ is never required computationally.

\section{TR-IP-SSQP Method}\label{sec:3}

In this section, we describe our algorithm design, which is summarized in Algorithm \ref{Alg:STORM}. At every iteration $k$, the algorithm requires samples $\xi^g_k, \xi^f_k,\xi_{k}^{\val}$ to estimate the objective gradient, the objective values, and the~objective gradient used to validate the reduction of the barrier parameter, respectively.$\;\;$

We first formalize the randomness of the algorithm. We define $\F_{-1}\subset\F_0\subset\F_1\subset\cdots$ as a filtration of $\sigma$-algebras. The $\sigma$-algebra $\F_{k-1}$ is generated by $\{\xi^g_j, \xi^f_j,{\red \xi_{j}^{\val}}\}_{j=0}^{k-1}$ and thus contains the randomness before the $k$-th iteration. We define $\F_{k-1/2}$ as the $\sigma$-algebra generated by $\{\xi^g_j,\xi^f_j,{\red\xi^{\val}_j}\}_{j=0}^{k-1}\cup\{\xi^g_k\}$, which adds the
randomness of $\bar g_k$ to that in $\F_{k-1}$. {\red We also define $\F_{k-1/4}$ as the $\sigma$-algebra generated by $\{\xi^g_j,\xi^f_j,\xi^{\val}_j\}_{j=0}^{k-1}\cup\{\xi^g_k,\xi^f_k\}$, which further adds the randomness of $\bar f_k$ and $\bar f_{s_k}$.} Here, $\xi^f_k$ collects the samples used for both $\barf_k$ and $\barf_{s_k}$. When $\barH_k$ is constructed from samples, these samples are contained in $\xi_k^g$, so that $\barH_k\in\F_{k-1/2}$. From these definitions, we have $\F_{k-1}\subset \F_{k-1/2}\subset \F_{k-1/4}\subset\F_{k}$, with $\F_{-1}=\sigma(\bx_0,\bs_0)$.

Now, we introduce the probabilistic oracles constructed in every iteration.

\subsection{Probabilistic oracles}

In each iteration, the zeroth- and first-order probabilistic oracles generate estimates of the objective value $\bar{f}(\bx,\xi)$ and the objective gradient $\bar{g}(\bx,\xi)$, respectively, where $\xi$ denotes a random variable defined on some probability space. These estimates are required to satisfy certain adaptive accuracy conditions with a high but fixed probability. 

Let $\kappa_g>0$ and $p_g,p_f\in(0,1)$ be user-specified parameters. In the $k$-th iteration, $\Delta_k$ denotes the trust-region radius and $\barepsilon_k$ denotes the reliability parameter, both of which are updated from the $(k-1)$-th iteration.

The probabilistic first-order oracle is defined at $\bx_k$ as follows, which indicates that the noise of the gradient estimate is $\mO(\Delta_k)$ with probability at least $1-p_g$.

\begin{definition}[Probabilistic first-order oracle]\label{def:Probabilistic first-order oracle}
Given $\bx_k$, the oracle computes $\barg_k\coloneqq \barg(\bx_k,\xi_k^g)$, an estimate of the objective gradient $\nabla f_k$, such that
\begin{equation*}
\A_k=\{\|\barg_k-\nabla f_k\|\leq  \kappa_g\Delta_k\} \quad\quad \text{satisfies}\quad\quad P(\A_k \mid {\red \mathcal{F}_{k-1}}  )\geq 1 - p_g.
\end{equation*}
\end{definition}

The zeroth-order probabilistic oracle is defined at both the current iterate $\bx_k$ and the trial iterate $\bx_{s_k}= \bx_k+\Delta\bx_k$ as follows.

\begin{definition}[Probabilistic zeroth-order oracle]\label{def:Probabilistic zeroth-order oracle}
Given $\bx_k$ and $\bx_{s_k}$, the oracle computes $\barf_k \coloneqq \barf(\bx_k,\xi_k^f)$ and $\barf_{s_k} \coloneqq \barf(\bx_{s_k},\xi_k^f)$, the estimates of the objective values $f_k \coloneqq f(\bx_k)$ and $f_{s_k}\coloneqq f(\bx_{s_k})$, such that the following two conditions hold.

\vskip5pt
\noindent$\bullet$\textbf{(i)} The absolute errors $e_k \coloneqq |\barf_k- f_k|$ and $e_{s_k} \coloneqq |\barf_{s_k}- f_{s_k}|$ are sufficiently small with a fixed probability with {\red $\kappa_{f,k} = \frac{\kappa_{fcd}}{16\|R_k\|} \min\big\{1,\frac{\epsilon_s}{\Delta_{\max}}\big\}\eta^3$}:
\begin{equation*}
\B_k=\left\{\max\left(e_k,e_{s_k}\right)\leq{\red \kappa_{f,k}}\Delta_k^{2}\right\}\quad \text{ satisfies } \quad P(\B_k\mid {\red \mathcal{F}_{k-1/2}} )\geq 1- p_f,
\end{equation*}

\noindent$\bullet$\textbf{(ii)}
The estimates $\barf_k,\barf_{s_k}$ are sufficiently \textit{reliable} so that they have controlled variance:
\begin{equation*}
\max\left\{\mE\left[e_k^2\mid {\red \mathcal{F}_{k-1/2}} \right],\mE\left[e_{s_k}^2\mid{\red \mathcal{F}_{k-1/2}} \right]\right\}\leq\barepsilon_k^2.
\end{equation*}
\end{definition}

If no trial step is computed, no objective-value estimates are drawn. In this case, we set $e_k=e_{s_k}=0$ and take $\mathcal B_k$ to be the sure event. We adopt the same zeroth- and first-order oracle design as in \cite{Fang2024Trust} (with $\alpha=0$ in their notation), which was developed for trust-region SSQP methods for solving equality-constrained problems. This unified oracle design allows equality- and inequality-constrained problems to be treated consistently, without requiring any modifications in sampling mechanisms to accommodate inequality constraints.

Estimates satisfying the above probabilistic oracle conditions can be constructed via a variety of techniques, including sample averaging, median-of-means estimators, or finite-difference approximations in gradient-free settings. We refer the interested reader to \cite[Section 3.2]{Fang2025High} and the references therein for a detailed discussion of estimator construction and the associated sample size requirements.

{\red In interior-point methods, the barrier parameter $\theta_k$ must be carefully managed to align with the algorithm's progress. If $\theta_k$ decays too rapidly, the barrier's effect vanishes prematurely; if it decays too slowly, the barrier problem fails to adequately approximate the original problem. Because inherent noise in stochastic optimization prevents us from measuring this progress exactly, we instead rely on a gradient estimate that is accurate with high probability. To formalize this, we introduce a validation oracle, which is queried immediately after each iterate update and takes two input sequences, $\{\epsilon_j\}$ and $\{\delta_k\}$, satisfying $\epsilon_j \to 0$ and $\sum \delta_k < \infty$.}

{\red \begin{definition}[Validation oracle]\label{def:val}
Given $\bx_{k+1}$, the oracle computes $\barg_{k+1}^{\val}\coloneqq \barg(\bx_{k+1},\xi_{k}^{\val})$, an estimate of the objective gradient $\nabla f_{k+1}$, such that 
\begin{equation*}
\mathcal{C}_k = \{ \|\barg_{k+1}^{\val} - \nabla f_{k+1}\| \leq \epsilon_j \}\quad\text{satisfies}\quad
P\left(\mathcal{C}_k \mid \mathcal{F}_{k-1/4} \right) \geq 1-\delta_k,
\end{equation*}
where $j$ is the index of the current barrier parameter, i.e., $\theta_k=\beta_j$.
\end{definition}
}

\begin{remark}
Our oracle design permits biased estimates and imposes no parametric assumptions on the noise; in particular, the gradient noise may possess infinite variance. These conditions are substantially weaker than those adopted in fixed-sampling frameworks \citep{Berahas2021Sequential, Fang2024Fully, curtis2026single, Curtis2023Stochastic}, where gradient estimates are required to be unbiased and their variance either uniformly bounded or subject to a growth condition. Recent non-asymptotic analyses \citep{Jin2023Sample, Cao2024First, Berahas2025Sequential} also employ probabilistic oracle designs but assume that the noise in objective-value estimation is sub-exponential at each iteration, implying the existence of infinite orders of moments and therefore restricting attention to light-tailed noise distributions. Recent literature \citet{Fang2025High, Scheinberg2025Stochastic} has aimed to relax the sub-exponential noise to heavy-tailed noise for equality-constrained or unconstrained problems.
\end{remark}

\subsection{Algorithm design}\label{sec:Algorithm design}

Our method requires input parameters: $p_g, p_f, \eta, \zeta, \epsilon_s\in(0,1)$, $\kappa_{fcd}\in(0,1]$, $\Delta_{\max}$, $\rho,\gamma>1$, $\kappa_g>0$, along with positive sequences {\red $\{\beta_j,\epsilon_j\}\downarrow 0$, where $\{\beta_j\}$ is strictly decreasing, and $\{\delta_k\}$ such that $\sum_k\delta_k<\infty$}.

The method is initialized at $\bx_0,\bs_0$, $\Delta_0\in (0,\Delta_{\max}]$, $\barepsilon_0$, $\barmu_0>0$. We allow $\bx_0,\bs_0$ to be infeasible ($c(\bx_0)\neq \b0, h(\bx_0)+\bs_0\neq \b0$) but require $\bs_0>0$. 

\begin{algorithm}[t]
\caption{A Trust-Region Interior-Point SSQP (TR-IP-SSQP) Method}\label{Alg:STORM}
\begin{algorithmic}[1]
\STATE \textbf{Input:} Initialization $(\bx_0,\bs_0)$, $\Delta_0\in(0,\Delta_{\max}]$, $\barepsilon_0, \barmu_0>0$; parameters $p_g$, $p_f$, $\eta$, $\zeta$, $\epsilon_s\in(0,1)$, $\kappa_{fcd}\in(0,1]$, $\Delta_{\max}$, $\rho,\gamma>1$, $\kappa_g>0$, positive sequences {\red $\{\beta_j,\epsilon_j\}\downarrow 0$ with $\{\beta_j\}$ strictly decreasing, and $\{\delta_k\}$ with $\sum_k\delta_k<\infty$. Set $j = 1$, $ \theta_0 = \beta_1$}.
\FOR {$k=0,1,\cdots,$}
\STATE Obtain $\barg_k$ and compute $\bar{Q}_k$, then generate $\barH_k$ and construct $\bar{W}_k$.  
\IF {\eqref{def:acc1} does not hold} \label{line:acc_check}
\STATE Set $\bx_{k+1}=\bx_k,\bs_{k+1}=\bs_k,\Delta_{k+1}=\Delta_k/\gamma,\barepsilon_{k+1}=\barepsilon_k/\gamma$. \hfill $\triangleright$ \textbf{{\footnotesize Unsuccessful Step}}
\ELSE
\STATE Compute $\bd_k$ as in Section \ref{sec: Computing the (rescaled) trial step} and compute $\text{Pred}_k$ following \eqref{def:Pred_k}.
\WHILE{\eqref{eq:threshold_Predk} does not hold}
\STATE Update $\barmu_k=\rho\barmu_k$.
\ENDWHILE
\STATE Set $\bx_{s_k} \hskip-1pt = \hskip-1pt \bx_k\hskip-1pt+\hskip-1pt\Delta\bx_k$, $\bs_{s_k}\hskip-1pt=\hskip-1pt\bs_k\hskip-1pt+\hskip-1pt\Delta\bs_k$; obtain $\barf_k, \barf_{s_k}$; compute $\text{Ared}_k$ by \eqref{def:Ared_k}.
\IF {$\text{Ared}_k/\text{Pred}_k\geq\eta$} 
\STATE Set $\bx_{k+1}\hskip-1pt=\hskip-1pt\bx_{s_k},\bs_{k+1} \hskip-1pt=\hskip-1pt \bs_{s_k}$, $\Delta_{k+1}\hskip-1pt = \hskip-1pt\min\{\gamma\Delta_k,\Delta_{\max}\}$. \hfill $\triangleright$ \textbf{{\footnotesize Successful Step}}
\IF{$-\text{Pred}_k\geq\barepsilon_k$}
\STATE Set $\barepsilon_{k+1}=\gamma\barepsilon_{k}$. \hfill $\triangleright$ \textbf{{\footnotesize Reliable Step}}
\ELSE 
\STATE Set $\barepsilon_{k+1}=\barepsilon_{k}/\gamma$. \hfill $\triangleright$ \textbf{{\footnotesize Unreliable Step}}
\ENDIF
\ELSE 
\STATE Set $\bx_{k+1}=\bx_{k},\bs_{k+1}=\bs_{k},\Delta_{k+1}=\Delta_k/\gamma,\barepsilon_{k+1}=\barepsilon_k/\gamma$. \hfill $\triangleright$~\textbf{{\footnotesize Unsuccessful~Step}}
\ENDIF
\ENDIF
\STATE {\red Obtain $\barg_{k+1}^{\val}$ and compute $\bar{Q}_{k+1}^{\val}$.}
\IF{{\red $\|\bar{Q}_{k+1}^{\val}\| \leq 2 \epsilon_j$}} 
\STATE {\red Set $j \leftarrow j + 1$ and $\theta_{k+1} = \beta_j$.}  \hfill $\triangleright$ \textbf{{\footnotesize Reduce Barrier Parameter}}
\ELSE 
\STATE {\red Set $\theta_{k+1} = \theta_k$.} \hfill $\triangleright$ \textbf{{\footnotesize Maintain Barrier Parameter}}
\ENDIF
\STATE \label{line:mu_carry} Set $\barmu_{k+1}=\barmu_k$.  
\ENDFOR
\end{algorithmic}
\end{algorithm}

Given $(\bx_k,\bs_k,\theta_k,\Delta_k,\barepsilon_k,\barmu_k)$ at the $k$-th iteration, we first obtain a stochastic gradient estimate $\barg_k$ from the probabilistic first-order oracle (cf. Definition \ref{def:Probabilistic first-order oracle}). We then compute the estimated stationarity measure
\begin{equation*}
\bar{Q}_k = (P_k\barpsi_k ; \vartheta_k),
\end{equation*}
where, as defined in Section \ref{sec: Computing the (rescaled) trial step},
\begin{equation*}
{\red \barpsi_k =(\barg_k;-\theta_k S_k^{-1}\boldsymbol{e})}, \;\;
\vartheta_k = (c_k;h_k+\bs_k),\;\;
P_k = I - A_k^T[A_kA_k^T]^{-1}A_k, \; {\red \setlength{\arraycolsep}{2pt}A_k = \begin{pmatrix}
G_k & 0 \\
J_k & I
\end{pmatrix}}.
\end{equation*}
We also construct $\barH_k \in \F_{k-1/2}$ to estimate $\nabla^2_{\bx}\L_k$ and set {\red $\bar{W}_k=\text{diag}(\barH_k,\theta_kS_k^{-2})$}.

Next, we check the condition
\begin{equation}\label{def:acc1}
\frac{\|\bar{Q}_k\|}{\max\{1,\|\bar{W}_k\|\}}\geq \eta\Delta_k.
\end{equation}
If \eqref{def:acc1} fails, we say that the $k$-th iteration is \textit{unsuccessful} and set $\bx_{k+1}=\bx_k, \bs_{k+1}=\bs_k,
\Delta_{k+1}=\Delta_k/\gamma, \barepsilon_{k+1}=\barepsilon_k/\gamma$. We then proceed to the barrier parameter update step below, and then to the $(k+1)$-th iteration. Otherwise, we compute the trial step $\bd_k=(\Delta\bx_k;\Delta\bs_k)$ as described in Section \ref{sec: Computing the (rescaled) trial step}.

With the trial step $\bd_k$, we further compute the predicted reduction
\begin{equation}\label{def:Pred_k}
\text{Pred}_k=\barpsi_k^T\bd_k + \frac{1}{2}\bd_k^T\bar{W}_k\bd_k+\barmu_k(\|\vartheta_k+A_k \bd_k\|-\|\vartheta_k\|)
\end{equation}
and update the merit parameter as $\barmu_k\leftarrow\rho\barmu_k$ until
{\red \begin{equation}
\label{eq:threshold_Predk}
\text{Pred}_k\leq -\frac{\kappa_{fcd}}{2} \|\bar{Q}_k\| \min\left\{\frac{\Delta_k}{\|R_k\|},\frac{\epsilon_s}{\|R_k\|},\frac{\|\bar{Q}_k\|}{\|\bar{W}_k\|}\right\}.
\end{equation}}
Then, we set $(\bx_{s_k},\bs_{s_k})\coloneqq (\bx_k+\Delta\bx_k,\bs_k+\Delta\bs_k)$, and obtain $\barf_k$ and $\barf_{s_k}$ from the probabilistic zeroth-order oracle (cf. Definition \ref{def:Probabilistic zeroth-order oracle}).

We adopt an $\ell_2$ merit function to balance the reduction in the objective function and the constraint violation, where the log-barrier term is also included:
\begin{equation}\label{merit_func}
\L_{\barmu, \theta}(\bx,\bs) = f(\bx) - \theta \sum_{i=1}^n\ln(\bs^{(i)}) + \barmu \|(c(\bx); h(\bx)+\bs)\|.
\end{equation}
At the $k$-th iteration, we compute its estimate at $(\bx_k,\bs_k)$ and $(\bx_{s_k},\bs_{s_k})$, denoted as $\bar{\L}_{\barmu_k,\theta_k}^k$ and $\bar{\L}_{\barmu_k,\theta_k}^{s_k}$, and define the actual reduction as
\begin{multline}
\label{def:Ared_k}
\text{Ared}_k =\bar{\L}_{\barmu_k,\theta_k}^{s_k}-\bar{\L}_{\barmu_k,\theta_k}^k 
=\barf_{s_k}-\barf_k-\theta_k\sum_{i=1}^n \left\{\ln(\bs_{s_k}^{(i)})-\ln(\bs_k^{(i)})\right\} \\
+\barmu_k(\|(c_{s_k} ; h_{s_k}+\bs_{s_k})\|-\|(c_k ; h_k+\bs_k)\|).
\end{multline}
To determine whether the iterate is updated, we check
\begin{equation}\label{def:accuracy}
\text{(a)}\; \text{Ared}_k/\text{Pred}_k\geq\eta \qquad\quad\text{and}\quad\qquad \text{(b)}\; -\text{Pred}_k\geq\barepsilon_k.
\end{equation}

\noindent $\bullet$ \textbf{Case 1: (\ref{def:accuracy}a) holds}. We say that the $k$-th iteration is \textit{successful} and update the iterate and the trust-region radius as
\begin{equation*}
\bx_{k+1}=\bx_{s_k},\quad\quad \bs_{k+1}=\bs_{s_k} ,\quad\quad \Delta_{k+1}=\min\{\gamma\Delta_k,\Delta_{\max}\}.
\end{equation*}
\noindent If (\ref{def:accuracy}b) also holds, 
we say that the $k$-th iteration is \textit{reliable} and increase the reliability parameter by $\barepsilon_{k+1}=\gamma\barepsilon_k$. If (\ref{def:accuracy}b) does not hold, we say that the $k$-th iteration is \textit{unreliable} and decrease the reliability parameter by $\barepsilon_{k+1}=\barepsilon_k/\gamma$.

\noindent $\bullet$ \textbf{Case 2: (\ref{def:accuracy}a) does not hold}. We say that the $k$-th iteration is \textit{unsuccessful}. We let $\bx_{k+1}=\bx_k,\bs_{k+1}=\bs_k, \Delta_{k+1}=\Delta_k/\gamma,\barepsilon_{k+1}=\barepsilon_k/\gamma$.

\vspace{0.2cm}

{\red \noindent \textbf{Barrier parameter update.} To determine whether to update the barrier parameter, we first obtain $\barg_{k+1}^{\val}$ following the validation oracle (Definition \ref{def:val}), and evaluate $A_{k+1},\vartheta_{k+1}$. Then, we construct $\barpsi_{k+1}^{\val} = (\barg_{k+1}^{\val}; -\theta_k S_{k+1}^{-1} \boldsymbol{e})$ and 
\begin{equation*}
\bar{Q}_{k+1}^{\val} = (P_{k+1}\barpsi_{k+1}^{\val}; \vartheta_{k+1}) . 
\end{equation*}
Note that we use the barrier parameter $\theta_k$ while the gradient and constraint values are evaluated at $(\bx_{k+1},\bs_{k+1})$. Then we check the condition
\begin{equation}\label{val_condition}
\|\bar{Q}_{k+1}^{\val}\| \leq 2\epsilon_j.
\end{equation}
If \eqref{val_condition} holds, we reduce the barrier parameter by setting $j \leftarrow j + 1$ and $\theta_{k+1} = \beta_j$. Otherwise we maintain the current barrier parameter by setting $\theta_{k+1} = \theta_k = \beta_j$.
}

\vspace{0.1cm}

The criterion (\ref{def:accuracy}a) ensures that the actual reduction achieves a fixed fraction of the predicted reduction, while the criterion (\ref{def:accuracy}b) controls the level of stochasticity in the estimates. Based on the values of $\text{Pred}_k$ and $\bar{\epsilon}_k$, we classify a successful iteration as either \emph{reliable} or \emph{unreliable}, and update the reliability parameter $\bar{\epsilon}_k$ accordingly. This mechanism provides additional flexibility in the selection of sample sizes and may help reduce the overall sampling cost.

\begin{remark}
Since the trust-region subproblem is formulated and solved in terms of the whole trial step $\bd_k=(\Delta\bx_k;\Delta\bs_k)$, we compute the predicted reduction directly using $\bd_k$ (cf. \eqref{def:Pred_k}). In the convergence analysis, however, it is more convenient to work with the components $\Delta\bx_k$ and $\Delta\bs_k$ separately. Accordingly, we rewrite the predicted reduction as follows, yielding an expression equivalent to \eqref{def:Pred_k},
\begin{align}\label{def:Pred_k2}
\text{Pred}_k  &=    \barg_k^T\Delta\bx_k + \frac{1}{2}\Delta\bx_k^T\barH_k\Delta\bx_k -\theta_k\boldsymbol{e}^TS_k^{-1}\Delta\bs_k + \frac{1}{2} \theta_k \Delta\bs_k^TS_k^{-2}\Delta\bs_k \notag\\
& \quad +\barmu_k\left\{\|(c_k+G_k\Delta\bx_k; h_k +\bs_k + J_k\Delta\bx_k+\Delta\bs_k)\| - \|(c_k; h_k+\bs_k)\| \right\}.
\end{align}
\end{remark}

\begin{remark}\label{rem:Qk}
Similar to interior-point methods in deterministic optimization \citep{Byrd1999Interior,Byrd2000Trust,Curtis2010Interior,Curtis2025Interior}, our algorithm design involves an estimated stationarity measure $\bar{Q}_k$.
Our convergence analysis focuses on both $\bar{Q}_k$ and its deterministic counterpart,
\begin{equation}\label{def:true_KKT_measure}
Q_k \coloneqq (P_k \psi_k; \vartheta_k)  \quad\text{where }\;\;{\red \psi_k = (\nabla f_k; -\theta_k S_k^{-1} \boldsymbol{e})}.
\end{equation}
In fact, $\|Q_k\|$ is closely related to the KKT residual $\|\nabla \L_{\theta_k}\|$ of the barrier problem \eqref{BarrierProblem}. {\red With the Lagrange multipliers defined as in \eqref{def:true_Lag_para}, we have
$\|Q_k\|\leq\|\nabla\L_{\theta_k}\| \leq \Upsilon_L \| Q_k\|$ for all $k\geq 0$, where $\Upsilon_L=\sqrt{1+\kappa_{2,A}}$ with $\kappa_{2,A}$ given in Assumption \ref{assump1} (see the proof of Theorem \ref{thm:global_conv}). Therefore, to establish global convergence, we show that $\|Q_k\|$, evaluated at the iterations where the barrier parameter is reduced, vanishes almost surely.}
\end{remark}

\section{Convergence Analysis} \label{sec:4}

We study the convergence of the proposed algorithm in this section. We first state the roadmap of the analysis.

{\red \textbf{Roadmap:} The analysis is organized around the stages of the barrier parameter: stage $j$ consists of the iterations with $\theta_k=\beta_j$, and $N_j$ denotes its first iteration. Section \ref{sec:4.1} studies a single stage. Lemma \ref{lemma:bound} shows that within stage $j$ the slack variables are bounded below by $\beta_j/\tau^{\max}$, which gives stage-dependent bounds on $\|R_k\|$ and $\|\bar W_k\|$. Lemmas \ref{lemma:line6_not_hold}--\ref{lemma:guarantee_succ_step_KKT} show that if the gradient and function estimates are accurate and $\Delta_k$ is small compared with $\|Q_k\|$, the iteration is successful, and Lemma \ref{lemma:reduction_of_successful_iter} shows that a successful iteration with accurate function estimates decreases the true merit function by $\mathcal{O}(\Delta_k^2)$. Lemmas \ref{lemma:upper_bound_Phi}--\ref{lemma:One_Step_Rec} combine these facts into an expected decrease of a Lyapunov function, from which Lemma \ref{lemma:vanish_Delta_k} obtains $\sum_{k=N_j}^{N_{j+1}-1}\Delta_k^2<\infty$ almost surely. Lemma \ref{lemma:liminf_Qk_0} then shows that if a stage never ends, $\liminf_k\|Q_k\|=0$. In Section \ref{sec:4.2}, we combine this with the validation oracle to show that every stage ends after finitely many iterations (Lemma \ref{lemma:finite_termination}), and Theorem \ref{thm:global_conv} shows that the iterates at which the barrier parameter is decreased are asymptotically stationary for Problem \eqref{Intro_StoProb}, with KKT limit points. }

Now we state assumptions that are needed in our analysis.

\begin{assumption}\label{assump1}
The iterates and trial points, $\{\bx_k,\bs_k\}$ and $\{\bx_{s_k}, \bs_{s_k}\}$, are contained within an open convex set $\Omega$. The objective function $f(\bx)$ is continuously differentiable and bounded below by $f_{\inf}$, and its gradient $\nabla f(\bx)$ is Lipschitz continuous on $\Omega$ with constant $L_{\nabla f}$. The constraints $c(\bx),h(\bx)$ are continuously differentiable, and their respective Jacobians $G(\bx),J(\bx)$ are Lipschitz continuous on $\Omega$ with constants $L_G,L_J$, respectively. 
There exist constants $\kappa_c,\kappa_h$, $\kappa_{\nabla f}$ such that $\|c_k\| \leq \kappa_{c}$, $\|h_k\| \leq \kappa_{h}$, and $\|\nabla f_k\| \leq \kappa_{\nabla f}$ for all $k \geq 0$. There exist $\kappa_{1,A},\kappa_{2,A}>0$ such that $\kappa_{1,A}\cdot I \preceq A_kA_k^T\preceq \kappa_{2,A}\cdot I$ for all $k\geq0$. {\red There exists a constant $\tau^{\max}>0$ such that $\theta_k/\bs_k^{\min} \leq \tau^{\max}$ for all $k\geq 0$, where $\bs_k^{\min}\coloneqq\min_{1\leq i\leq n}\bs_k^{(i)}$.} Finally, there exist constants $\kappa_B, \kappa_S > 1$ such that $\|\barH_k\| \leq \kappa_B $ and $\|S_k\|\leq \kappa_S$.
\end{assumption}

Assumption \ref{assump1} is standard in both deterministic and stochastic SQP literature \citep{Byrd1987Trust, Powell1990trust, ElAlem1991Global, Conn2000Trust, Byrd1999Interior,Byrd2000Trust, Berahas2021Stochastic, Curtis2023Stochastic, Fang2024Fully,Fang2024Trust}. In particular, it implies that  $\|A_k^T[A_kA_k^T]^{-1}\|\leq 1/\sqrt{\kappa_{1,A}}$, $G_k$ has full row rank,
\begin{equation*} 
\max\{\|G_k\|,\|J_k\| \}\leq\sqrt{\kappa_{2,A}}, \qquad \|G_k^T[G_kG_k^T]^{-1}\| \leq \frac{1}{\sqrt{\kappa_{1,A}}},
\end{equation*}
and $\|\bs_k\|\leq \sqrt{n}\cdot \kappa_S \eqqcolon \Upsilon_s$. {\red Since $\btau_k = \theta_k S_k^{-1}\boldsymbol{e}$ (cf. \eqref{def:true_Lag_para}), imposing an upper bound on $\theta_k/\bs_k^{\min} $ is equivalent to assuming a uniform upper bound $\tau^{\max}$ on the Lagrange multipliers associated with the inequality constraints. Moreover, the uniform boundedness of $\btau_k$, together with the definition of $\blambda_k$ in \eqref{def:true_Lag_para}, directly implies that $\blambda_k$ is also uniformly bounded. This assumption of uniformly bounded Lagrange multipliers is standard in the analysis of constrained optimization problems \citep{Fang2024Fully,Na2022adaptive,Berahas2021Sequential}.}

{\red Throughout the convergence
analysis, we use $\barmu_k$ to denote its value after the
merit-parameter update at iteration $k$ such that \eqref{eq:threshold_Predk} is satisfied, or its
unchanged value if no trial step is computed. To establish global convergence, we also need the merit parameter to be bounded in every stage of the barrier parameter (i.e., for every fixed value of $\beta_j$); thus, we impose the assumption below.

\begin{assumption}\label{assump4}
For every value of $\beta_j$, there exists a deterministic constant $\hat{\mu}_j$ such that for all $k$ in the $j$-th stage, we have $\barmu_k\leq \hatmu_j$.
\end{assumption}
}
Assuming the boundedness of the merit parameter is standard in the literature for both line-search-based methods \citep{Berahas2021Sequential, Berahas2021Stochastic, Berahas2022Accelerating, Curtis2023Stochastic} and trust-region-based methods \citep{Fang2024Fully, Fang2024Trust}. However, unlike line-search-based methods, our approach does not require the merit parameter to be sufficiently large, which typically necessitates either stronger assumptions on the noise distribution (e.g., the symmetry assumption in \cite{ Berahas2021Stochastic,Berahas2021Sequential}) or a more complicated algorithm design (e.g., the {\red additional condition} in \cite[(19)]{Na2022adaptive}). {\red We emphasize that Assumption \ref{assump4} is not implied by Definitions \ref{def:Probabilistic first-order oracle}--\ref{def:val}. We revisit Assumption \ref{assump4} in Section \ref{subsec:merit_para}, where we show that it holds when the gradient estimation error is bounded. Such boundedness can be enforced by projecting the gradient estimate onto a Euclidean ball as in \cite{Lu2026Variance}, i.e., $\bar{\boldsymbol g}_k\leftarrow\mathrm{Proj}_{\mathcal B(\kappa_{\nabla f})}(\bar{\boldsymbol g}_k),\; \mathrm{Proj}_{\mathcal B(r)}(\boldsymbol y)=\min\{1,\,r/\|\boldsymbol y\|\}\,\boldsymbol y$, where $\mathcal B(r)$ is the closed Euclidean ball of radius $r$ centered at the origin, and $\kappa_{\nabla f}$ is the bound on $\|\nabla f_k\|$ in Assumption \ref{assump1}. Since $\nabla f_k\in\mathcal B(\kappa_{\nabla f})$ and the projection is nonexpansive, Definition \ref{def:Probabilistic first-order oracle} is preserved. Moreover, $\|\bar{\boldsymbol g}_k-\nabla f_k\|\le2\kappa_{\nabla f}$, so Assumption \ref{ass:bdd_err} holds with $M=2\kappa_{\nabla f}$, even under heavy-tailed noise.}

\subsection{Fundamental lemmas within a barrier parameter stage}\label{sec:4.1}

{\red In this subsection, we fix a value $\beta_j$ and analyze iterations $k$ within the stage $j$, i.e., $\theta_k = \beta_j$. We denote by $N_j$ the first iteration of stage $j$, i.e., $N_j = \min\{k:\theta_{k}=\beta_j\}$, with $N_j=\infty$ if stage $j$ is never entered. In particular, $N_1=0$.

We first prove a lemma to demonstrate that within stage $j$, $\bs_k^{\min}$ is lower bounded, while $\|R_k\|$ and $\|\bar{W}_k\|$ are upper bounded, with all bounds depending on $\beta_j$.

\begin{lemma}\label{lemma:bound}
Under Assumption \ref{assump1}, for every iteration $k$ in stage $j$, we have
\begin{equation*}
\bs_k^{\min} \geq \frac{\beta_j}{\tau^{\max}}, \; \|R_k\| \leq \kappa_{R,\beta_j} \coloneqq \max\left\{1, \frac{\tau^{\max}}{\beta_j}\right\}, \; \|\bar{W}_k\| \leq \kappa_{W,\beta_j}\coloneqq \kappa_B + \frac{(\tau^{\max})^2}{\beta_j}.
\end{equation*}
\end{lemma}

\begin{proof}
(1) By Assumption \ref{assump1}, we have $\theta_k/\bs_k^{\min} \leq \tau^{\max}$. Since $\theta_k = \beta_j$, we obtain $\bs_k^{\min} \geq \beta_j/\tau^{\max}$.
(2) By the definition of $R_k$, we have $\|R_k\| = \max\{1, 1/\bs_k^{\min}\}$. Combining this relation with the lower bound on $\bs_k^{\min}$ leads to the result.
(3) By the definition of $\bar{W}_k$, we have $\|\bar{W}_k\| \leq \| \barH_k \| + \theta_k \| S_k^{-2}\| = \| \barH_k \| + \theta_k/(\bs_k^{\min})^2 \leq \kappa_B + \theta_k/\bs_k^{\min} \cdot 1/\bs_k^{\min} \leq \kappa_B + (\tau^{\max})^2/\beta_j$. We thus complete the proof.  
\end{proof}
}

In the following lemma, we demonstrate that when the estimates of gradients are accurate and the trust-region radius is sufficiently small compared to $\|Q_k\|$, then Line~\ref{line:acc_check} of the algorithm will not be triggered.

\begin{lemma}\label{lemma:line6_not_hold}
Under Assumption \ref{assump1}, on the event $\A_k$ and for every iteration $k$ in stage $j$, if
\begin{equation}\label{eq:lemma3:delta}
\|Q_k\|\geq (\kappa_g + \eta\kappa_{W,\beta_j})\Delta_k   
\end{equation}
is satisfied, then Line \ref{line:acc_check} of Algorithm \ref{Alg:STORM} will not be triggered.
\end{lemma}

\begin{proof}
By the relationship $ \|Q_k\|-\|\bar{Q}_k\| \leq  \|Q_k-\bar{Q}_k\| = \|P_k(\psi_k-\bar{\psi}_k)\| \leq \|\psi_k-\bar{\psi}_k\| = \|\nabla f_k-\barg_k\|$, we have $\|Q_k\|-\|\bar{Q}_k\| \leq \|\nabla f_k-\barg_k\|\leq \kappa_g \Delta_k$ on the event $\A_k$. In addition, by Lemma \ref{lemma:bound}, we have $\|\bar{W}_k\| \leq \kappa_{W,\beta_j}$. As $\kappa_B \geq 1$, we also have $\kappa_{W,\beta_j} \geq 1$.
Combining these results, we find that \eqref{eq:lemma3:delta} implies $\|\bar{Q}_k\|/\max\{1,\|\bar{W}_k\|\}\geq \eta\Delta_k$, thus Line \ref{line:acc_check} will not be triggered.
\end{proof}

In the next lemma, we investigate the difference between the predicted reduction $\text{Pred}_k$ in \eqref{def:Pred_k} and the true actual reduction  $\L_{\barmu_{k},\theta_k}^{s_k}-\L_{\barmu_{k},\theta_k}^{k}$ computed using the $\ell_2$ merit function \eqref{merit_func}. We show that their difference can be upper bounded by $\cO(\Delta_k^2)$.

\begin{lemma}\label{lemma:diff_ared_pred_wo_corr_step}
Under Assumptions \ref{assump1}, \ref{assump4}, on the event $\A_k$ and for every iteration $k$ in stage $j$, we have
\begin{equation}\label{eq:ared-pred}
\big|\L_{\barmu_{k},\theta_k}^{s_k}-\L_{\barmu_{k},\theta_k}^{k}-\text{Pred}_k\big| \leq \Upsilon_{1,j}\Delta_k^2,
\end{equation}
where $\Upsilon_{1,j}=\kappa_g +\max\{\frac{1}{2}(L_{\nabla f}+\kappa_B+\hatmu_j (L_G+L_J)),\beta_1(\frac{1}{1-\epsilon_s}+\frac{1}{2}) \}$.
\end{lemma}

\begin{proof}
See Appendix \ref{proof_of_lemma:diff_ared_pred_wo_corr_step}.
\end{proof}

In the next lemma, we demonstrate that if the estimates of the objective models are accurate and the trust-region radius is sufficiently small compared to $\|Q_k\|$ and $\epsilon_s$, the $k$-th iteration is guaranteed to be successful.

\begin{lemma}\label{lemma:guarantee_succ_step_KKT}

Under Assumptions \ref{assump1}, \ref{assump4}, on the event $\A_k\cap\B_k$ and for every iteration $k$ in stage $j$, if
\begin{equation}\label{delta:lemma_guarantee_succ_step_KKT}
\Delta_k\leq \min\left\{\frac{4(1-\eta) \kappa_{fcd}\|Q_k\|}{8\kappa_{R,\beta_j}\Upsilon_{1,j} + \kappa_{fcd} \min\left\{1,\frac{\epsilon_s}{\Delta_{\max}}\right\}\eta^3 + 4(1-\eta)\kappa_{fcd} \kappa_g },\frac{ \|Q_k\|}{\kappa_g+\kappa_{W,\beta_j}},\epsilon_s\right\},
\end{equation}
then the $k$-th iteration is successful. 
\end{lemma}

\begin{proof}
To prove that the $k$-th iteration is successful, it suffices to show that Line \ref{line:acc_check} in Algorithm \ref{Alg:STORM} is not triggered and $\text{Ared}_k/\text{Pred}_k\geq\eta$, i.e., (\ref{def:accuracy}a), is satisfied. Since $\eta<1$, we find \eqref{delta:lemma_guarantee_succ_step_KKT} implies \eqref{eq:lemma3:delta}, thus the conclusion of Lemma \ref{lemma:line6_not_hold} indicates that Line~\ref{line:acc_check} is not triggered. Next, we show $\text{Ared}_k/\text{Pred}_k\geq\eta$ holds. By the definition of $\text{Ared}_k$ in \eqref{def:Ared_k} in stage $j$, we have
\begin{align*}
\left|\frac{\text{Ared}_k}{\text{Pred}_k}-1\right| & = \left|\frac{\bar{\L}_{\barmu_{k} , \theta_k }^{s_k}-\bar{\L}_{\barmu_{k},\theta_k }^k}{\text{Pred}_k}-1 \right| \\
& \leq \frac{|\bar{\L}_{\barmu_{k},\theta_k }^{s_k}-\L_{\barmu_{k},\theta_k }^{s_k}|+|\L_{\barmu_{k},\theta_k }^{s_k}-\L_{\barmu_{k},\theta_k }^k-\text{Pred}_k|+|\L_{\barmu_{k},\theta_k }^k-\bar{\L}_{\barmu_{k},\theta_k }^{k}|}{|\text{Pred}_k|}.
\end{align*} 
Note that \eqref{delta:lemma_guarantee_succ_step_KKT} implies $\|Q_k\|\geq (\kappa_g + \kappa_{W,\beta_j} )\Delta_k$. On the event $\A_k$, we also have $\|Q_k\| - \|\bar{Q}_k\| \leq \kappa_g \Delta_k$. Thus, $\|\bar{Q}_k\|\geq \kappa_{W,\beta_j} \Delta_k$.
Since $\max\{1,\|\bar{W}_k\|\}\leq\kappa_{W,\beta_j}$, we have $\Delta_k \leq \|\bar{Q}_k\|/\|\bar{W}_k\|$. Noting that $\|R_k\| \geq 1$, we further have $\Delta_k/ \|R_k\| \leq \|\bar{Q}_k\|/\|\bar{W}_k\|$, which implies
\begin{equation*}
\text{Pred}_k
\stackrel{\eqref{eq:threshold_Predk}}{\leq}
-\frac{\kappa_{fcd}}{2} \cdot \|\bar{Q}_k\|\min\left\{\frac{\Delta_k}{\|R_k\|},\frac{\epsilon_s}{\|R_k\|},\frac{\|\bar{Q}_k\|}{\|\bar{W}_k\|}\right\}
\stackrel{\eqref{delta:lemma_guarantee_succ_step_KKT}}{=}
-\frac{\kappa_{fcd}\|\bar{Q}_k\|}{2\|R_k\|}\Delta_k.
\end{equation*}
With $\|R_k\|\leq \kappa_{R,\beta_j}$ (cf. Lemma \ref{lemma:bound}), we have
\begin{equation}\label{eq:lemma6_pred}
\text{Pred}_k \leq -\frac{\kappa_{fcd}\|\bar{Q}_k\|}{2\|R_k\|}\Delta_k \leq 
-\frac{\kappa_{fcd}}{2\kappa_{R,\beta_j}}\|\bar{Q}_k\|\Delta_k.
\end{equation}
Since $|\L_{\barmu_k,\theta_k}^{s_k}-\barL_{\barmu_k,\theta_k}^{s_k}|=|f_{s_k}-\barf_{s_k}|$ and $|\L_{\barmu_k,\theta_k}^k-\barL_{\barmu_k,\theta_k}^k|=|f_k-\barf_k|$, we have
\begin{align*}
\left|\frac{\text{Ared}_k}{\text{Pred}_k}-1\right| & \leq \frac{\left|f_{s_k}-\barf_{s_k}\right|+\left|\L_{\barmu_{k} , \theta_k}^{s_k}-\L_{\barmu_{k}, \theta_k}^k-\text{Pred}_k\right|+\left|f_k-\barf_{k}\right|}{\left|\text{Pred}_k\right|} \\
& \leq \frac{ \left|\L_{\barmu_{k}, \theta_k}^{s_k}-\L_{\barmu_{k}, \theta_k}^k-\text{Pred}_k\right|+ 2\kappa_{f,k} \Delta_k^2}{\left|\text{Pred}_k\right|} \qquad(\text{since }\B_k \text{ holds}) \\
& \stackrel{\mathclap{\eqref{eq:ared-pred}} }{\leq } \; \frac{ \Upsilon_{1,j} \Delta_k^2+ 2\kappa_{f,k} \Delta_k^2}{\left|\text{Pred}_k\right|}\\
&  \stackrel{\eqref{eq:lemma6_pred}}{\leq} \; \frac{ (8\kappa_{R,\beta_j}\Upsilon_{1,j} + \kappa_{fcd} \min\left\{1,\frac{\epsilon_s}{\Delta_{\max}}\right\}\eta^3)\Delta_k}{4 \kappa_{fcd} \|\bar{Q}_k\|} \stackrel{\eqref{delta:lemma_guarantee_succ_step_KKT}}{\leq} 1-\eta,
\end{align*} 
where the last inequality also uses $\|\bar{Q}_k\|\geq\|Q_k\|-\kappa_g\Delta_k$. This implies $\text{Ared}_k/\text{Pred}_k\geq \eta$; thus, the $k$-th iteration must be successful.
\end{proof}

The next lemma shows that if the iteration is successful and objective value estimates are accurate, then the reduction in the \textit{true} merit function is $\mathcal{O}(\Delta_k^2)$.

\begin{lemma}\label{lemma:reduction_of_successful_iter}
Under Assumptions \ref{assump1}, \ref{assump4}, on the event $\B_k$ and for every iteration $k$ in stage $j$, if the $k$-th iteration is successful, then
\begin{equation}\label{eq12}
\L_{\barmu_{k}, \theta_k }^{k+1}-\L_{\barmu_{k}, \theta_k }^k\leq -\frac{3\kappa_{fcd}}{8\kappa_{R,\beta_j}} \min\left\{1,\frac{\epsilon_s}{\Delta_{\max}}\right\}\eta^3 \Delta_k^2 .
\end{equation}
\end{lemma}

\begin{proof}
For a successful iteration, Line \ref{line:acc_check} of the algorithm cannot be triggered, thus we must have
$\|\bar{Q}_k\|/\max\{1,\|\bar{W}_k\|\}\geq \eta\Delta_k$. This relation implies
\begin{align}\label{eq:pred_1st_order}
\text{Pred}_k & \stackrel{\mathclap{\eqref{eq:threshold_Predk}}}{\leq} -\frac{\kappa_{fcd}}{2} \|\bar{Q}_k\|\min\left\{ \frac{\Delta_k}{\|R_k\|},\frac{\epsilon_s}{\|R_k\|},\frac{\|\bar{Q}_k\|}{\|\bar{W}_k\|}\right\} \stackrel{\mathclap{\eqref{def:acc1}}}{\leq} -\frac{\kappa_{fcd}}{2} \eta\Delta_k\min\left\{ \frac{\Delta_k}{\|R_k\|},\frac{\epsilon_s}{\|R_k\|},\eta\Delta_k \right\} \notag\\
& \leq -\frac{\kappa_{fcd} \eta\Delta_k}{2\|R_k\|} \min\left\{ \Delta_k,\frac{\epsilon_s}{\Delta_{\max}}\Delta_k,\eta\Delta_k \right\}  \leq -\frac{\kappa_{fcd}}{2\|R_k\|} \min\left\{1,\frac{\epsilon_s}{\Delta_{\max}}\right\} \eta^2\Delta_k^2,
\end{align}
since $\eta<1$ and $\Delta_k \leq \Delta_{\max}$. On the event $\B_k$, we have $\left| f_{s_k}-\barf_{s_k} \right| + \left| f_k-\barf_k \right| \leq 2\kappa_{f,k}\Delta_k^2$. Recalling that $\text{Ared}_k =\bar{\L}_{\barmu_{k}, \theta_k }^{k+1} - \bar{\L}_{\barmu_{k}, \theta_k }^k  $ and $\text{Ared}_k/\text{Pred}_k\geq\eta$ for a successful iteration, we have
\begin{align*}
\L_{\barmu_{k} , \theta_k }^{k+1} & -\L_{\barmu_{k}, \theta_k }^k
\leq \left| \L_{\barmu_{k}, \theta_k }^{k+1}-\bar{\L}_{\barmu_{k}, \theta_k }^{k+1} \right| + \text{Ared}_k + \left| \bar{\L}_{\barmu_{k}, \theta_k }^k-\L_{\barmu_{k}, \theta_k }^k \right| \\
& \leq \left| f_{s_k}-\barf_{s_k} \right| + \eta \cdot \text{Pred}_k + \left| f_k-\barf_k \right| \\
& \stackrel{\mathclap{\eqref{eq:pred_1st_order}}}{\leq} \; 2\kappa_{f,k}\Delta_k^2 -\frac{\kappa_{fcd}}{2\|R_k\|} \min\left\{1,\frac{\epsilon_s}{\Delta_{\max}}\right\}\eta^3 \Delta_k^2 \\
& \leq -\frac{3\kappa_{fcd}}{8\|R_k\|} \min\left\{1,\frac{\epsilon_s}{\Delta_{\max}}\right\}\eta^3 \Delta_k^2 \leq -\frac{3\kappa_{fcd}}{8\kappa_{R,\beta_j}} \min\left\{1,\frac{\epsilon_s}{\Delta_{\max}}\right\}\eta^3 \Delta_k^2.
\end{align*} 
This completes the proof.
\end{proof}

{\red To proceed with the analysis, we use a Lyapunov function, defined for iteration $k$ in stage $j$ as
\begin{equation*}
\Phi_{\barmu_k , \theta_k}^k = \frac{1}{\barmu_k}\left[f_k - f_{\inf} - \theta_k \sum_{i=1}^n\ln\left(\frac{\bs^{(i)}_k}{\Upsilon_s}\right) + \barmu_k \|(c_k; h_k+\bs_k)\|  +a_{\beta_j}\Delta_k^2 + b \barepsilon_k\right],
\end{equation*}
where $a_{\beta_j} = \frac{\kappa_{fcd}}{8\kappa_{R,\beta_j}(\gamma^2-\gamma^{-2})} \min\big\{1,\frac{\epsilon_s}{\Delta_{\max}}\big\}\eta^3$ and  $b=\frac{\eta}{4(\gamma-\gamma^{-1})}$. Since $\bs_k^{(i)}\leq\Upsilon_s$ by Assumption \ref{assump1}, we have $\Phi_{\barmu_k,\theta_k}^k \geq 0$.

By the algorithm design, for iteration $k$ in stage $j$, we have $\theta_{k+1}\in\{\beta_{j+1},\beta_j\}$. In either case, we have $\theta_{k+1}\leq \theta_k$, and $\Phi_{\barmu_{k+1},\theta_{k+1}}^{k+1} \leq \Phi_{\barmu_k,\theta_k}^{k+1}$, which is formalized in the following lemma.

\begin{lemma}\label{lemma:upper_bound_Phi}
Under Assumption \ref{assump1}, for every iteration $k$ in stage $j$, we have
\begin{multline*}
\Phi_{\barmu_{k+1},\theta_{k+1}}^{k+1} \leq \Phi_{\barmu_{k+1},\theta_{k}}^{k+1} \leq \Phi_{\barmu_k,\theta_k}^{k+1}  =   \frac{1}{\barmu_{k}} \Bigg[f_{k+1} - f_{\inf} - \theta_k \sum_{i=1}^n\ln\left(\frac{\bs^{(i)}_{k+1}}{\Upsilon_s}\right) \\
+ \barmu_k \|(c_{k+1}; h_{k+1} + \bs_{k+1} )\| + a_{\beta_j}\Delta_{k+1}^2 + b \barepsilon_{k+1}\Bigg] .
\end{multline*}
\end{lemma}

\begin{proof}
By Assumption \ref{assump1}, we have $0<\bs_{k+1}^{(i)}\le \Upsilon_s$ for all $i=1,\ldots,n$, and hence $\ln(\bs_{k+1}^{(i)}/\Upsilon_s)\le 0$, so the barrier term $-\theta\sum_{i=1}^n\ln(\bs_{k+1}^{(i)}/\Upsilon_s)$ is nondecreasing in $\theta$. Moreover, since $ \frac{1}{\kappa_{R,\beta_j}} =\min\left\{1,\frac{\beta_j}{\tau^{\max}}\right\}$, the definition of $a_{\beta_j}$ shows that it is nondecreasing in $\beta_j$. Therefore, using $\theta_{k+1}\le\theta_k$ and $\beta_{j+1}\le\beta_j$, we obtain $\Phi_{\barmu_{k+1},\theta_{k+1}}^{k+1} \le \Phi_{\barmu_{k+1},\theta_{k}}^{k+1}$.

For the second inequality, note that $f_{k+1}-f_{\inf}\ge0$, and all terms in
\begin{equation*}
f_{k+1}-f_{\inf}
- \theta_k \sum_{i=1}^n
\ln\left(\frac{\bs_{k+1}^{(i)}}{\Upsilon_s}\right)
+a_{\beta_j}\Delta_{k+1}^2+b\barepsilon_{k+1}
\end{equation*}
are nonnegative. Since $\barmu_{k+1}\ge\barmu_k>0$, the definition of $\Phi_{\barmu_{k+1},\theta_k}^{k+1}$ gives
\begin{multline*}
\Phi_{\barmu_{k+1},\theta_k}^{k+1} \leq \Phi_{\barmu_{k} , \theta_k}^{k+1} =
\frac{1}{\barmu_k}\Bigg[
f_{k+1}-f_{\inf}
- \theta_k \sum_{i=1}^n
\ln\left(\frac{\bs_{k+1}^{(i)}}{\Upsilon_s}\right) \\
+\barmu_k\|(c_{k+1};h_{k+1}+\bs_{k+1})\|
+a_{\beta_j}\Delta_{k+1}^2
+b\barepsilon_{k+1}
\Bigg].
\end{multline*}
Combining the two inequalities proves the result.
\end{proof}
}

In what follows, we investigate the one-step performance of Algorithm \ref{Alg:STORM} by examining the reduction in the Lyapunov function when iteration $k$ is in stage $j$.

We first consider the case when the estimates of objective values are accurate.

\begin{lemma}\label{lemma:ABCC'_hold}
Under Assumptions \ref{assump1}, \ref{assump4}, on the event $\B_k$ and for every iteration $k$ in stage $j$, we have
{\red \begin{equation*}
\Phi_{\barmu_{k+1},\theta_{k+1}}^{k+1} - \Phi_{\barmu_{k},\theta_{k}}^{k} \leq \frac{1}{\barmu_k} \left[ a_{\beta_j}\left(\frac{1}{\gamma^2}- 1\right)\Delta_k^2 + b\left(\frac{1}{\gamma}- 1\right) \barepsilon_k \right] .
\end{equation*}
}
\end{lemma}

\begin{proof}
See Appendix \ref{proof_of_lemma:ABCC'_hold}.
\end{proof}

We now examine the reduction in the Lyapunov function when estimates of the objective values are inaccurate.

\begin{lemma}\label{lemma:Phi_(ABCC')^c}
Under Assumptions \ref{assump1}, \ref{assump4}, on the event $\B_k^c$ and for every iteration $k$ in stage $j$, we have
\begin{equation}\label{eq:lemma_phi_reduction_not_all_acc}
\Phi_{\barmu_{k+1},\theta_{k+1}}^{k+1} - \Phi_{\barmu_{k},\theta_{k}}^{k} \leq \frac{1}{\barmu_{k}} \left[ e_k + e_{s_k} + a_{\beta_j}\left(\frac{1}{\gamma^2}-1\right)\Delta_k^2 + b\left(\frac{1}{\gamma}-1\right)\barepsilon_k \right].
\end{equation}
\end{lemma}

\begin{proof}
See Appendix \ref{proof_of_lemma:Phi_(ABCC')^c}
\end{proof}

Lemma \ref{lemma:ABCC'_hold} shows that when objective value estimates are accurate, a decrease in the Lyapunov function is guaranteed. Lemma \ref{lemma:Phi_(ABCC')^c} shows that when objective value estimates are inaccurate, the value of the Lyapunov function may increase. In the next lemma, we show that as long as the probability of obtaining an inaccurate objective value estimate is below a deterministic threshold, a reduction in the Lyapunov function is guaranteed in expectation.

\begin{lemma}\label{lemma:One_Step_Rec}
Under Assumptions \ref{assump1}, \ref{assump4}, for every iteration $k$ in stage $j$, if 
\begin{equation}\label{upper_bound_pf}
p_f\leq\frac{\eta^2}{64(\gamma+1)^2},
\end{equation}
then we have
{\red
\begin{equation}\label{eq25}
\mathbb E[ \Phi_{\barmu_{k+1},\theta_{k+1}}^{k+1} - \Phi_{\barmu_{k},\theta_{k}}^{k} \mid \mathcal F_{k-1/2}]\leq \frac{a_{\beta_j}}{\hatmu_{j}} \left(\frac{1}{\gamma^{2}}-1\right)\Delta_k^2.
\end{equation}
}
\end{lemma}

\begin{proof}
Combining the preceding two lemmas yields
\begin{equation*}
\Phi_{\barmu_{k+1},\theta_{k+1}}^{k+1} - \Phi_{\barmu_{k},\theta_{k}}^{k}
\leq \frac{1}{\barmu_{k}}\left[
a_{\beta_j}\left(\frac{1}{\gamma^{2}}-1\right)\Delta_k^2
+b\left(\frac{1}{\gamma}-1\right)\barepsilon_k
+(e_k+e_{s_k})\mathbf 1_{\mathcal B_k^c}\right].
\end{equation*}
By the Cauchy--Schwarz inequality and Definition \ref{def:Probabilistic zeroth-order oracle}, we have
\begin{align*}
\mathbb E[ (e_k+e_{s_k})\mathbf 1_{\mathcal B_k^c} \mid\mathcal F_{k-1/2}] & \leq \left( \mathbb E[(e_k+e_{s_k})^2\mid\mathcal F_{k-1/2}] \right)^{1/2}
\left( P(\mathcal B_k^c\mid\mathcal F_{k-1/2}) \right)^{1/2}\\
& \leq \left( 2\mathbb E[e_k^2\mid\mathcal F_{k-1/2}] +2\mathbb E[e_{s_k}^2\mid\mathcal F_{k-1/2}] \right)^{1/2} \left( P(\mathcal B_k^c\mid\mathcal F_{k-1/2}) \right)^{1/2}\\
&\leq 2 \barepsilon_k\sqrt{p_f}.
\end{align*}
Therefore, taking expectation conditional on $\F_{k-1/2}$, we have
\begin{equation*}
\mathbb E[\Phi_{\barmu_{k+1},\theta_{k+1}}^{k+1} - \Phi_{\barmu_{k},\theta_{k}}^{k} \mid \mathcal F_{k-1/2}]\leq \frac{1}{\barmu_{k}} \left[ a_{\beta_j}\left(\frac{1}{\gamma^{2}}-1\right)\Delta_k^2 +b\left(\frac{1}{\gamma}-1\right)\barepsilon_k +2\sqrt{p_f}\,\barepsilon_k \right].
\end{equation*}
Since \eqref{upper_bound_pf} implies $2\sqrt{p_f} \leq\frac{\eta}{4(\gamma+1)} = b\left(1-\frac{1}{\gamma}\right)$, we further have
\begin{equation*}
\mathbb E[\Phi_{\barmu_{k+1},\theta_{k+1}}^{k+1} - \Phi_{\barmu_{k},\theta_{k}}^{k}\mid \mathcal F_{k-1/2}]\leq \frac{a_{\beta_j}}{\barmu_{k}} \left(\frac{1}{\gamma^{2}}-1\right)\Delta_k^2 \leq \frac{a_{\beta_j}}{\hatmu_{j}} \left(\frac{1}{\gamma^{2}}-1\right)\Delta_k^2,
\end{equation*}
where the last inequality uses $\barmu_k \leq \hatmu_j$ as iteration $k$ is in stage $j$. This completes the proof.
\end{proof}

Summing the expected decrease in Lemma \ref{lemma:One_Step_Rec} over the iterations of a stage, we show next that the squared trust-region radii are summable within each stage. In particular, if a stage never terminates, the trust-region radius converges to zero.

{\red \begin{lemma}\label{lemma:vanish_Delta_k}
Under Assumptions \ref{assump1} and \ref{assump4}, suppose $p_f$ satisfies \eqref{upper_bound_pf}. Fix $j\ge 1$. For $m\ge 0$ and $M>0$, define $\mathcal{E}_{m,M} := \{N_j = m\}\cap\big\{\Phi^{m}_{\barmu_m,\theta_m}\le M\big\}$. Then we have
\begin{equation}\label{eq:cor-bound}
\mathbb{E}\Big[\mathbf{1}_{\mathcal{E}_{m,M}}\sum_{k=m}^{N_{j+1}-1}\Delta_k^2\Big] \le \frac{M\hat\mu_j}{a_{\beta_j}(1-\gamma^{-2})},
\end{equation}
where, if stage $j$ never terminates, $N_{j+1}=\infty$ and the sum is an infinite series. Consequently, almost surely on $\{N_j<\infty\}$, $\sum_{k=N_j}^{N_{j+1}-1}\Delta_k^2<\infty$. In particular, almost surely on $\{N_j<\infty,\,N_{j+1}=\infty\}$, $\sum_{k\ge N_j}\Delta_k^2<\infty$ and $\Delta_k\to 0$.
\end{lemma}

\begin{proof}
In this proof, we denote $N \coloneqq N_{j+1}$ and $\mathcal{E}\coloneqq \mathcal{E}_{m,M}$ to simplify notation. For $k\ge m$, we define the stopped process
\begin{equation*}
X_k \coloneqq \1_{\mathcal{E}}\Big(\Phi_{\barmu_{k\wedge N},\theta_{k\wedge N}}^{k\wedge N}
+\frac{a_{\beta_j}}{\hat\mu_j}(1-\gamma^{-2})
\sum_{i=m}^{(k\wedge N)-1}\Delta_i^2\Big),
\end{equation*}
where $\theta_{k\wedge N}$ is the barrier parameter at iteration $k\wedge N$. Whether the barrier parameter is reduced at the end of iteration $i-1$ is determined by the validation test \eqref{val_condition}, which depends only on the samples up to iteration $i-1$. Hence, $\theta_i$ is $\F_{i-1}$-measurable, so $\{N_j=i\}\in\F_{i-1}\subset\F_{i-1/2}$ and $\{N_{j+1}\le i\}\in\F_{i-1}\subset\F_{i-1/2}$ for all $i$; that is, $N_j$ and $N_{j+1}$ are stopping times with respect to $\{\F_{k-1/2}\}$. In particular, $\mathcal{E}\in\F_{m-1/2}$. Since, in addition, each $\Phi_{\barmu_i,\theta_i}^i$ is nonnegative and $\F_{i-1/2}$-measurable, and each $\Delta_i$ is $\F_{i-1}$-measurable, $X_k\ge0$ and $X_k$ is $\F_{k-1/2}$-measurable. Moreover, $X_m=\1_{\mathcal{E}}\Phi^{m}_{\barmu_m,\theta_m}\le M$.
 
On $\{k<N\}$ we have $k\wedge N=k$ and $(k+1)\wedge N=k+1$, while on $\{k\ge N\}$ both equal $N$. Therefore, we have
\begin{equation*}
X_{k+1}-X_k=\1_{\mathcal{E}}\1_{\{k<N\}}\Big(\Phi_{\barmu_{k+1},\theta_{k+1}}^{k+1} - \Phi_{\barmu_k,\theta_k}^k
+\frac{a_{\beta_j}}{\hat\mu_j}(1-\gamma^{-2})\Delta_k^2\Big).
\end{equation*}
On $\mathcal{E}\cap\{k<N\}$, iteration $k$ is in stage $j$ and $\theta_{k+1}\in\{\beta_j,\beta_{j+1}\}$. Taking conditional expectations given $\F_{k-1/2}$, \eqref{eq25} gives
\begin{equation*}
\mathbb{E}[X_{k+1}\mid\F_{k-1/2}]
\le X_k+\1_{\mathcal{E}}\1_{\{k<N\}}\frac{a_{\beta_j}}{\hat\mu_j}
\Big(\big(\gamma^{-2}-1\big)+\big(1-\gamma^{-2}\big)\Big)\Delta_k^2
=X_k .
\end{equation*}
Taking expectations and inducting yields $\mathbb{E}[X_k]\le\mathbb{E}[X_m]\le M$ for all $k\ge m$.
 
Since the stopped potential is nonnegative,
\begin{equation*}
\frac{a_{\beta_j}}{\hat\mu_j}(1-\gamma^{-2})\,
\mathbb{E}\Big[\1_{\mathcal{E}}\sum_{i=m}^{(k\wedge N)-1}\Delta_i^2\Big]\le\mathbb{E}[X_k]\le M
\qquad\text{for all }k\ge m .  
\end{equation*}
As $k\to\infty$, the partial sums increase to $\sum_{i=m}^{N-1}\Delta_i^2$. The monotone convergence theorem therefore gives \eqref{eq:cor-bound}.
 
By \eqref{eq:cor-bound}, $\sum_{i=m}^{N-1}\Delta_i^2<\infty$ almost surely on $\mathcal{E}_{m,M}$ for every $M>0$. On $\{N_j=m\}$, every term of $\Phi^{m}_{\barmu_m,\theta_m}$ is finite, since $\bs_m>0$, $0<\barmu_0\le\barmu_m\le\hat\mu_j$ by Assumption \ref{assump4}, and $\Delta_m,\barepsilon_m<\infty$. Hence $\bigcup_{M\in\mathbb{N}}\mathcal{E}_{m,M}=\{N_j=m\}$, and $\sum_{k=N_j}^{N_{j+1}-1}\Delta_k^2<\infty$ almost surely on $\{N_j=m\}$. Taking the countable union over $m\geq 0$ gives the same conclusion almost surely on $\{N_j<\infty\}$. On $\{N_j<\infty,\,N_{j+1}=\infty\}$ this sum equals $\sum_{k\ge N_j}\Delta_k^2<\infty$, and hence $\Delta_k\to 0$. We complete the proof.
\end{proof}
}

We next show that a stage that never terminates must contain iterates with arbitrarily small stationarity measure $\|Q_k\|$. 

\begin{lemma}\label{lemma:liminf_Qk_0}
Under Assumptions \ref{assump1} and \ref{assump4}, suppose $p_f$ satisfies \eqref{upper_bound_pf}, and $p_g+p_f<\tfrac12$. Fix any stage $j\ge1$. Then almost surely on $\{N_j<\infty,\ N_{j+1}=\infty\}$ (stage $j$ is entered and never ends), we have
\begin{equation*}
\liminf_{k\to\infty}\|Q_k\|=0 .
\end{equation*}
\end{lemma}

\begin{proof}
We note that  $\tilde{w}_k = \sum_{i=0}^{k-1}(\boldsymbol{1}_{(\A_i\cap\B_i)} - \mE[\boldsymbol{1}_{(\A_i\cap\B_i)} | \mF_{i-1}])$ is a martingale since
\begin{equation*}
\mE[\tilde{w}_{k+1}|\mF_{k-1}] = \tilde{w}_k + \mE[\boldsymbol{1}_{(\A_k\cap\B_k)} | \mF_{k-1}] - \mE[\boldsymbol{1}_{(\A_k\cap\B_k)} | \mF_{k-1}] = \tilde{w}_k.
\end{equation*}
Using the fact that $\boldsymbol{1}_{(\A_k\cap\B_k)}\leq 1$ and \cite[Theorem 2.19]{Hall2014Martingale}, we know $\tilde{w}_k/k\rightarrow 0$ a.s. Define $w_k = \sum_{i=0}^{k-1}(2\cdot\boldsymbol{1}_{(\A_i\cap\B_i)}-1)$. Since $\A_i\in\mF_{i-1/2}$, we have $\mE[\boldsymbol{1}_{(\A_i\cap\B_i)} | \mF_{i-1}] = \mE[\boldsymbol{1}_{\A_i}P(\B_i\mid\mF_{i-1/2}) | \mF_{i-1}]\geq (1-p_f)(1-p_g)\geq 1-p_g-p_f$, and thus we have
\begin{equation*}
\frac{w_k}{k} = \frac{2\tilde{w}_k}{k} + \frac{1}{k}\sum_{i=0}^{k-1}(2\mE[\boldsymbol{1}_{(\A_i\cap\B_i)} | \mF_{i-1}]-1) \geq  \frac{2\tilde{w}_k}{k} + 1 - 2(p_g+p_f).
\end{equation*}
As  $p_g+p_f< 0.5$, we know from the above display that $w_k\rightarrow \infty$ almost surely. Next, we prove $\liminf_{k\rightarrow\infty}\|Q_k\|=0$ by contradiction.
 
Suppose that on the event $\{N_j<\infty,\ N_{j+1}=\infty\}$, there exist $\epsilon_1>0$ and $K_1\ge N_j$ such that $\|Q_k\|\ge\epsilon_1$ for all $k\ge K_1$. Let
\begin{equation*}
\alpha \coloneqq \min\Big\{\frac{\Delta_{\max}}{\gamma},\ \varphi\epsilon_1,\ \epsilon_s\Big\}
\end{equation*}
with $ \varphi:=\min\left\{\frac{4(1-\eta)\kappa_{fcd}}{8\kappa_{R,\beta_j}\Upsilon_{1,j}+\kappa_{fcd}\min\{1,\frac{\epsilon_s}{\Delta_{\max}}\}\eta^3+4(1-\eta)\kappa_{fcd}\kappa_g},\;\frac{1}{\kappa_g+\kappa_{W,\beta_j}}\right\}$. Since $\Delta_k\to0$ (cf. Lemma \ref{lemma:vanish_Delta_k}), there exists $K_1'\ge K_1$ such that $\Delta_k\le \alpha$ for all $k\ge K_1'$. Therefore, for all $k\geq K_1'$, we have $\Delta_k \leq \min\{\varphi\|Q_k\|,\epsilon_s\}$, which combined with Lemma \ref{lemma:guarantee_succ_step_KKT} shows that if $\A_k\cap\B_k$ holds, the iteration must be successful. Since $\Delta_k\leq\Delta_{\max}/\gamma$, we have $\Delta_{k+1}=\gamma\Delta_k$. On the other hand, if $(\A_k\cap\B_k)^c$ holds, the iteration can be successful or not, in which case we have $\Delta_{k+1} \geq \Delta_k/ \gamma$. Let $l_k=\log_\gamma\left(\frac{\Delta_k}{\alpha}\right)$, which satisfies $l_k\leq 0$ for all $k\geq K_1'$. By this definition, we know for $k\geq K_1'$, if $\A_k\cap\B_k$ holds, then $l_{k+1}=l_k+1$; otherwise, $l_{k+1} \geq l_k-1$. Comparing this jumping process with $\{w_k\}$, we know $l_k-l_{K_1'} \geq w_k-w_{K_1'}$ for all $k\geq K_1'$. Thus, $l_k\rightarrow\infty$, which contradicts $l_k\leq 0$ for all $k\geq K_1'$. This contradiction concludes $\liminf_{k\rightarrow\infty}\|Q_k\|=0$. We thus complete the~proof.
\end{proof}

\subsection{Global convergence}\label{sec:4.2}

Now we establish the global convergence of Algorithm \ref{Alg:STORM}. We first show that the barrier parameter is reduced infinitely often, thus every stage is entered and terminates after finitely many iterations almost surely.

{\red \begin{lemma}\label{lemma:finite_termination}
Under Assumptions \ref{assump1} and \ref{assump4}, suppose $p_f$ satisfies \eqref{upper_bound_pf},
and $p_g+p_f<\tfrac12$. Then, almost surely,
\begin{enumerate}
\item[(i)] $\mathcal{C}_k^c$ occurs for only finitely many $k$; and
\item[(ii)] $N_j<\infty$ for all $j\ge1$, i.e., every stage is entered and terminates after finitely many iterations.
\end{enumerate}
\end{lemma}
 
\begin{proof}
(i) By Definition \ref{def:val}, $P(\mathcal{C}_k^c)=\mathbb{E}[P(\mathcal{C}_k^c\mid \F_{k-1/4})]\le\delta_k$, so
\begin{equation}\label{eq:BC_lemma}
\textstyle\sum_{k=0}^\infty P(\mathcal{C}_k^c)\le\sum_{k=0}^\infty\delta_k<\infty,
\end{equation}
and (i) follows from the Borel--Cantelli lemma.

(ii) Fix $j\ge1$. By (i) and Lemma \ref{lemma:liminf_Qk_0}, almost surely on $\{N_j<\infty,\ N_{j+1}=\infty\}$ there exists $K$ such that $\mathcal{C}_k$ holds for all $k\ge K$ and $\liminf_{k\to\infty}\|Q_k\|=0$. Pick $k\ge\max\{K,N_j\}$ with $\|Q_{k+1}\|\le\epsilon_j$. Since $N_{j+1}=\infty$, we have $\theta_k=\theta_{k+1}=\beta_j$. Thus $\bar\psi^{\val}_{k+1}$ and $\psi_{k+1}$ share the barrier term $-\beta_jS_{k+1}^{-1}\boldsymbol{e}$. Hence
\begin{equation*}
\|\bar Q^{\val}_{k+1}\|\le\|Q_{k+1}\|+\|\bar g^{\val}_{k+1}-\nabla f_{k+1}\|\le2\epsilon_j,
\end{equation*}
so \eqref{val_condition} passes at iteration $k$, contradicting $N_{j+1}=\infty$. This leads to
$P(\{N_j<\infty,\ N_{j+1}=\infty\})=0$. Since $N_1=0$ and
$\{N_{j+1}=\infty\}\subseteq\{N_j=\infty\}\cup \{N_j<\infty,\ N_{j+1}=\infty\}$, induction gives
$P(N_j=\infty)=0$ for all $j$. We complete the proof.
\end{proof}

Since every stage terminates almost surely, the iterations $N_{j+1}$ at which the barrier parameter is reduced are well defined for all $j$. The main theorem below shows that the residuals evaluated at these iterations vanish, and that their limit points are KKT points of Problem \eqref{Intro_StoProb}.

\begin{theorem}\label{thm:global_conv}
Under Assumptions \ref{assump1} and \ref{assump4}, suppose $p_f$ satisfies \eqref{upper_bound_pf}, and $p_g+p_f<\tfrac12$.  Let $Q^{\val}_{N_{j+1}}$, $V^{\val}_{N_{j+1}}$,  $\blambda^{\val}_{N_{j+1}}$, and $\btau^{\val}_{N_{j+1}}$ denote the quantities in \eqref{def:true_KKT_measure}, \eqref{def:KKT_measure} and \eqref{def:true_Lag_para}, evaluated at $(\bx_{N_{j+1}},\bs_{N_{j+1}})$ with barrier parameter $\beta_j$. Then, almost surely,
\begin{equation*}
\lim_{j\to\infty}\|Q^{\val}_{N_{j+1}}\|=0
\qquad\text{and}\qquad
\lim_{j\to\infty}\|V^{\val}_{N_{j+1}}\|=0,
\end{equation*}
and every limit point of $\{\bx_{N_j}\}$ is a KKT point of \eqref{Intro_StoProb}. In particular, if $\{\bx_{N_j}\}$ is bounded, it has KKT limit points.
\end{theorem}

\begin{proof}
By Lemma \ref{lemma:finite_termination}, every stage terminates after finitely many iterations almost surely. Hence the indices $N_{j+1}$ are well defined for all $j$. Since $N_{j+1}\geq N_j+1$, we have $N_{j+1}\to\infty$ as $j\to\infty$. 

The following relations hold for any iteration $k$. In this paragraph, let $\blambda_k,\btau_k$ be given by \eqref{def:true_Lag_para}, and let $\psi_k$, $Q_k$, $V_k$ be given by \eqref{def:true_KKT_measure} and \eqref{def:KKT_measure}, all with $\theta_k$ replaced by $\beta_j$. We denote $r_k=\nabla f_k+G_k^T\blambda_k+J_k^T\btau_k$ and $P_{G,k}:=I-G_k^T[G_kG_k^T]^{-1}G_k$. By the definitions in \eqref{def:true_Lag_para}, we have $  r_k = P_{G,k}(\nabla f_k+J_k^T\btau_k) =(P_{G,k},-P_{G,k}J_k^T) \psi_k $. Since $A_k^T=\begin{pmatrix}G_k^T&J_k^T\\0&I\end{pmatrix}$, we have
\begin{equation*}
(P_{G,k},-P_{G,k}J_k^T)A_k^T
=(P_{G,k}G_k^T,\;P_{G,k}J_k^T-P_{G,k}J_k^T)=0 .
\end{equation*}
Because $P_k$ is the orthogonal projection onto $\ker (A_k)$, this
implies $(P_{G,k},-P_{G,k}J_k^T) = (P_{G,k},-P_{G,k}J_k^T)P_k$, and so
$r_k=(P_{G,k},-P_{G,k}J_k^T)P_k\psi_k$. Since $\|P_{G,k}\|\leq1$ and $\|(I,-J_k^T)\|=(1+\|J_k\|^2)^{1/2}\leq\sqrt{1+\kappa_{2,A}}$ (cf. Assumption \ref{assump1}), we have
\begin{equation*}
\| r_k \| \leq\| P_{G,k}(I,-J_k^T)\|\,\|P_k\psi_k\| \leq \sqrt{1+\kappa_{2,A}} \|P_k\psi_k\| .
\end{equation*}
In addition, since $\psi_k+A_k^T(\blambda_k;\btau_k)=(r_k;0)$ and $P_kA_k^T=0$, we find $P_k\psi_k=P_k(r_k;0)$. Therefore, $\|P_k\psi_k\|\leq\|r_k\|$, and we have
\begin{equation}\label{eq:fin-r}
\|P_k\psi_k\|\leq\|r_k\|\leq\sqrt{1+\kappa_{2,A}}\|P_k\psi_k\| .
\end{equation}
Since $-\beta_j S_k^{-1}\boldsymbol{e}+\btau_k=0$, the definition in \eqref{1st_order_point} leads to $\nabla\L_{\beta_j}(\bx_k,\bs_k,\blambda_k,\btau_k) = (r_k;0;\vartheta_k)$. Hence $\|\nabla\L_{\beta_j}(\bx_k,\bs_k,\blambda_k,\btau_k)\|^2=\|r_k\|^2+\|\vartheta_k\|^2$, while $\|Q_k\|^2=\|P_k\psi_k\|^2+\|\vartheta_k\|^2$. Together with \eqref{eq:fin-r}, we find
\begin{equation}\label{eq:fin-Q}
\|Q_k\|\leq\|\nabla\L_{\beta_j}(\bx_k,\bs_k,\blambda_k,\btau_k)\|
\leq\sqrt{1+\kappa_{2,A}}\|Q_k\| .
\end{equation}

To relate $Q_k$ to the residual $V_k$ in \eqref{def:KKT_measure}, note that the componentwise minimum is $1$-Lipschitz in its first argument, so $\left\|\min\{-h_k,\btau_k\}-\min\{\bs_k,\btau_k\} \right\| \leq\|h_k+\bs_k\|$. Also, $\bs_k^{(i)}\btau_k^{(i)}=\beta_j$ and $\min\{a,b\}\leq\sqrt{ab}$ give $\|\min\{\bs_k,\btau_k\}\|\leq\sqrt{n\beta_j}$. Therefore, by \eqref{eq:fin-Q},
\begin{align}
\|V_k\| & =\|(r_k;c_k;\min\{-h_k,\btau_k\})\|\nonumber\\
&\leq\|(r_k;c_k;\min\{-h_k,\btau_k\}-\min\{\bs_k,\btau_k\})\|
+\|\min\{\bs_k,\btau_k\}\|\nonumber\\
&\leq\|(r_k;c_k;h_k+\bs_k)\|+\sqrt{n\beta_j}\nonumber\\
&=\|\nabla\L_{\beta_j}(\bx_k,\bs_k,\blambda_k,\btau_k)\|+\sqrt{n\beta_j}
\;\leq\;\sqrt{1+\kappa_{2,A}}\|Q_k\|+\sqrt{n\beta_j}. \label{eq:fin-V}
\end{align}
Note that the check passes at iteration $N_{j+1}-1$, so 
\begin{equation*}
\|\bar Q^{\val}_{N_{j+1}}\|\leq 2\epsilon_j ,
\end{equation*}
where $\bar Q^{\val}_{N_{j+1}} =(P_{N_{j+1}}\bar\psi^{\val}_{N_{j+1}};\vartheta_{N_{j+1}})$ and $\bar\psi^{\val}_{N_{j+1}}=(\bar g^{\val}_{N_{j+1}};-\beta_jS_{N_{j+1}}^{-1}\boldsymbol{e})$.

By \eqref{eq:BC_lemma} and the Borel--Cantelli lemma, $\mathcal{C}_k$ holds for all large $k$ almost surely. Since $N_{j+1}\to\infty$ and iteration $N_{j+1}-1$ lies in stage $j$, for each run of the algorithm there is a finite $J$ such that, for all $j\geq J$, $\|\bar g^{\val}_{N_{j+1}}-\nabla f_{N_{j+1}}\|\leq\epsilon_j$. For any $j\geq J$, by definition, $Q_{N_{j+1}}^{\val}=(P_{N_{j+1}}\psi_{N_{j+1}}^{\val};\vartheta_{N_{j+1}})$ with $\psi_{N_{j+1}}^{\val} = (\nabla f_{N_{j+1}};-\beta_jS_{N_{j+1}}^{-1}\boldsymbol{e})$. 

As $P_{N_{j+1}}$ is an orthogonal projection and $\epsilon_j\downarrow0$, we have
\begin{equation*}
\|Q_{N_{j+1}}^{\val}\|
\leq\|\bar Q^{\val}_{N_{j+1}}\|
+\|\bar g^{\val}_{N_{j+1}}-\nabla f_{N_{j+1}}\|
\leq 2\epsilon_j+\epsilon_j=3\epsilon_j \longrightarrow0 .
\end{equation*}
Applying \eqref{eq:fin-V} with $k=N_{j+1}$ and noting that $\beta_j\downarrow0$, we also have
\begin{equation*}
\|V_{N_{j+1}}^{\val}\|
\leq 3\sqrt{1 + \kappa_{2,A}}\epsilon_j+\sqrt{n\beta_j}
\longrightarrow0 .
\end{equation*}

Let $\bx_*$ be a limit point of $\{\bx_{N_j}\}$, so along a subsequence $\bx_{N_{j+1}}\rightarrow \bx_*$. The slack variables are bounded (i.e., $\|\bs_k\|\leq\Upsilon_s$), so along a further subsequence, we have $\bs_{N_{j+1}}\rightarrow  \bs_*\geq0$. Here we slightly abuse the notation in defining the subsequences to simplify the presentation. As $\|\vartheta_{N_{j+1}}\|\leq\|Q_{N_{j+1}}^{\val}\| \leq 3\epsilon_j\to0$, and by continuity of $c(\bx)$ and $h(\bx)$ (cf. Assumption \ref{assump1}), we have
\begin{equation*}
c(\bx_*) = \b0, \qquad h(\bx_*) + \bs_* = \b0,\quad\text{and}\quad h(\bx_*)\leq \b0.
\end{equation*}
We denote $\btau_{N_{j+1}}^{\val} = \beta_jS_{N_{j+1}}^{-1}\boldsymbol{e}$, and $\blambda_{N_{j+1}}^{\val}$ is given by \eqref{def:true_Lag_para} with this $\btau_{N_{j+1}}^{\val}$, and $r_{N_{j+1}}^{\val}=\nabla f_{N_{j+1}}+G_{N_{j+1}}^T\blambda_{N_{j+1}}^{\val}+J_{N_{j+1}}^T\btau_{N_{j+1}}^{\val}$. The fraction-to-boundary condition in \eqref{def:SQPsubproblem_1} gives $\bs_{N_{j+1}}^{(i)}\geq(1-\epsilon_s)\bs_{N_{j+1}-1}^{(i)}$ (trivially if iteration $N_{j+1}-1$ is unsuccessful). By Assumption \ref{assump1}, at iteration $N_{j+1}-1$, where $\theta_{N_{j+1}-1}=\beta_j$, we have $\beta_j/\bs^{(i)}_{N_{j+1}-1}\leq\tau^{\max}$. Hence
\begin{equation}\label{eq:bound_tau}
0<(\btau_{N_{j+1}}^{\val})^{(i)} = \frac{\beta_j}{\bs_{N_{j+1}}^{(i)}} \leq\frac{\tau^{\max}}{1-\epsilon_s}.
\end{equation}
Combining \eqref{def:true_Lag_para} and \eqref{eq:bound_tau}, we have
\begin{equation*}
\|\blambda_{N_{j+1}}^{\val}\| \leq\frac1{\sqrt{\kappa_{1,A}}} \left[\kappa_{\nabla f}+\sqrt{\kappa_{2,A}}\,\|\btau_{N_{j+1}}^{\val}\|\right] \leq\frac1{\sqrt{\kappa_{1,A}}} \left[\kappa_{\nabla f}+\sqrt{n\kappa_{2,A}}\, \frac{\tau^{\max}}{1-\epsilon_s}\right].
\end{equation*}
As these multipliers are bounded, along a further subsequence we have
\begin{equation*}
(\blambda_{N_{j+1}}^{\val},\btau_{N_{j+1}}^{\val})\longrightarrow(\blambda_*,\btau_*), \quad \text{and} \quad \btau_*\geq \b0 .
\end{equation*}
By \eqref{eq:fin-r}, \eqref{eq:fin-Q} with $k=N_{j+1}$,
\begin{equation*}
\|r_{N_{j+1}}^{\val}\| \leq \sqrt{1+\kappa_{2,A}} \|P_{N_{j+1}}\psi_{N_{j+1}}^{\val}\| \leq\sqrt{1+\kappa_{2,A}}\|Q_{N_{j+1}}^{\val}\| \leq 3\sqrt{1+\kappa_{2,A}}\epsilon_j\longrightarrow0 .  
\end{equation*}
By continuity of $\nabla f(\bx)$, $G(\bx)$ and $J(\bx)$, passing to the limit gives
\begin{equation*}
\nabla f(\bx_*)+G(\bx_*)^T\blambda_* + J(\bx_*)^T\btau_* = 0 .
\end{equation*}
Finally, for every $i$, we have
\begin{align*}
(\btau_{N_{j+1}}^{\val})^{(i)} h^{(i)}_{N_{j+1}} & =-(\btau_{N_{j+1}}^{\val})^{(i)} \bs_{N_{j+1}}^{(i)} +(\btau_{N_{j+1}}^{\val})^{(i)}(h^{(i)}_{N_{j+1}} + \bs_{N_{j+1}}^{(i)})\\
&=-\beta_j + (\btau_{N_{j+1}}^{\val})^{(i)}(h^{(i)}_{N_{j+1}} + \bs_{N_{j+1}}^{(i)}) \longrightarrow0 ,
\end{align*}
because $(\btau_{N_{j+1}}^{\val})^{(i)}$ is bounded, $\|h_{N_{j+1}} + \bs_{N_{j+1}} \|\to0$ and $\beta_j\to0$. Hence $\btau_*^{(i)}h^{(i)}(\bx_*)=0$ for every $i$. Together with $c(\bx_*)=0$, $h(\bx_*)\leq0$, stationarity and $\btau_*\geq0$, this shows that $(\bx_*,\blambda_*,\btau_*)$ satisfies the KKT conditions \eqref{KKT1}. We complete the proof.
\end{proof}

In stochastic optimization, establishing liminf-type convergence has become the standard for interior-point methods. Previous gradient-based algorithms have achieved this for bound-constrained problems \citep{Curtis2023Stochastic} and for nonlinear inequalities with affine equalities \citep{curtis2026single}. We extend this analysis to sequential quadratic programming, establishing liminf convergence for general nonlinear equality and inequality constraints. Furthermore, we identify the subsequence where the barrier parameter is updated as the one that converges.  

}

\subsection{Merit parameter behavior}
\label{subsec:merit_para}

In this subsection, we re-examine Assumption \ref{assump4} and demonstrate that it is satisfied, with a stage-dependent bound $\hatmu_j$, provided the gradient estimation error is bounded. The assumption that the gradient estimation error is bounded is standard in the SSQP literature for constrained optimization; see, e.g.,  \cite{Berahas2021Stochastic,Berahas2021Sequential,Curtis2021Inexact,Na2022adaptive,Na2021Inequality,Fang2024Fully,Fang2024Trust}. It is satisfied, for example, when sampling from an empirical distribution, as is common in many machine learning problems.

\begin{assumption}\label{ass:bdd_err}
{\red There exists $M>0$ such that $\|\nabla f_k-\barg_k\|\leq M$ for all $k \geq 0$.}
\end{assumption}

{\red \begin{lemma}\label{lemma: mu_stabilize}
Suppose that Assumptions \ref{assump1} and \ref{ass:bdd_err} hold. Then \eqref{eq:threshold_Predk} can be satisfied. Furthermore, for every $j\geq1$ there exists a deterministic constant $\hatmu_j$ such that $\barmu_k\leq\hatmu_j$ for all $k$ in stage $j$; that is, Assumption \ref{assump4} holds. Moreover, $\barmu_k$ is increased only finitely many times in each stage.
\end{lemma}
}

\begin{proof}
See Appendix \ref{proof_of_lemma: mu_stabilize}.
\end{proof}

\section{Numerical Experiments}\label{sec:5}

In this section, we explore the empirical performance of TR-IP-SSQP (Algorithm~\ref{Alg:STORM})\footnote{Code and datasets are available at \url{https://github.com/ychenfang/TRIPSSQP}.}. We first evaluate its performance on a subset of inequality-constrained problems from the CUTEst test set~\cite{Gould2014CUTEst}, and then on constrained logistic regression problems built from UCI datasets~\cite{kelly2023uci} and synthetic datasets.

{\red We compare our method with the active-set SSQP method of~\cite{Na2021Inequality}, which uses a line search and an exact augmented Lagrangian merit function to handle stochastic objectives with deterministic equality and inequality constraints. We denote this method by LS-SSQP.} To compare adaptive and fixed sampling, we further benchmark TR-IP-SSQP against its \emph{fully stochastic} counterpart, referred to as Fully-TR-IP-SSQP. Fully-TR-IP-SSQP follows the framework of~\cite{Fang2024Fully} while handling inequality constraints by interior-point methods. This variant employs a fixed sampling scheme in which a single sample is drawn per iteration to estimate the objective gradient; no objective value estimate is required, and the iterates are always updated. The main modifications relative to~\cite{Fang2024Fully} are as follows. In generating the trust-region radius, we replace $\|\bar\nabla\mathcal L_k\|$ with $\|\bar Q_k\|$. The trial step is computed as described in Section \ref{sec: Computing the (rescaled) trial step}, with the merit function defined in \eqref{merit_func}. The threshold for merit parameter (i.e., \eqref{eq:threshold_Predk}) is replaced by $\mathrm{Pred}_k\le - \|\bar Q_k\|\min\{\Delta_k,\epsilon_s\}+\frac12\|\bar W_k\|\min\{\Delta_k^2,\epsilon_s^2\}$.

\subsection{Algorithm setup}\label{sec:5.1}

At the $k$-th iteration, TR-IP-SSQP draws $|\xi_k^g|$ samples to estimate the objective gradient, $|\xi_k^f|$ samples to estimate both objective values $\bar f_k$ and $\bar f_{s_k}$, and $|\xi_k^{\val}|$ samples for the validation oracle. If the noise in gradient estimation and objective value estimation is Gaussian with second moments bounded by $\sigma_g^2$ and $\sigma_f^2$, respectively, Definitions~\ref{def:Probabilistic first-order oracle}--\ref{def:val} hold with
\begin{equation}\label{eq:5.1}
|\xi_k^g|=\Big\lceil\frac{c_g\sigma_g^2\log(1/p_g)}{\kappa_g^2\Delta_k^2}\Big\rceil,\quad
|\xi_k^f|=\Big\lceil\frac{c_f\sigma_f^2\log(2/p_f)}{\min\{\kappa_{f,k}^2\Delta_k^4,\bar\epsilon_k^2\}}\Big\rceil,\quad
{\red |\xi_k^{\val}|=\Big\lceil\frac{c_g\sigma_g^2(1+\log(1/\delta_k))}{\epsilon_j^2}\Big\rceil}
\end{equation}
for some constants $c_g$, $c_f$. {\red For computational efficiency, we use $c_g=10^{-6}$, $c_f=10^{-10}$, and $p_g=p_f=0.05$ with capped sample sizes, without enforcing the theoretical oracle conditions or the sufficient bound~\eqref{upper_bound_pf}.} 

The parameters of TR-IP-SSQP are set as follows: $\bar\mu_0=\Delta_0=\bar\epsilon_0=1$, $\Delta_{\max}=10$, $\rho=\gamma=1.5$, $\eta=0.6$, $\zeta=0.5$, $\epsilon_s=0.9$, $\kappa_{fcd}=1$, $\kappa_g=0.01$, and {\red $\kappa_{f,k}$ is given in Definition~\ref{def:Probabilistic zeroth-order oracle}. We let $\beta_j\in\{0.9999^j,0.999^j,j^{-0.1},j^{-0.5}\}$, $\epsilon_j=\beta_j$, and $\delta_k=(k+1)^{-1.1}$.} For Fully-TR-IP-SSQP, following~\cite{Fang2024Fully}, we set $\zeta=10$, $\delta=10$, $\mu_{-1}=1$, $\rho=1.5$, and the radius-scaling parameter $\beta_k=0.5$ (in the notation of~\cite{Fang2024Fully}). For LS-SSQP, we follow the algorithm setup in~\cite{Na2021Inequality} and use the identity matrix as the Hessian estimate.

Since trust-region methods permit indefinite Hessian matrices, we try four different $\bar H_k$ for TR-IP-SSQP and Fully-TR-IP-SSQP:
\begin{enumerate}
\item Identity matrix (Id). This choice is used in various existing SSQP literature, see, e.g., \cite{Berahas2021Stochastic,Berahas2021Sequential,Na2022adaptive,Na2021Inequality}.
\item Symmetric rank-one (SR1) update.  We initialize $\barH_{0}=I$ and for $k\geq 1$, $\barH_k$ is updated as    
\begin{equation*}
\barH_{k}=\barH_{k-1}+\frac{(\boldsymbol{y}_{k-1}-\barH_{k-1}\Delta\bx_{k-1})(\boldsymbol{y}_{k-1}-\barH_{k-1}\Delta\bx_{k-1})^T}{(\boldsymbol{y}_{k-1}-\barH_{k-1}\Delta\bx_{k-1})^T\Delta\bx_{k-1}},
\end{equation*}
where $\boldsymbol{y}_{k-1}=\bar{\nabla}_{\bx}\L_{k}-\bar{\nabla}_{\bx}\L_{k-1}$, and $\Delta\bx_{k-1}=\bx_{k}-\bx_{k-1}$.
\item Estimated Hessian (EstH). At the $k$-th iteration, we estimate the objective Hessian matrix $\bar{\nabla}^2f_k$ using \textit{one} single sample and set $\barH_k=\bar{\nabla}_{\bx}^2\L_k=\bar{\nabla}^2f_k+{\red\sum_{i=1}^m}\bar{\blambda}_k^{(i)}\nabla^2 c^{(i)}_k + {\red \sum_{i=1}^n}\btau_k^{(i)}\nabla^2 h^{(i)}_k$ with $\btau_k=\theta_kS_k^{-1}\boldsymbol{e}$ and $\bar{\blambda}_k=-[G_kG_k^T]^{-1}G_k\bigl(\barg_k+J_k^T\btau_k\bigr)$.
\item Averaged Hessian (AveH). At the $k$-th iteration, we estimate the Hessian matrix as in EstH and set $\barH_k=\frac{1}{50}\sum_{i=k-49}^{k}\bar{\nabla}_{\bx}^2\L_i$.
\end{enumerate}

\vskip0.1cm
{\red Assumption~\ref{assump1} requires $\|\bar H_k\|\le\kappa_B$. This bound holds trivially for Id, but not automatically for the other three choices. We therefore rescale every Hessian approximation as $\bar H_k\leftarrow\bar H_k\cdot\min\{1,\,\kappa_B/\|\bar H_k\|_F\}$ with $\kappa_B=10^3$, which costs $\cO(d^2)$ operations and guarantees $\|\bar H_k\|\le\|\bar H_k\|_F\le\kappa_B$.}

\begin{figure}[t]
\centering
\subfigure[$\beta_j=0.9999^j$]{\includegraphics[trim=0 4.8bp 0 26.88bp,clip,width=0.45\textwidth]{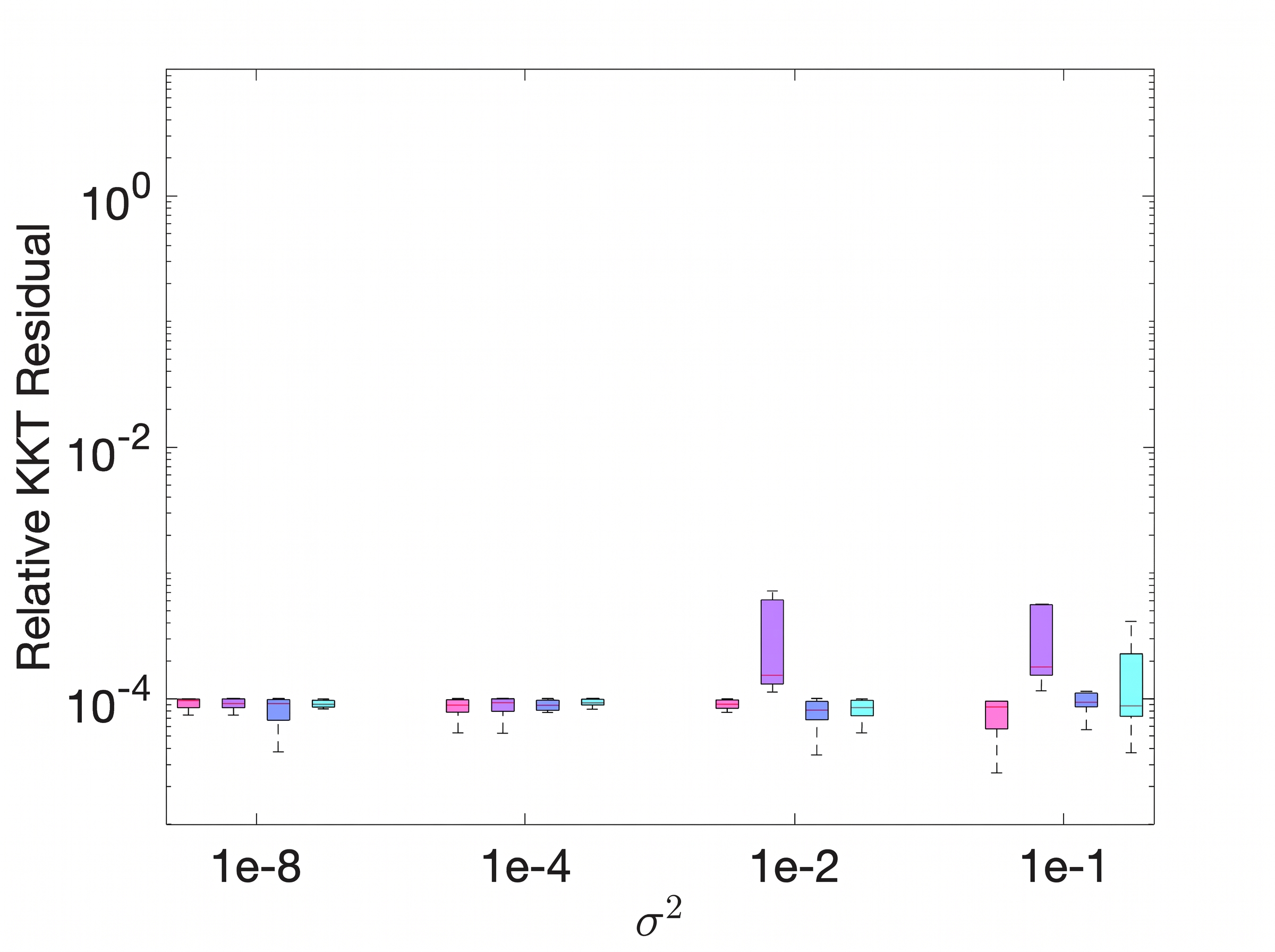}}\quad
\subfigure[$\beta_j=0.999^j$]{\includegraphics[trim=0 20bp 0 112bp,clip,width=0.45\textwidth]{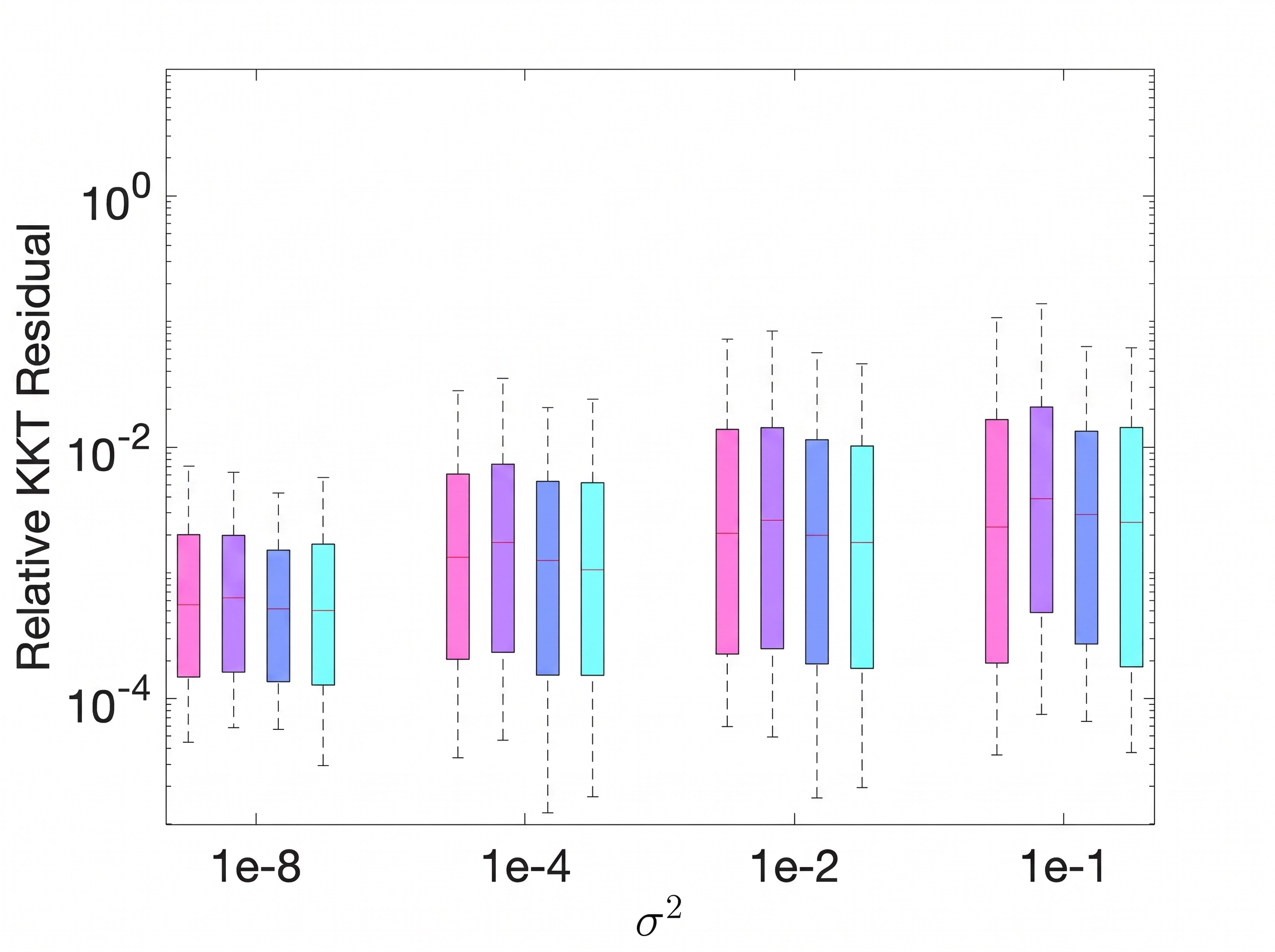}}\par
\subfigure[$\beta_j=j^{-0.1}$]{\includegraphics[trim=0 4.8bp 0 26.88bp,clip,width=0.45\textwidth]{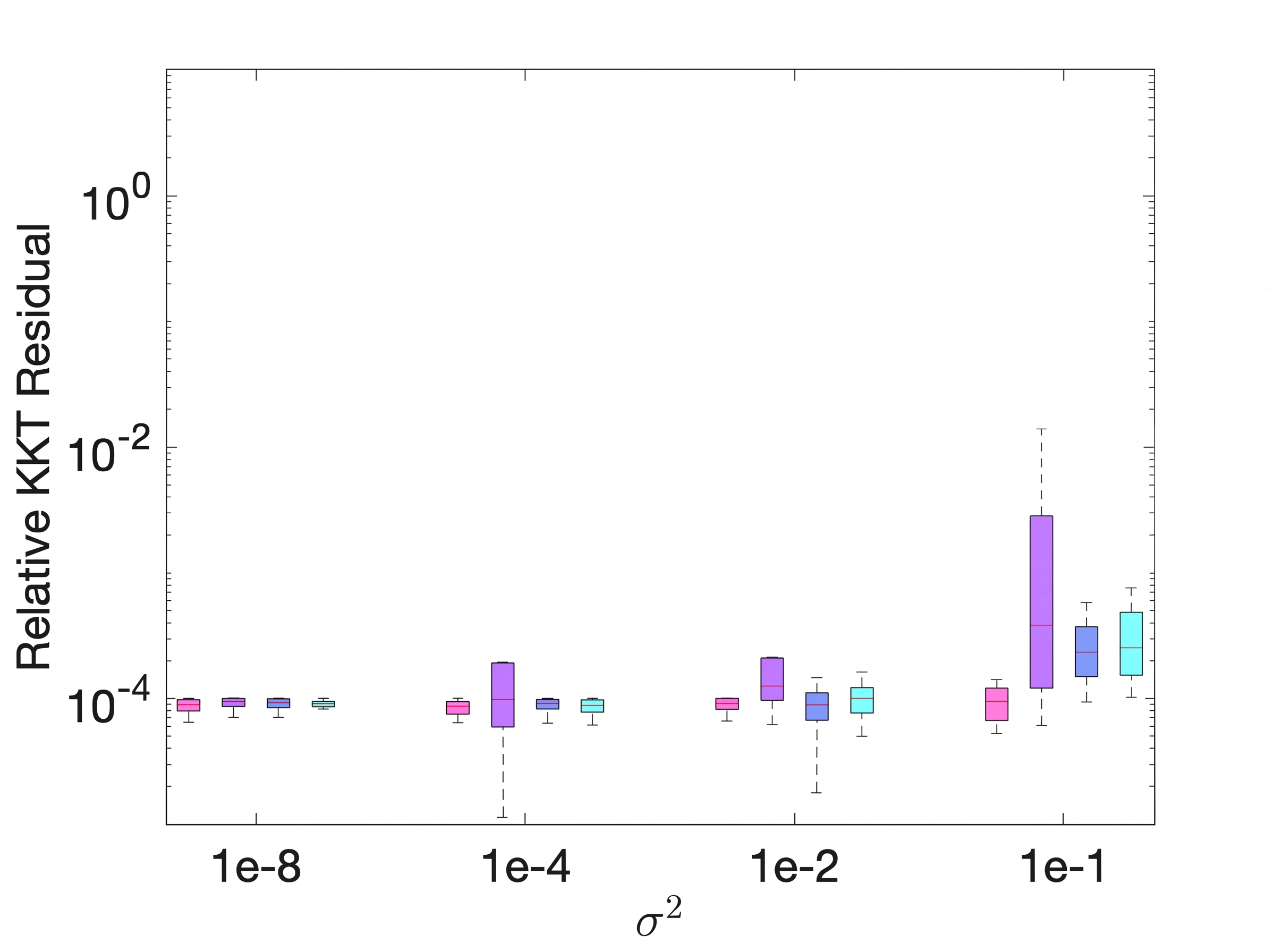}}\quad
\subfigure[$\beta_j=j^{-0.5}$]{\includegraphics[trim=0 20bp 0 112bp,clip,width=0.45\textwidth]{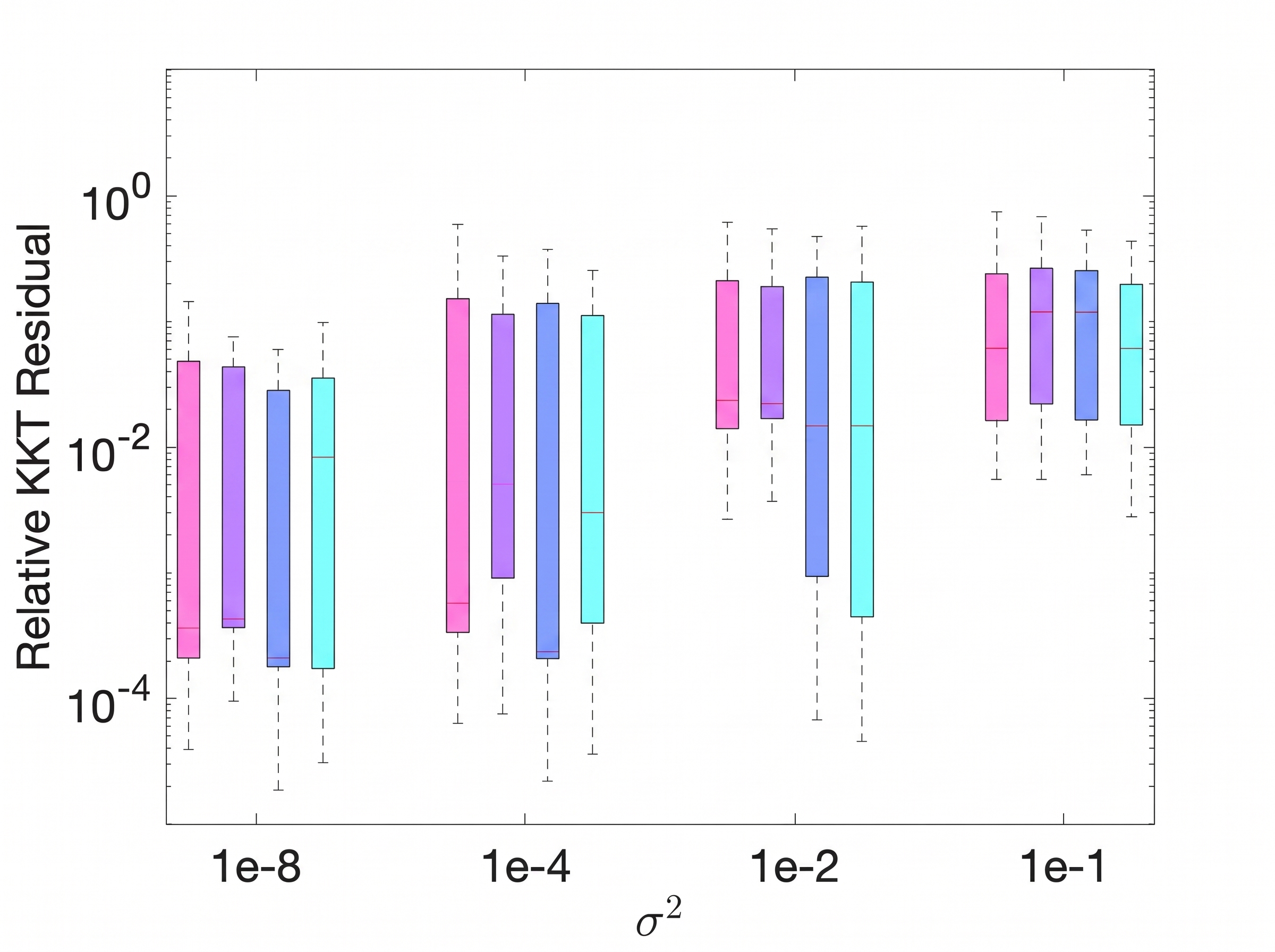}}\par
\includegraphics[width=0.85\textwidth]{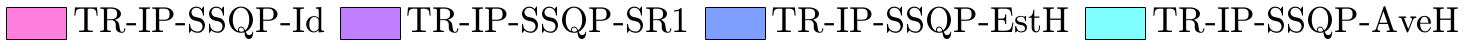}\par
\caption{Relative KKT residuals of TR-IP-SSQP with different barrier parameter sequences. In every plot, each box corresponds to TR-IP-SSQP with one kind of Hessian construction and summarizes the residuals over all test problems.}
\label{fig:cutest}  
\end{figure}

\subsection{CUTEst}\label{subsec:CUTEst}

We select problems from the CUTEst test set that (1) have a non-constant objective with inequality or bound constraints, (2) satisfy $d < 1000$, {\red (3) have no infeasible constraints, (4) have no network constraints, and (5) have fewer constraints than variables, leading to 37 problems in total}. The initialization provided by the CUTEst package is used throughout with each entry of $\bs_0$ generated by taking the absolute value of a standard normal sample. We generate objective model estimates based on deterministic evaluations provided by the CUTEst package. Specifically, the estimate of $f_k$ is drawn from $\N(f_k, \sigma^2)$, the estimate of $\nabla f_k$ is drawn from $\N(\nabla f_k, \sigma^2(I+\boldsymbol{1}\boldsymbol{1}^T))$, where $\boldsymbol{1}$ denotes the $d$-dimensional all-one vector. The $(i,j)$ (and $(j,i)$) entry of the estimate of $\nabla^2 f_k$ is drawn from $\N((\nabla^2 f_k)_{i,j},\sigma^2)$. In experiments, we set $\sigma_g^2=2d\sigma^2$ and $\sigma_f^2=\sigma^2$. We experiment with four noise levels: $\sigma^2 \in \{10^{-8}, 10^{-4}, 10^{-2}, 10^{-1}\}$ and three experiments, with a maximum sample size of $10^4$ for each estimate.

\begin{figure}[t]
\centering
\subfigure[$\sigma^2=10^{-8}$]{\performanceprofile{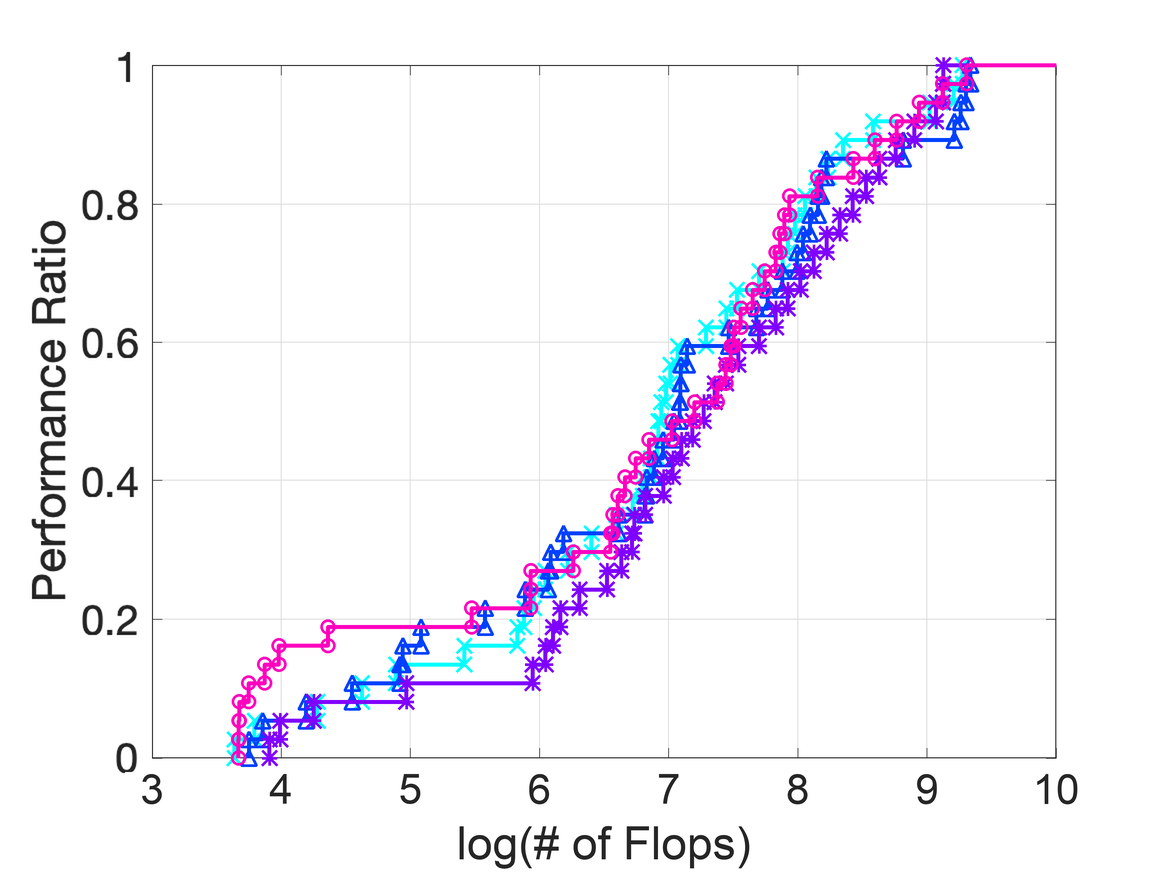}}\quad
\subfigure[$\sigma^2=10^{-4}$]{\performanceprofile{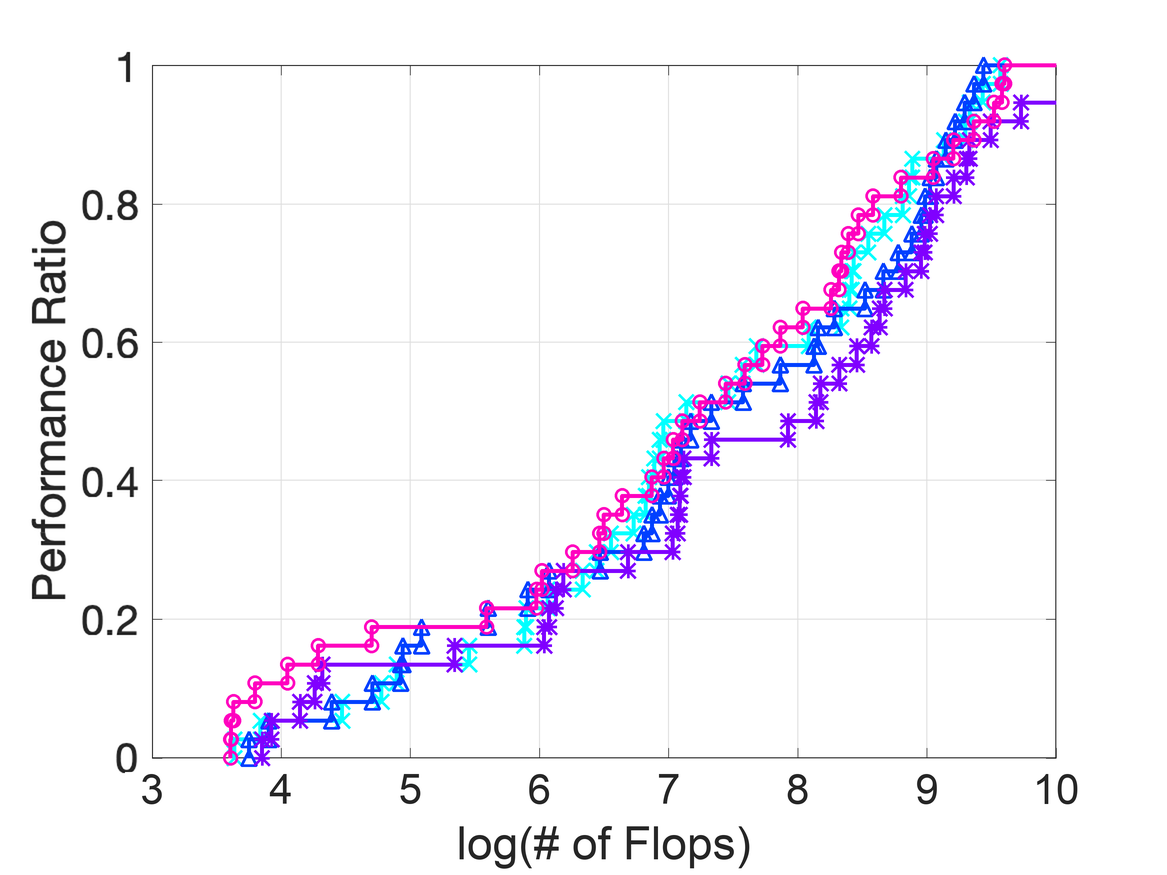}}\par
\subfigure[$\sigma^2=10^{-2}$]{\performanceprofile{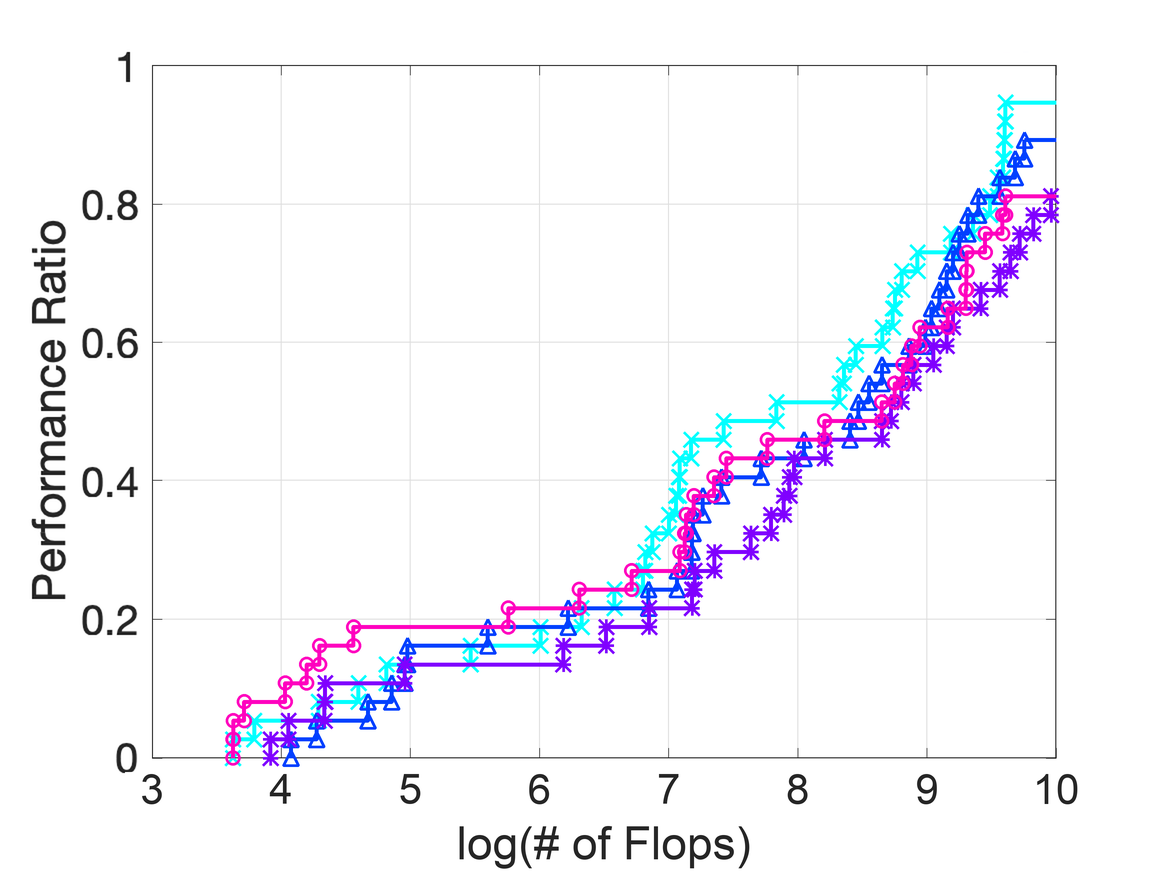}}\quad
\subfigure[$\sigma^2=10^{-1}$]{\performanceprofile{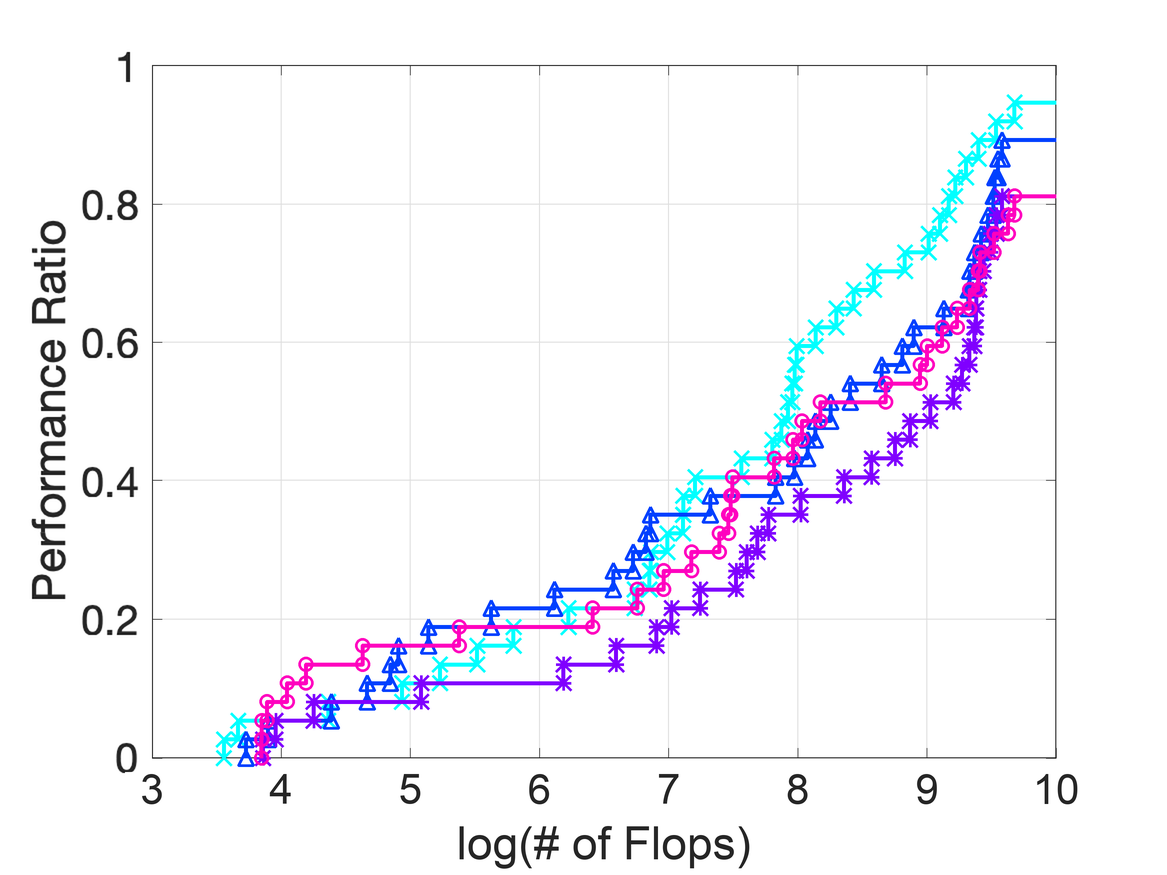}}\par
\includegraphics[width=0.85\textwidth]{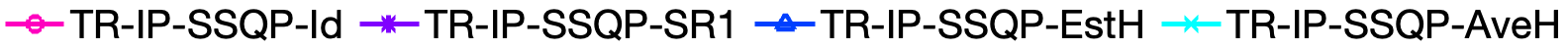}\par
\caption{Performance profiles of TR-IP-SSQP when $\beta_j=0.9999^j$. Each trajectory represents one Hessian construction.}
\label{fig:cutest2}  
\end{figure}

\noindent$\bullet$ \textbf{Experiment 1:} We investigate the effect of the barrier parameter sequence $\{\beta_j\}$ on the performance of TR-IP-SSQP. The computational budget is set to $10^5$ iterations.

\vskip4pt
\noindent $\bullet$ \textbf{Experiment 2:} We examine the impact of different Hessian constructions on the performance of TR-IP-SSQP. We fix $\beta_j=0.9999^j$, and the total computational budget is set to $10^{10}$ floating-point operations (flops). {\red Each flop count includes the cost of sampling: one stochastic gradient or objective value is charged $d$ flops and one stochastic Hessian $d^2$ flops.} In addition, we count the operations of Hessian construction and all other $\cO(d^2)$ operations, while omitting $\cO(d)$ operations; constructing the Hessian costs~$\cO(d)$ for Id and $\cO(d^2)$ for SR1, EstH, and AveH. 

\vskip4pt
\noindent $\bullet$ \textbf{Experiment 3:} We compare TR-IP-SSQP, LS-SSQP, and Fully-TR-IP-SSQP to assess adaptive versus fixed sampling. {\red We fix $\beta_j=0.9999^j$ and impose a total budget of
$10^{10}$ sample-equivalent work units, with the
per-iteration cost defined as
\begin{equation}\label{eq:5.4}
|\xi_k^g|+|\xi_k^f|+|\xi_k^{\val}|+d\cdot\boldsymbol{1}\{\text{EstH or AveH is used}\}.
\end{equation}}
For LS-SSQP and Fully-TR-IP-SSQP, we count all samples they draw in the same way.

For each algorithm and problem instance, under each noise level, we evaluate
\begin{equation*}
\text{Relative KKT Residual} = \frac{\|(\nabla f_k + G_k^T\blambda_k + J_k^T\btau_k; c_k; \min\{-h_k,\btau_k\})\|}{\max\{\|(\nabla f_0 + G_0^T\blambda_0 + J_0^T\btau_0; c_0; \min\{-h_0,\btau_0\})\|,1\}},
\end{equation*}
and terminate the algorithm if $\text{Relative KKT Residual}\leq 10^{-4}$ or the computational budget is exhausted. For Experiment 1, we report the relative KKT residual averaged over 5 independent runs. For Experiments 2 and 3, we report the performance ratio averaged over 5 independent runs, which measures the proportion of test problems that a method solves within a given computational budget.

Results of Experiments 1 and 2 appear in Figures~\ref{fig:cutest} and~\ref{fig:cutest2}, respectively.
When $\beta_j$ decays slowly (Figures~\ref{fig:cutest}(a) and~\ref{fig:cutest}(c)), all methods achieve high accuracy under low noise. At $\sigma^2=10^{-8}$, their relative KKT residuals remain close to $10^{-4}$ with little variability across problems. Thus TR-IP-SSQP reliably approaches stationarity in nearly deterministic settings, regardless of the Hessian approximation. In Figure~\ref{fig:cutest}(a), EstH remains close to $10^{-4}$ across
all noise levels, while SR1 deteriorates and AveH loses accuracy at the
highest noise level. In Figure~\ref{fig:cutest}(c), all variants except
Id deteriorate at the highest noise level. AveH no longer consistently improves performance, unlike in~\cite{Na2022Hessian,Fang2024Fully,Fang2024Trust,Fang2025High}: at the highest noise level, it may attain lower final accuracy than EstH (Figure~\ref{fig:cutest}(a)), although it solves problems with fewer flops (Figure~\ref{fig:cutest2}(d)). A plausible explanation is that $\theta_k$ enters the Hessian through $\btau_k=\theta_kS_k^{-1}\boldsymbol{e}$ and decreases over iterations, so averaging Hessians generated under different $\theta_k$ may not improve curvature accuracy near a solution. SR1 exhibits substantially greater dispersion and degradation, suggesting greater sensitivity to stochastic perturbations, and typically requires more flops even when it converges (Figure~\ref{fig:cutest2}).
The sequence $\{\beta_j\}$ is also critical for robustness. Slowly decreasing $\beta_j$ controls residual growth even at moderate noise levels ($\sigma^2=10^{-4}$ or $10^{-2}$), while faster sequences (Figures~\ref{fig:cutest}(b) and~\ref{fig:cutest}(d)) yield substantially larger residuals even~at~low~noise.

\begin{table}[t]
\centering\small
\caption{Outcomes of Experiment~3: percentages of the test problems
solved / budget exhausted. Entries are medians over five independent runs.}
\label{tab:outcomes}
\resizebox{\textwidth}{!}{
{\red \begin{tabular}{lcccc}
\toprule
Method & $\sigma^2=10^{-8}$ & $\sigma^2=10^{-4}$
       & $\sigma^2=10^{-2}$ & $\sigma^2=10^{-1}$\\
\midrule
TR-IP-SSQP-Id
 & 91.89 / 8.11 & 78.38 / 21.62
 & 67.57 / 32.43 & 59.46 / 40.54\\
TR-IP-SSQP-SR1
 & 91.89 / 8.11 & 64.86 / 35.14
 & 54.05 / 45.95 & 35.14 / 64.86\\
TR-IP-SSQP-EstH
 & 94.59 / 5.41 & 81.08 / 18.92
 & 67.57 / 32.43 & 56.76 / 43.24\\
TR-IP-SSQP-AveH
 & 100.00 / 0.00 & 81.08 / 18.92
 & 70.27 / 29.73 & 48.65 / 51.35\\
Fully-TR-IP-SSQP-Id
 & 100.00 / 0.00 & 81.08 / 18.92
 & 59.46 / 40.54 & 32.43 / 67.57\\
Fully-TR-IP-SSQP-SR1
 & 91.89 / 8.11 & 81.08 / 18.92
 & 62.16 / 37.84 & 37.84 / 62.16\\
Fully-TR-IP-SSQP-EstH
 & 72.97 / 27.03 & 70.27 / 29.73
 & 51.35 / 48.65 & 37.84 / 62.16\\
Fully-TR-IP-SSQP-AveH
 & 78.38 / 21.62 & 70.27 / 29.73
 & 62.16 / 37.84 & 43.24 / 56.76\\
LS-SSQP
 & 75.68 / 24.32 & 70.27 / 29.73
 & 56.76 / 43.24 & 40.54 / 59.46\\
\bottomrule
\end{tabular}
}
}
\end{table}

\begin{table}[t]
\centering\small
\caption{Sample sizes of TR-IP-SSQP-Id in Experiment~3:
median sample sizes, percentages of iterations with truncation,
and the share of objective-value samples in the total cost~\eqref{eq:5.4}. Entries are medians over five independent runs.}
\label{tab:samples}
\resizebox{\textwidth}{!}{
{\red \begin{tabular}{cccccc}
\toprule
$\sigma^2$
& median $|\xi_k^g|$
& median $|\xi_k^f|$
& truncated $|\xi_k^g|$
& truncated $|\xi_k^f|$
& objective cost share\\
\midrule
$10^{-8}$ & $1$ & $228$
 & $1.5\%$ & $35\%$ & $91\%$\\
$10^{-4}$ & $20$ & $10^4$
 & $17\%$ & $62\%$ & $83\%$\\
$10^{-2}$ & $10^4$ & $10^4$
 & $54\%$ & $76\%$ & $65\%$\\
$10^{-1}$ & $10^4$ & $10^4$
 & $69\%$ & $80\%$ & $60\%$\\
\bottomrule
\end{tabular}
}
}
\end{table}

Experiment~3 is summarized in Figure~\ref{fig:cutest3} and
Tables~\ref{tab:outcomes}--\ref{tab:samples}. The performance profiles
are averaged over five independent runs, while both tables report
medians over these runs. Table~\ref{tab:outcomes} complements the performance profiles with the fractions of problems solved or unsolved within the budget, where no run terminated because of a numerical breakdown. At low noise ($\sigma^2=10^{-8}$), adaptive sampling with all four Hessian constructions and fixed sampling with Id and SR1 converge on nearly all problems. This supports the algorithmic design and shows that fixed sampling can match adaptive sampling under sufficiently small noise. TR-IP-SSQP-EstH and TR-IP-SSQP-AveH are among the best-performing configurations at small to moderate noise levels but fall behind TR-IP-SSQP-Id at $\sigma^2=10^{-1}$, consistent with Experiment 1. As noise increases, both sampling strategies degrade. At the two highest noise levels, TR-IP-SSQP with Id, EstH, and AveH solves more problems than its fixed-sampling counterpart (e.g., $59.46\%$ versus $32.43\%$ for Id at $\sigma^2=10^{-1}$). The exception is SR1, underscoring the difficulty of stabilizing quasi-Newton updates under stochastic perturbations. LS-SSQP, which also uses adaptive sampling, is outperformed by TR-IP-SSQP with Id, EstH, and AveH at every noise level.

{\red Table~\ref{tab:samples} reports sample sizes for TR-IP-SSQP-Id. At the two highest noise levels, the median gradient
and objective batch sizes both equal the cap. Consequently, these experiments assess the method under capped sampling, rather than in an adaptive sampling regime. At the lowest noise level, both batch sizes vary across iterations.
At the next noise level, gradient sampling remains largely adaptive,
while objective sampling is frequently capped. Although the reduced sampling prefactors and truncation need not satisfy the theoretical oracle conditions, TR-IP-SSQP-Id still solves the majority of the problems at every noise level (Table~\ref{tab:outcomes}).}

\begin{figure}
  \centering
  \subfigure[$\sigma^2=10^{-8}$]{\performanceprofile{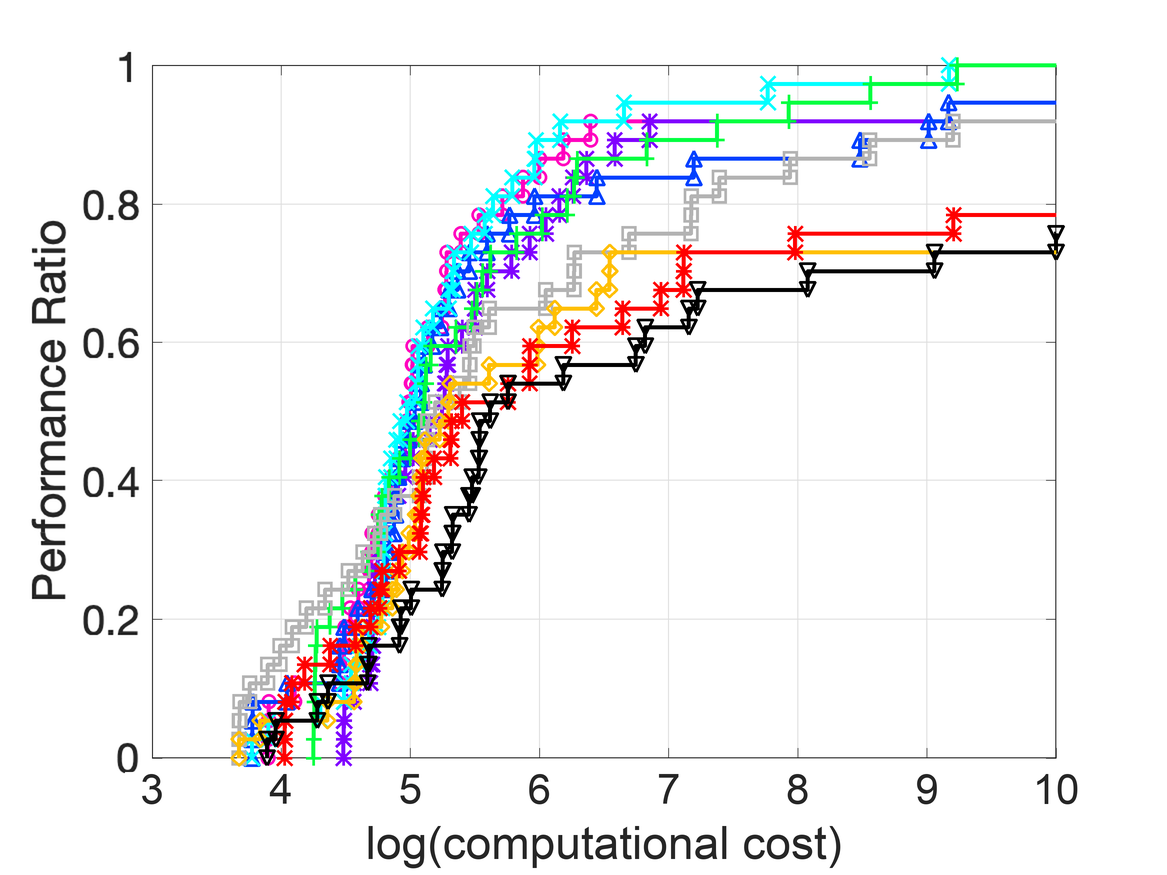}}\quad
  \subfigure[$\sigma^2=10^{-4}$]{\performanceprofile{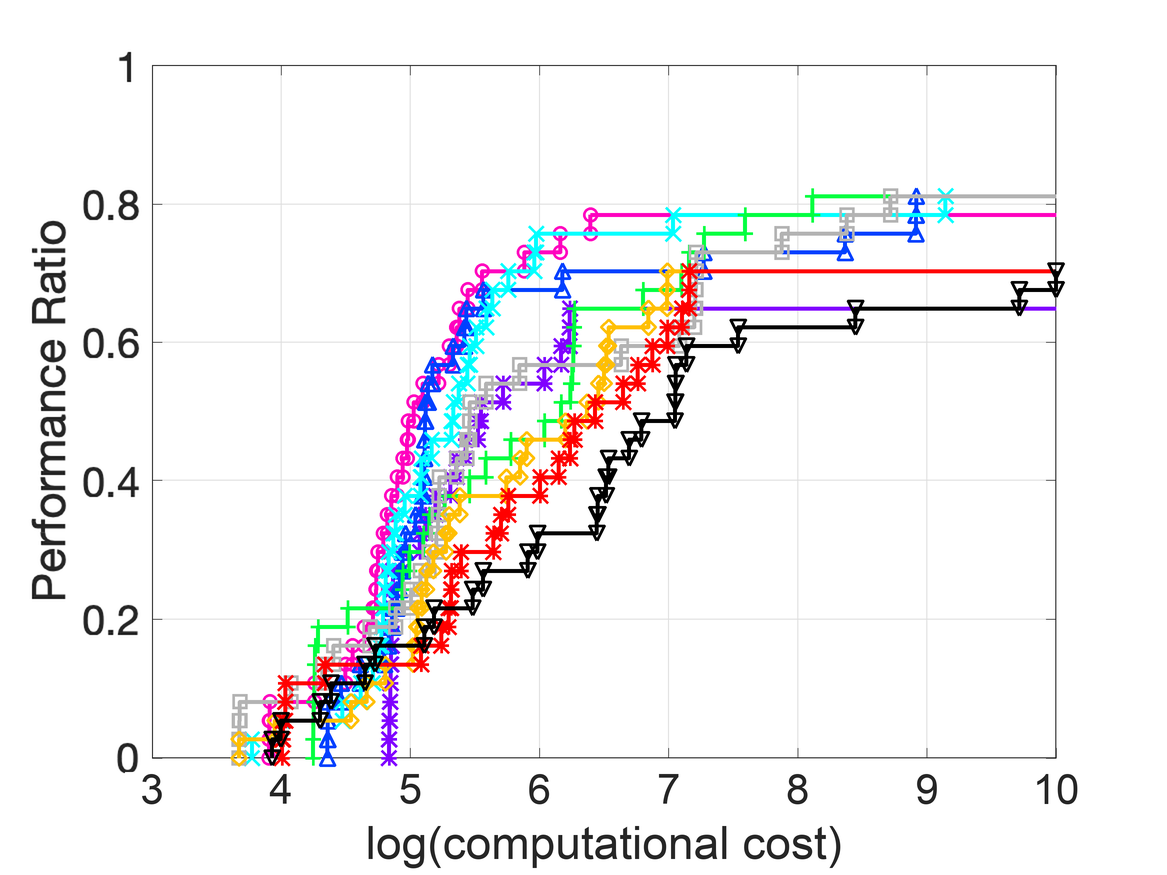}}\par
  \subfigure[$\sigma^2=10^{-2}$]{\performanceprofile{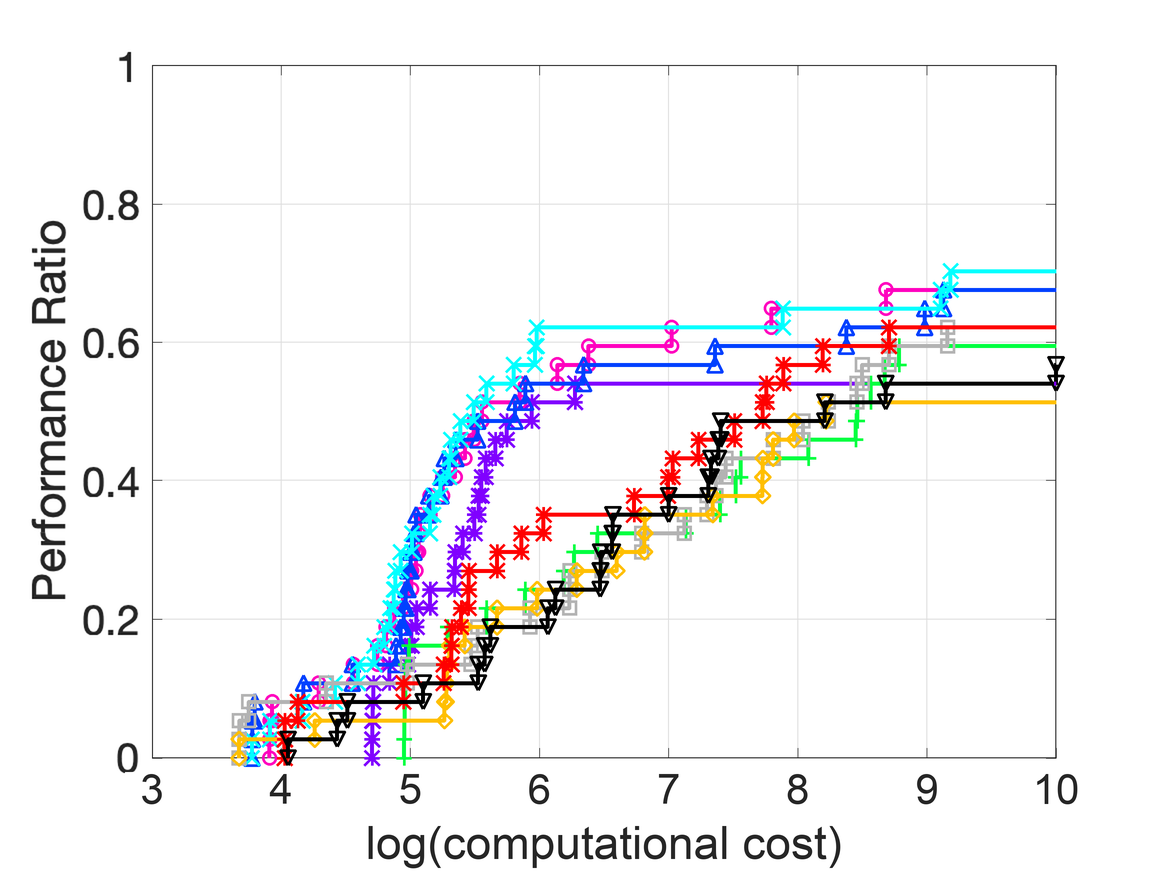}}\quad
  \subfigure[$\sigma^2=10^{-1}$]{\performanceprofile{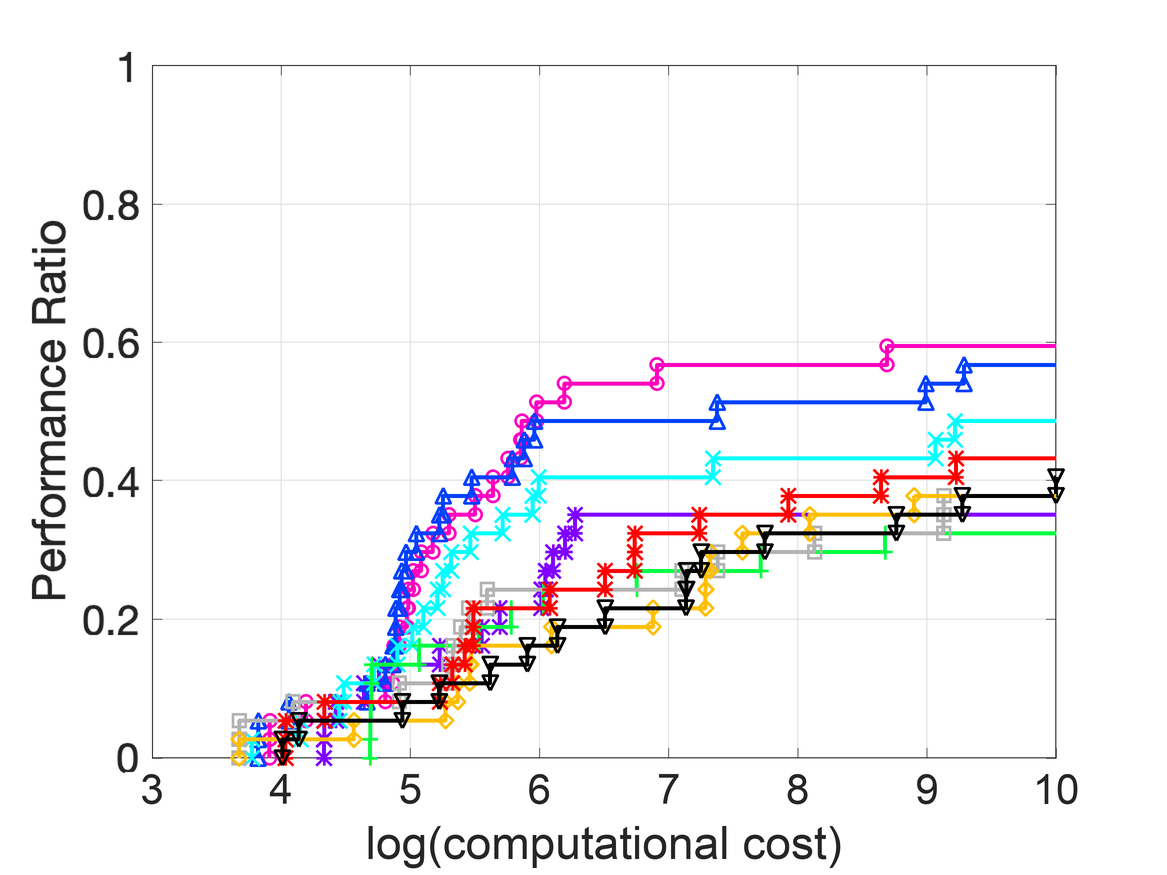}}\par
  \includegraphics[width=\textwidth]{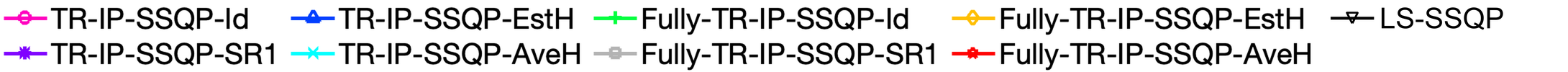}\par
  \caption{Performance profiles of TR-IP-SSQP, Fully-TR-IP-SSQP, and LS-SSQP when $\beta_j=0.9999^j$. Each trajectory represents one algorithm configuration.}
  \label{fig:cutest3}
\end{figure}

\subsection{Constrained logistic regression}\label{sec:5.3}

We consider problems of the form
\begin{equation*}
\min_{x\in\mathbb R^d}\; f(x)=\frac1N\sum_{i=1}^N\log\big(1+e^{-y_i(z_i^Tx)}\big)
\quad\text{s.t.}\quad Ax=b,\;\;\|x\|_2^2\le c^2,
\end{equation*}
where $\{(z_i,y_i)\}_{i=1}^N$ are $N$ samples with $z_i\in\mathbb R^d$ and label $y_i\in\{-1,+1\}$. We write the norm constraint in squared form so that it is differentiable everywhere. The parameters $A\in\mathbb R^{5\times d}$ and $b\in\mathbb R^5$ are generated with entries following a standard normal distribution, and we set $c=1+\omega^2$ with $\omega$ following a standard normal distribution. {\red When $A$ has full row rank, $Ax=b$ is solvable and every solution has the form $x=A^\dagger b+z$ with $z\in\ker(A)$, so that $\|x\|^2=\|A^\dagger b\|^2+\|z\|^2$. Hence the constraints admit a strictly feasible point if and only if $\|A^\dagger b\|<c$. We therefore keep a draw only if $\sigma_{\min}(A)\ge10^{-6}$ and $\|A^\dagger b\|\le c-0.1$, and redraw otherwise.}

\begin{figure}
\centering
\subfigure[$\beta_j=0.9999^j$]{\performanceprofile{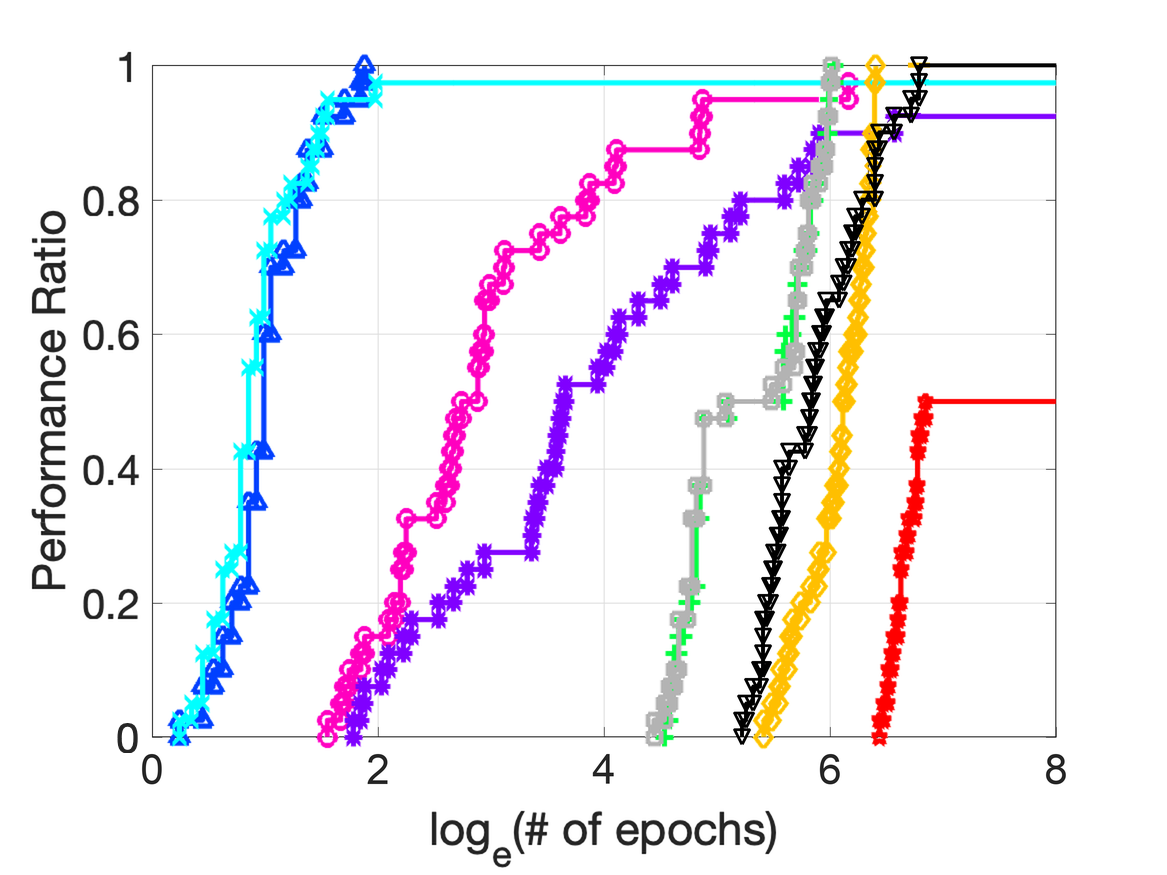}}\quad
\subfigure[$\beta_j = j^{-0.1}$]{\performanceprofile{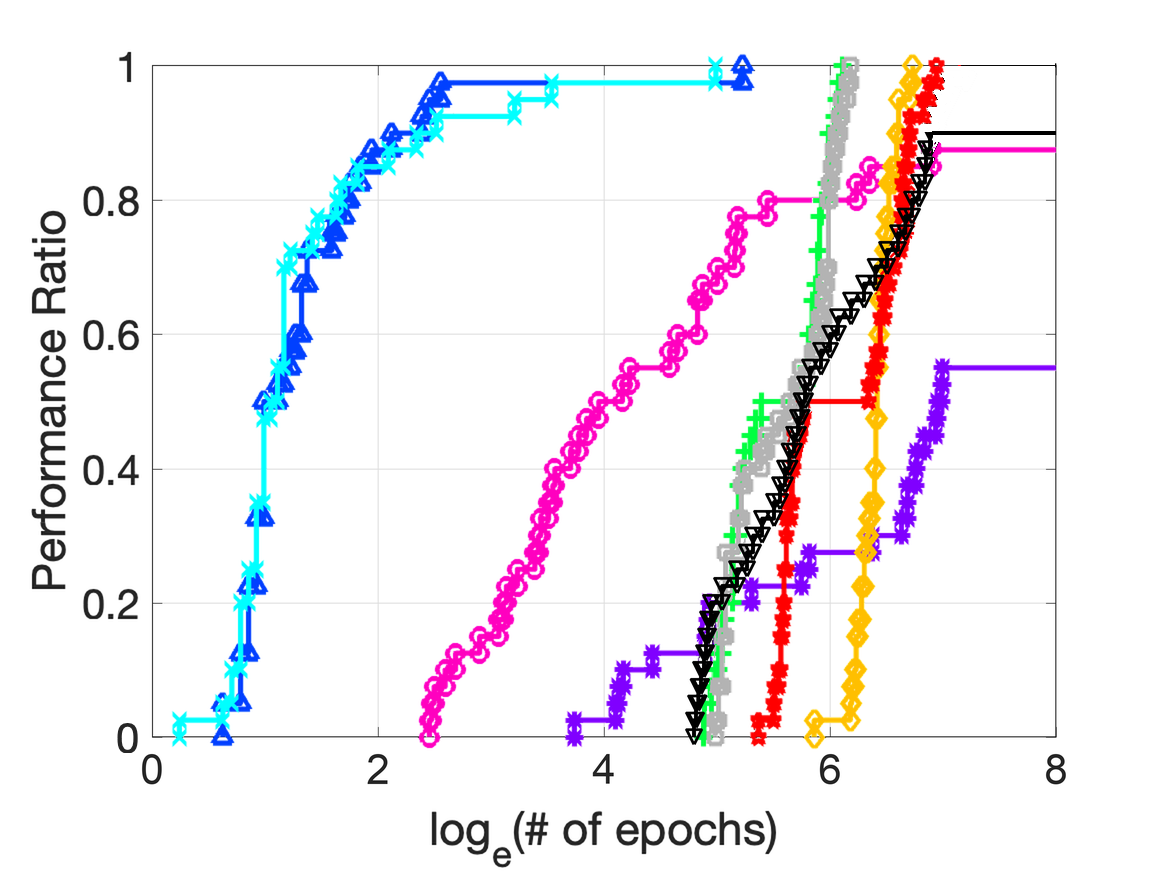}}\par
\includegraphics[width=\textwidth]{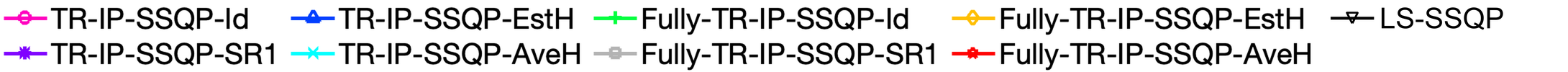}\par
\caption{Performance profiles of TR-IP-SSQP, Fully-TR-IP-SSQP, and LS-SSQP over 40 constrained logistic regression problem instances. Each trajectory represents one method.}
\label{fig:logit}  
\end{figure}

We experiment with four datasets: \texttt{German Credit Data, Mushroom} from the UCI dataset \cite{kelly2023uci}, and \texttt{normal, exponential}, which are synthetic. For \texttt{normal} and \texttt{exponential} datasets, we set $d=15, N=60,000$, with 30,000 samples for each class. In the \texttt{normal} dataset, if $y_i=1$, then each entry of $\boldsymbol{z}_i$ is generated following $\N(0,1)$; if $y_i=-1$, then each entry of $\boldsymbol{z}_i$ is generated following $\N(5,1)$. In the \texttt{exponential} dataset, if $y_i=1$, then each entry of $\boldsymbol{z}_i$ is generated following $\exp(1)$; if $y_i=-1$, then each entry of $\boldsymbol{z}_i$ is generated following $5+\exp(1)$. In each iteration, we subsample from the $N$ samples following \eqref{eq:5.1}  to estimate objective values and gradients, with $\sigma_g^2=\sigma_f^2=10^{-8}$ as empirical sampling scales. For each dataset, 10 independent combinations of $\{A,\boldsymbol{b},c\}$ are generated, resulting in 40 problem instances. For each instance, the entries of $\bx_0$ are drawn from a standard normal distribution, and those of $\bs_0$ are generated by taking the absolute values of standard normal samples. The computational budget is set to 1000 epochs, {\red where one epoch corresponds to $N$ individual data-point oracle evaluations. The budget includes evaluations used for the search gradient, both objective-value estimates, validation gradients, and stochastic Hessians when used.} The stopping condition is $\text{Relative KKT Residual}\leq 10^{-4}$ or exhaustion of the computational budget. We compare TR-IP-SSQP and Fully-TR-IP-SSQP, with all four Hessian constructions, and LS-SSQP, for the barrier sequences $\beta_j=0.9999^{j}$ and $\beta_j=j^{-0.1}$.

We report in Figure~\ref{fig:logit} the performance profiles averaged over five independent runs. The results indicate that incorporating curvature information substantially improves efficiency in this experiment. In particular, TR-IP-SSQP equipped with estimated or averaged Hessian information (EstH and AveH) consistently dominates the identity (Id) and SR1 variants, achieving high performance ratios with significantly fewer epochs. Unlike the CUTEst problems, here the Hessian estimates come from the data rather than from artificial noise, and the logistic loss has a positive semidefinite Hessian, so the curvature information is reliable.
With Id, TR-IP-SSQP solves many instances with fewer epochs, but Fully-TR-IP-SSQP attains a higher final solved fraction. LS-SSQP is faster than Fully-TR-IP-SSQP with EstH and AveH in Figure~\ref{fig:logit}(a). In contrast, as observed in Section~\ref{subsec:CUTEst}, the SR1 update degrades TR-IP-SSQP performance under both sequences $\{\beta_j\}$.

\vspace{-0.2cm}
\section{Conclusion}\label{sec:6}

We proposed a trust-region interior-point stochastic sequential quadratic programming (TR-IP-SSQP) method for finding first-order stationary points of optimization problems with a stochastic objective and deterministic nonlinear equality and inequality constraints. At each iteration, the method constructs stochastic oracles that satisfy adaptive accuracy conditions with a fixed probability. {\red The barrier parameter is taken from a prescribed decreasing sequence and is reduced once a stochastic stationarity test is passed. Under standard assumptions, we showed that, almost surely, the KKT residual vanishes along the iterations at which the barrier parameter is reduced, and that every limit point of these iterates is a KKT point.} We demonstrated the practical performance of the method on a subset of problems from the CUTEst test set and on constrained logistic regression problems.

{\red Establishing iteration and sample complexity bounds is an important direction for future work, as our current analysis is asymptotic. Within a fixed barrier stage, we expect such bounds to follow from the frameworks of \cite{Blanchet2019Convergence,Cao2024First,Jin2023Sample} under light-tailed noise, with constants that depend on the barrier parameter. Obtaining such bounds for the whole method under our oracle conditions requires addressing three challenges. First, the constants in Lemmas \ref{lemma:bound} and \ref{lemma: mu_stabilize} grow as $\beta_j$ decreases, requiring a joint choice of the sequences $\{\beta_j\}$ and $\{\epsilon_j\}$. Second, although each stage terminates almost surely, our analysis does not bound its length. Third, our oracles allow heavy-tailed noise, while the existing stopping-time complexity analyses assume lighter tails. Addressing these challenges remains an open research direction.}

\vspace{-0.2cm}
\section*{AI Declarations}

The methodology, theoretical results, and proofs were developed entirely by the authors before the recent substantial advances in AI mathematical reasoning. All writing was done by the authors. OpenAI's GPT 6.1 Sol is used to assist with writing code for the simulations and proofread during revision. The authors take full responsibility for all content. 

\bibliographystyle{siamplain}
\bibliography{ref}

\appendix

\section{Proofs of Section \ref{sec:4.1}}

\subsection{Proof of Lemma \ref{lemma:diff_ared_pred_wo_corr_step}}\label{proof_of_lemma:diff_ared_pred_wo_corr_step}

\begin{proof}
In this proof, we use the formula of $\text{Pred}_k$ defined in \eqref{def:Pred_k2}, which is written in terms of the components $\Delta\bx_k$ and $\Delta\bs_k$. Combining \eqref{merit_func} and \eqref{def:Pred_k2}, we have
\begin{align*}
\L_{\barmu_{k},\theta_k}^{s_k} & -\L_{\barmu_{k},\theta_k}^{k}-\text{Pred}_k = f_{s_k}-f_k-\barg_k^T\Delta\bx_k-\frac{1}{2}\Delta\bx_k^T\barH_k\Delta\bx_k \\
& \quad - \theta_k \left\{\sum_{i=1}^{n}\ln(\bs_{s_k}^{(i)})-\sum_{i=1}^{n}\ln(\bs_k^{(i)})-\boldsymbol{e}^TS_k^{-1}\Delta\bs_k + \frac{1}{2}\Delta\bs_k^TS_k^{-2}\Delta\bs_k\right\} \\
& \quad  -\barmu_{k}\left\{\|(c_k+G_k\Delta\bx_k ; h_k +\bs_k + J_k\Delta\bx_k+\Delta\bs_k)\| - \|(c_{s_k} ; h_{s_k}+\bs_{s_k})\| \right\}.
\end{align*} 
Noticing that $\boldsymbol{e}^TS_k^{-1}\Delta\bs_k = \sum_{i=1}^n \frac{\Delta\bs_k^{(i)}}{\bs_k^{(i)}} $ and $\Delta\bs_k^TS_k^{-2}\Delta\bs_k = \sum_{i=1}^n ( \frac{\Delta\bs_k^{(i)}}{\bs_k^{(i)}})^2$, we have
\begin{align*}
& \left| \L_{\barmu_{k},\theta_k}^{s_k}-\L_{\barmu_{k},\theta_k}^{k}-\text{Pred}_k \right| \leq \left| f_{s_k}-f_k-\barg_k^T\Delta\bx_k-\frac{1}{2}\Delta\bx_k^T\barH_k\Delta\bx_k \right| \\
& \qquad\qquad\qquad + \theta_k \sum_{i=1}^n \left| \ln(\bs_{s_k}^{(i)})-\ln(\bs_k^{(i)})-\frac{\Delta\bs_k^{(i)}}{\bs_k^{(i)}}+\frac{1}{2}\left(\frac{\Delta\bs_k^{(i)}}{\bs_k^{(i)}}\right)^2 \right|\\
& \qquad\qquad\qquad + \barmu_{k}\left| \|(c_k+G_k\Delta\bx_k; h_k +\bs_k + J_k\Delta\bx_k+\Delta\bs_k)\| - \|(c_{s_k};h_{s_k}+\bs_{s_k})\|  \right|. 
\end{align*} 
By the Taylor expansion of $f(\bx)$, the Lipschitz continuity of $\nabla f(\bx)$ (Assumption \ref{assump1}), and the Cauchy-Schwarz inequality, we have
\begin{equation*}
\left| f_{s_k}-f_k-\barg_k^T\Delta\bx_k -\frac{1}{2}\Delta\bx_k^T\barH_k\Delta\bx_k \right| \leq \|\nabla f_k-\barg_k\| \|\Delta\bx_k \| +\frac{1}{2}(L_{\nabla f}+\kappa_B)\|\Delta\bx_k\|^2.   
\end{equation*}
Using the same derivations and recalling that $\bs_{s_k}=\bs_k+\Delta\bs_k$, we obtain
\begin{align*}
& \left|\|(c_{s_k} ; h_{s_k}+\bs_{s_k})\|-\|(c_k+G_k\Delta\bx_k ; h_k +\bs_k + J_k\Delta\bx_k+\Delta\bs_k)\|\right| \\
& \leq \|c_{s_k} - c_k - G_k\Delta\bx_k \| + \|h_{s_k}+\bs_{s_k} - h_k - \bs_k - J_k\Delta\bx_k - \Delta\bs_k \|  \leq \frac{L_G+L_J}{2}\|\Delta\bx_k\|^2.  
\end{align*}
Furthermore, for any scalars $\sigma,\sigma'$ satisfying $\sigma>0$, $\sigma'\geq-\epsilon_s\sigma$, we denote $I_{\sigma,\sigma'} =
[\min\{\sigma,\sigma+\sigma'\}, \max\{\sigma,\sigma+\sigma'\}]$, and have
\begin{equation}\label{eq:sigma}
\left| \ln(\sigma+\sigma') - \ln(\sigma) - \frac{\sigma'}{\sigma} \right| \leq \sup_{t \in I_{\sigma,\sigma'}}\left| \frac{\sigma'}{t} - \frac{\sigma'}{\sigma}\right| \leq \frac{\sigma}{\sigma+\sigma'}\left( \frac{\sigma'}{\sigma} \right)^2 \leq \frac{1}{1-\epsilon_s}\left( \frac{\sigma'}{\sigma} \right)^2.
\end{equation}
Therefore, setting $\sigma = \bs_k^{(i)}$ and $\sigma'=\Delta \bs_k^{(i)}$ for $i=1,\dots,n$, we obtain
\begin{align*}
& \theta_k \sum_{i=1}^n \left| \ln(\bs_{s_k}^{(i)})-\ln(\bs_k^{(i)})-\frac{\Delta\bs_k^{(i)}}{\bs_k^{(i)}}+\frac{1}{2}\left(\frac{\Delta\bs_k^{(i)}}{\bs_k^{(i)}}\right)^2 \right| \\
& \leq \theta_k \sum_{i=1}^n \left| \ln(\bs_{s_k}^{(i)})-\ln(\bs_k^{(i)})-\frac{\Delta\bs_k^{(i)}}{\bs_k^{(i)}} \right| +\frac{1}{2}\theta_k\sum_{i=1}^n \left(\frac{\Delta\bs_k^{(i)}}{\bs_k^{(i)}}\right)^2\\
& \stackrel{\mathclap{\eqref{eq:sigma}}}{\leq} \theta_k \sum_{i=1}^n \left(\frac{1}{1-\epsilon_s}+\frac{1}{2}\right)\left(\frac{\Delta\bs_k^{(i)}}{\bs_k^{(i)}}\right)^2  = \theta_k \left(\frac{1}{1-\epsilon_s}+\frac{1}{2}\right) \|S_k^{-1}\Delta\bs_k\|^2.    
\end{align*}
In stage $j$, we have $\barmu_{k}\leq\hatmu_j$ (Assumption \ref{assump4}). In addition, since $\|\barH_k\|\leq\kappa_B$ (Assumption \ref{assump1}) and $\theta_k = \beta_j \leq \beta_1$, we combine the above displays and obtain
\begin{align*}
& \left|\L_{\barmu_{k}, \theta_k }^{s_k}-\L_{\barmu_{k}, \theta_k }^{k}-\text{Pred}_k\right| \\
& \leq \|\nabla f_k-\bar{g}_k\|\|\Delta\bx_k\|+\frac{L_{\nabla f}+\kappa_B+\hatmu_j (L_G+L_J)}{2}\|\Delta\bx_k\|^2+\beta_1\left(\frac{1}{1-\epsilon_s}+\frac{1}{2}\right) \|S_k^{-1}\Delta\bs_k\|^2\\
& \leq \|\nabla f_k-\bar{g}_k\|\Delta_k + \max\left\{\frac{L_{\nabla f}+\kappa_B+\hatmu_j (L_G+L_J)}{2},\beta_1\left(\frac{1}{1-\epsilon_s}+\frac{1}{2}\right) \right\} \Delta_k^2,
\end{align*}
where the last inequality follows from the trust-region constraint; that is, $\|\Delta\bx_k\|^2 + \|S_k^{-1}\Delta\bs_k\|^2 \leq \Delta_k^2$. Combining the above display with $\|\nabla f_k-\bar{g}_k\|\leq\kappa_g\Delta_k$ on the event $\A_k$, we complete the proof.
\end{proof}

\subsection{Proof of Lemma \ref{lemma:ABCC'_hold}}\label{proof_of_lemma:ABCC'_hold}
\begin{proof}
We consider the following three cases:

\noindent $\bullet$ \textbf{Reliable iteration:} When the iteration is reliable, we have
\begin{align}\label{eq14}
\L_{\barmu_{k} , \theta_k }^{k+1} - \L_{\barmu_{k}, \theta_k }^k \; & = \L_{\barmu_{k}, \theta_k }^{k+1} -\bar{\L}_{\barmu_{k}, \theta_k }^{k+1} + \bar{\L}_{\barmu_{k}, \theta_k }^{k+1}  - \bar{\L}_{\barmu_{k}, \theta_k }^k + \bar{\L}_{\barmu_{k}, \theta_k }^k - \L_{\barmu_{k}, \theta_k }^k \notag \\
& \leq \left| \L_{\barmu_{k}, \theta_k }^{k+1}-\bar{\L}_{\barmu_{k}, \theta_k }^{k+1} \right|  + \text{Ared}_k + \left| \bar{\L}_{\barmu_{k}, \theta_k }^k-\L_{\barmu_{k}, \theta_k }^k \right|  \notag \\
& \stackrel{\mathclap{(\ref{def:accuracy}a)}}{\leq}\; \left| f_{s_k}-\barf_{s_k} \right| + \eta \cdot \text{Pred}_k + \left| f_k-\barf_k \right|  \notag \\
& \stackrel{\mathclap{(\ref{def:accuracy}b)}}{\leq}\; \left|f_{s_k} - \barf_{s_k}\right| + \left|f_k - \barf_k\right| + \frac{1}{2}\eta  \text{Pred}_k -\frac{1}{2} \eta \barepsilon_k.
\end{align}
On the event $\B_k$, we have $ |f_{s_k} - \barf_{s_k}| + |f_k - \barf_k| \leq 2\kappa_{f,k}\Delta_k^2$. Since the iteration is successful, \eqref{eq:pred_1st_order} holds, and we thus have
\begin{align}\label{eq13}
\L_{\barmu_{k} ,  \theta_k  }^{k+1} - \L_{\barmu_{k},  \theta_k  }^k 
& \leq 2\kappa_{f,k}\Delta_k^2 -\frac{\kappa_{fcd}}{4\|R_k\|} \min\left\{1,\frac{\epsilon_s}{\Delta_{\max}}\right\} \eta^3\Delta_k^2 -\frac{1}{2} \eta \barepsilon_k  
\notag \\
& \leq -\frac{\kappa_{fcd}}{8\|R_k\|} \min\left\{1,\frac{\epsilon_s}{\Delta_{\max}}\right\} \eta^3\Delta_k^2 -\frac{1}{2} \eta \barepsilon_k.
\end{align}
For a reliable iteration, $\Delta_{k+1} \leq \gamma\Delta_k$ and $\barepsilon_{k+1}=\gamma\barepsilon_k$. Thus, by Lemma \ref{lemma:upper_bound_Phi}, we have
\begin{align*}
\Phi_{\barmu_{k+1}, \theta_{k+1} }^{k+1} & - \Phi_{\barmu_{k}, \theta_{k} }^{k}  \leq \Phi_{\barmu_{k}, \theta_{k} }^{k+1} - \Phi_{\barmu_{k}, \theta_{k} }^{k} \\
& =  \frac{1}{\barmu_{k}}[\L_{\barmu_{k},\theta_{k} }^{k+1} -  \L_{\barmu_{k},\theta_{k} }^{k} + a_{\beta_j}(\Delta_{k+1}^2 - \Delta_k^2) + b(\barepsilon_{k+1} - \barepsilon_k) ] \\
& \leq  \frac{1}{\barmu_{k}}[\L_{\barmu_{k},\theta_{k} }^{k+1} -  \L_{\barmu_{k},\theta_{k} }^{k} + a_{\beta_j}(\gamma^2-1)\Delta_k^2 + b(\gamma - 1) \barepsilon_k ] \\
& \stackrel{\eqref{eq13}}{\leq }\frac{1}{\barmu_{k}} \left[ -\frac{\kappa_{fcd}}{8\|R_k\|} \min\left\{1,\frac{\epsilon_s}{\Delta_{\max}}\right\} \eta^3\Delta_k^2 -\frac{1}{2} \eta \barepsilon_k + a_{\beta_j}(\gamma^2-1)\Delta_k^2 + b(\gamma - 1) \barepsilon_k \right] \\
& \leq \frac{1}{\barmu_{k}} \left[ -\frac{\kappa_{fcd}}{8\kappa_{R,\beta_j}} \min\left\{1,\frac{\epsilon_s}{\Delta_{\max}}\right\} \eta^3\Delta_k^2 -\frac{1}{4} \eta \barepsilon_k + a_{\beta_j}(\gamma^2-1)\Delta_k^2 \right].
\end{align*}
$\bullet$ \textbf{Unreliable iteration:} In this case, $\Delta_{k+1}\leq \gamma\Delta_k$ and $\barepsilon_{k+1}=\barepsilon_k/\gamma$; thus
\begin{align*}
\Phi_{\barmu_{k+1}, \theta_{k+1} }^{k+1} - \Phi_{\barmu_{k}, \theta_{k} }^{k}
& \leq  \frac{1}{\barmu_{k}}\left[\L_{\barmu_{k},\theta_k}^{k+1} -  \L_{\barmu_{k},\theta_k}^{k} + a_{\beta_j}(\gamma^2-1)\Delta_k^2 + b\left(\frac{1}{\gamma}- 1\right) \barepsilon_k \right] \\
& \stackrel{\mathclap{\eqref{eq12}}}{\leq} \frac{1}{\barmu_{k}} \left[ -\frac{3\kappa_{fcd}}{8\|R_k\|} \min\left\{1,\frac{\epsilon_s}{\Delta_{\max}}\right\} \eta^3\Delta_k^2  + a_{\beta_j}(\gamma^2-1)\Delta_k^2 + b\left(\frac{1}{\gamma}- 1\right) \barepsilon_k \right] \\
& \leq \frac{1}{\barmu_{k}} \left[ -\frac{3\kappa_{fcd}}{8\kappa_{R,\beta_j}} \min\left\{1,\frac{\epsilon_s}{\Delta_{\max}}\right\} \eta^3\Delta_k^2 + a_{\beta_j}(\gamma^2-1)\Delta_k^2 + b\left(\frac{1}{\gamma}- 1\right) \barepsilon_k \right].
\end{align*}
$\bullet$ \textbf{Unsuccessful iteration:} In this case, $\bx_{k+1}=\bx_k,\bs_{k+1}=\bs_k$, $\Delta_{k+1}=\Delta_k/\gamma$, and $\barepsilon_{k+1}=\barepsilon_k/\gamma$. Therefore, we have $\L_{\barmu_{k},\theta_k}^{k+1}-\L_{\barmu_{k},\theta_k}^k=0$ and
\begin{align*}
\Phi_{\barmu_{k+1}, \theta_{k+1} }^{k+1} - \Phi_{\barmu_{k}, \theta_{k} }^{k}
& \leq  \frac{1}{\barmu_{k}}\left[\L_{\barmu_{k},\theta_k}^{k+1} -  \L_{\barmu_{k},\theta_k}^{k} + a_{\beta_j}\left(\frac{1}{\gamma^2}- 1\right)\Delta_k^2 + b\left(\frac{1}{\gamma}- 1\right) \barepsilon_k \right] \\
& = \frac{1}{\barmu_{k}} \left[ a_{\beta_j}\left(\frac{1}{\gamma^2}- 1\right)\Delta_k^2 + b\left(\frac{1}{\gamma}- 1\right) \barepsilon_k \right].
\end{align*}
By the definitions of $a_{\beta_j}$ and $b$, we have $ -\frac{\kappa_{fcd}}{8\kappa_{R,\beta_j}} \min\left\{1,\frac{\epsilon_s}{\Delta_{\max}}\right\} \eta^3\Delta_k^2 + a_{\beta_j}(\gamma^2-1)\Delta_k^2 = a_{\beta_j}\left(\frac{1}{\gamma^2}- 1\right)\Delta_k^2$ and $ -\frac{1}{4}\eta\barepsilon_k\leq b\left(\frac{1}{\gamma}-1\right)\barepsilon_k$. Therefore, combining the above three cases, we have
\begin{equation*}
\Phi_{\barmu_{k+1}, \theta_{k+1} }^{k+1} - \Phi_{\barmu_{k}, \theta_{k} }^{k} \leq   \frac{1}{\barmu_{k}} \left[ a_{\beta_j}\left(\frac{1}{\gamma^2}- 1\right)\Delta_k^2 + b\left(\frac{1}{\gamma}- 1\right) \barepsilon_k \right] ,
\end{equation*}
and we complete the proof.
\end{proof}

\subsection{Proof of Lemma \ref{lemma:Phi_(ABCC')^c}}\label{proof_of_lemma:Phi_(ABCC')^c}
\begin{proof}
In the proof, we consider the following three cases.

\noindent $\bullet$ \textbf{Reliable iteration:} Noticing that the derivation of \eqref{eq14} does not rely on $\B_k$, we still have \eqref{eq14} in this case:
\begin{align*}
\Phi_{\barmu_{k+1},\theta_{k+1}}^{k+1} - \Phi_{\barmu_{k},\theta_{k}}^{k} & \leq \frac{1}{\barmu_k}\left[\L_{\barmu_{k},\theta_{k}}^{k+1}-\L_{\barmu_{k},\theta_{k}}^k + a_{\beta_j}(\Delta_{k+1}^2-\Delta_k^2) + b( \barepsilon_{k+1} -\barepsilon_k)\right] \notag \\
& \leq \frac{1}{\barmu_k}\left[e_k + e_{s_k}  + \frac{1}{2}\eta  \text{Pred}_k -\frac{1}{2} \eta \barepsilon_k + a_{\beta_j}(\gamma^2-1)\Delta_k^2 + b(\gamma-1) \barepsilon_k \right].
\end{align*}
Since the iteration is reliable, it is successful. Therefore, we have \eqref{eq:pred_1st_order} and
\begin{equation*}
\Phi_{\barmu_{k+1},\theta_{k+1}}^{k+1} - \Phi_{\barmu_{k},\theta_{k}}^{k} 
\leq \frac{1}{\barmu_k}\left[e_k + e_{s_k} - \frac{\kappa_{fcd}}{4\kappa_{R,\beta_j}} \min\left\{1,\frac{\epsilon_s}{\Delta_{\max}}\right\}\eta^3 \Delta_k^2 - \frac{1}{4}\eta \barepsilon_k + a_{\beta_j}(\gamma^2-1)\Delta_k^2\right],
\end{equation*}
which, by the same argument as at the end of the proof of Lemma \ref{lemma:ABCC'_hold}, is bounded by the right-hand side of \eqref{eq:lemma_phi_reduction_not_all_acc}.

\noindent $\bullet$ \textbf{Unreliable iteration:} Since the iteration is successful, we have $\text{Ared}_k/\text{Pred}_k\geq\eta$, which implies $\L_{\barmu_k,\theta_k}^{k+1} - \L_{\barmu_k,\theta_k}^{k} \leq e_k+e_{s_k}+\eta\,\text{Pred}_k$. By \eqref{eq:pred_1st_order} and Lemma \ref{lemma:bound}, we have
\begin{equation*}
\text{Pred}_k \leq -\frac{\kappa_{fcd}\eta^2}{2\|R_k\|} \min\left\{1,\frac{\epsilon_s}{\Delta_{\max}}\right\} \Delta_k^2 \leq -\frac{\kappa_{fcd}\eta^2}{2\kappa_{R,\beta_j}} \min\left\{1,\frac{\epsilon_s}{\Delta_{\max}}\right\}\Delta_k^2.
\end{equation*}
Therefore, we have
\begin{equation*}
\L_{\barmu_k,\theta_k}^{k+1} - \L_{\barmu_k,\theta_k}^{k} \leq e_k+e_{s_k}
-\frac{\kappa_{fcd}\eta^3}{2\kappa_{R,\beta_j}} \min\left\{1,\frac{\epsilon_s}{\Delta_{\max}}\right\}\Delta_k^2.
\end{equation*}
As the iteration is unreliable, we have $\Delta_{k+1}\leq\gamma\Delta_k$ and $\barepsilon_{k+1}=\barepsilon_k/\gamma$. Thus,
\begin{align*}
\Phi_{\barmu_{k+1},\theta_{k+1}}^{k+1} & - \Phi_{\barmu_{k},\theta_{k}}^{k} \leq \frac{1}{\barmu_k}
\Big[\L_{\barmu_k,\theta_{k}}^{k+1}-\L_{\barmu_k,\theta_{k}}^{k}+a_{\beta_j}(\Delta_{k+1}^2-\Delta_k^2)+b(\barepsilon_{k+1}-\barepsilon_k)\Big]\\
&\leq \frac{1}{\barmu_k} \left[ e_k+e_{s_k} -\frac{\kappa_{fcd}\eta^3}{2\kappa_{R,\beta_j}}\min\left\{1,\frac{\epsilon_s}{\Delta_{\max}}\right\}\Delta_k^2+a_{\beta_j}(\gamma^2-1)\Delta_k^2 +b\left(\frac{1}{\gamma}-1\right)\barepsilon_k\right].
\end{align*}
By the definition of $a_{\beta_j}$, we find $\frac{\kappa_{fcd}\eta^3}{2\kappa_{R,\beta_j}} \min\left\{1,\frac{\epsilon_s}{\Delta_{\max}}\right\}=4a_{\beta_j}(\gamma^2-\gamma^{-2})$. Substituting this identity gives
\begin{align*}
\Phi_{\barmu_{k+1},\theta_{k+1}}^{k+1} - \Phi_{\barmu_{k},\theta_{k}}^{k} &\leq \frac{1}{\barmu_k} \Big[e_k+e_{s_k}- 4a_{\beta_j}(\gamma^2-\gamma^{-2})\Delta_k^2 + a_{\beta_j}(\gamma^{2}-1)\Delta_k^2+b(\gamma^{-1}-1)\barepsilon_k \Big]\\
&\leq \frac{1}{\barmu_k}\left[e_k+e_{s_k} + a_{\beta_j} \left(\frac{1}{\gamma^2}-1\right)\Delta_k^2 + b\left(\frac{1}{\gamma}-1\right)\barepsilon_k\right] .
\end{align*}
$\bullet$ \textbf{Unsuccessful iteration:} In this case, $(\bx_{k+1},\bs_{k+1})=(\bx_k,\bs_k)$, $\Delta_{k+1}=\Delta_k/\gamma$, and $\barepsilon_{k+1}=\barepsilon_k/\gamma$. Thus, $\L_{\barmu_k,\theta_{k}}^{k+1}-\L_{\barmu_k,\theta_{k}}^{k}=0$ and we have
\begin{equation*}
\Phi_{\barmu_{k+1},\theta_{k+1}}^{k+1} - \Phi_{\barmu_{k},\theta_{k}}^{k} \leq
\frac{1}{\barmu_{k}} \left[a_{\beta_j}\left(\frac{1}{\gamma^2}-1\right)\Delta_k^2
+b\left(\frac{1}{\gamma}-1\right)\barepsilon_k \right].
\end{equation*}
Combining the above three cases, we complete the proof.
\end{proof}

\section{Proof of Lemma \ref{lemma: mu_stabilize}}\label{proof_of_lemma: mu_stabilize}

\begin{proof} 
Fix a stage $j$ and an iteration $k$ in stage $j$. It suffices to show that \eqref{eq:threshold_Predk} is satisfied if $\barmu_k$ is larger than a deterministic threshold depending only on $j$. Since $\|\vartheta_k+A_k\bd_k\|-\|\vartheta_k\|=-\bargamma_k\|\vartheta_k\|$, we have
\begin{align*}
\text{Pred}_k & = \barpsi_k^T\bd_k + \frac{1}{2}\bd_k^T\bar{W}_k\bd_k+\barmu_k(\|\vartheta_k+A_k\bd_k\|-\|\vartheta_k\|) \notag\\
& = (\barpsi_k+\bargamma_k\bar{W}_k\bv_k)^T\bt_k + \frac{1}{2}\bt_k^T\bar{W}_k\bt_k+\bargamma_k\barpsi_k^T\bv_k +\frac{1}{2}\bargamma_k^2\bv_k^T\bar{W}_k\bv_k-\barmu_k\bargamma_k\|\vartheta_k\| \notag\\
& \leq -\frac{\kappa_{fcd}}{2}\|Z_k^T(\barpsi_k+\bargamma_k\bar{W}_k\bv_k)\| \min\left\{\frac{\hat{\Delta}_k}{\|R_k\|}, \frac{\epsilon_s - \|S_k^{-1}\bw_k^s\|}{\|R_k\|},\frac{\|Z_k^T(\barpsi_k+\bargamma_k\bar{W}_k\bv_k)\|}{\|Z_k^T\bar{W}_kZ_k\|}\right\} \\
& \quad +\bargamma_k\|\barpsi_k\| \|\bv_k\| +\frac{1}{2}\bargamma_k\|\bar{W}_k\|\|\bv_k\|^2 -\barmu_k\bargamma_k\|\vartheta_k\|,
\end{align*}
\noindent where we have used \eqref{eq:cauchy1} and $\bargamma_k\leq 1$.
Since $\|Z_k^T(\barpsi_k+\bargamma_k\bar{W}_k\bv_k)\| \geq \|Z_k^T\barpsi_k\| - \bargamma_k\|\bar{W}_k\|\|\bv_k\|$, $\|S_k^{-1}\bw_k^s\|\leq \|R_k\bw_k\|\leq\bargamma_k\|R_k\|\|\bv_k\|$, and $\|Z_k\|\leq 1$, we have
\begin{align*}
& \|Z_k^T( \barpsi_k +\bargamma_k\bar{W}_k\bv_k)\|\min\left\{\frac{\hat{\Delta}_k}{\|R_k\|},\frac{\epsilon_s - \|S_k^{-1}\bw_k^s\|}{\|R_k\|},\frac{\|Z_k^T(\barpsi_k+\bargamma_k\bar{W}_k\bv_k)\|}{\|Z_k^T\bar{W}_kZ_k\|}\right\} \notag \\
& \geq \|Z_k^T\barpsi_k\|\min\left\{\frac{\hat{\Delta}_k}{\|R_k\|},\frac{\epsilon_s}{\|R_k\|}, \frac{\|Z_k^T\barpsi_k\|}{\|\bar{W}_k\|}\right\}-2\bargamma_k\|Z_k^T\barpsi_k\|\|\bv_k\|-\bargamma_k\|\bar{W}_k\|\|\bv_k\|\hat{\Delta}_k\\
& \geq \|Z_k^T\barpsi_k\|\min\left\{\frac{\Delta_k}{\|R_k\|}-\bargamma_k\|\bv_k\|,\frac{\epsilon_s}{\|R_k\|}, \frac{\|Z_k^T\barpsi_k\|}{\|\bar{W}_k\|}\right\} - 2\bargamma_k\|Z_k^T\barpsi_k\|\|\bv_k\|-\bargamma_k\|\bar{W}_k\|\|\bv_k\|\hat{\Delta}_k\\
& \geq \|Z_k^T\barpsi_k\|\min\left\{\frac{\Delta_k}{\|R_k\|},\frac{\epsilon_s}{\|R_k\|}, \frac{\|Z_k^T\barpsi_k\|}{\|\bar{W}_k\|}\right\} - 3\bargamma_k\|Z_k^T\barpsi_k\|\|\bv_k\|-\bargamma_k\|\bar{W}_k\|\|\bv_k\|\hat{\Delta}_k,
\end{align*}
\noindent where the first inequality uses that the left-hand side is nondecreasing in $\|Z_k^T(\barpsi_k+\bargamma_k\bar{W}_k\bv_k)\|$ (if $\|Z_k^T\barpsi_k\|<\bargamma_k\|\bar{W}_k\|\|\bv_k\|$, the right-hand side is negative), and the second inequality uses $\Delta_k/\|R_k\|-\bargamma_k\|\bv_k\|\leq \hat{\Delta}_k/\|R_k\|$, which follows from $\hat\Delta_k=\Delta_k-\|R_k\bw_k\|$. Using the facts that $\|\bv_k\|\leq\frac{1}{\sqrt{\kappa_{1,A}}}\|\vartheta_k\|\leq \frac{1}{\sqrt{\kappa_{1,A}}}(\kappa_c+\kappa_h+\Upsilon_s)$ (Assumption \ref{assump1}), $\|Z_k^T\barpsi_k\|\leq\|\barpsi_k\|$, $ \hat{\Delta}_k\leq\Delta_{\max}$, and $\kappa_{fcd}\leq 1$, the above two displays lead to
\begin{align*}
\text{Pred}_k & \leq -\frac{\kappa_{fcd}}{2} \|Z_k^T\barpsi_k\|\min\left\{\frac{\Delta_k}{\|R_k\|},\frac{\epsilon_s}{\|R_k\|}, \frac{\|Z_k^T\barpsi_k\|}{\|\bar{W}_k\|}\right\} + \frac{3}{2}\bargamma_k\|Z_k^T\barpsi_k\|\|\bv_k\| +\bargamma_k\|\barpsi_k\| \|\bv_k\|  \\
& \quad + \frac{1}{2}\bargamma_k\|\bar{W}_k\|\|\bv_k\|\hat{\Delta}_k +\frac{1}{2}\bargamma_k\|\bar{W}_k\|\|\bv_k\|^2 -\barmu_k\bargamma_k\|\vartheta_k\| \\
& \leq -\frac{\kappa_{fcd}}{2} \|Z_k^T\barpsi_k\|\min\left\{\frac{\Delta_k}{\|R_k\|},\frac{\epsilon_s}{\|R_k\|}, \frac{\|Z_k^T\barpsi_k\|}{\|\bar{W}_k\|}\right\} + 2.5 \bargamma_k\|\barpsi_k\|\|\bv_k\| \\
& \quad + \left(\frac{1}{2}\hat{\Delta}_k+\frac{1}{2\sqrt{\kappa_{1,A}}}\|\vartheta_k\|\right)\bargamma_k\|\bar{W}_k\|\|\bv_k\| -\barmu_k\bargamma_k\|\vartheta_k\| \\
& \leq -\frac{\kappa_{fcd}}{2}\|Z_k^T\barpsi_k\|\min\left\{\frac{\Delta_k}{\|R_k\|},\frac{\epsilon_s}{\|R_k\|},\frac{\|Z_k^T\barpsi_k\|}{\|\bar{W}_k\|}\right\}+\frac{2.5}{\sqrt{\kappa_{1,A}}}\bargamma_k\|\barpsi_k\|\|\vartheta_k\|\\
&\quad +\left(\frac{\Delta_{\max}}{2\sqrt{\kappa_{1,A}}}+\frac{\kappa_c+\kappa_h+\Upsilon_s}{2\kappa_{1,A}}\right)\bargamma_k\|\bar{W}_k\|\|\vartheta_k\| -\barmu_k\bargamma_k\|\vartheta_k\|.    
\end{align*}
\noindent Since $k$ is in stage $j$, Lemma \ref{lemma:bound} gives $\|\bar{W}_k\|\leq \kappa_{W,\beta_j}$. By Assumptions \ref{assump1} and \ref{ass:bdd_err}, $\|\barpsi_k\|\leq\|\nabla f_k\|+\|\nabla f_k-\barg_k\| + \|\theta_kS_k^{-1}\boldsymbol{e}\| \leq   M +\kappa_{\nabla f} + \sqrt{n}\,\tau^{\max}$. Therefore,
\begin{equation}\label{eq:App1}
\text{Pred}_k \leq -\frac{\kappa_{fcd}}{2}\|Z_k^T\barpsi_k\|\min\left\{\frac{\Delta_k}{\|R_k\|},\frac{\epsilon_s}{\|R_k\|},\frac{\|Z_k^T\barpsi_k\|}{\|\bar{W}_k\|}\right\}+\Upsilon_{3,j}\bargamma_k\|\vartheta_k\| -\barmu_k\bargamma_k\|\vartheta_k\|,    
\end{equation}
\noindent where $\Upsilon_{3,j}=\left(\frac{\Delta_{\max}}{2\sqrt{\kappa_{1,A}}}+\frac{\kappa_c+\kappa_h+\Upsilon_s}{2\kappa_{1,A}}\right)\kappa_{W,\beta_j}+\frac{2.5(M +\kappa_{\nabla f} + \sqrt{n}\,\tau^{\max})}{\sqrt{\kappa_{1,A}}}$. We discuss two cases.

\noindent $\bullet$ \textbf{Case 1.} $\min\{\Delta_k,\epsilon_s\}/\|R_k\|\leq \|Z_k^T\barpsi_k\|/\|\bar{W}_k\|$. In this case, the above inequality gives 
\begin{align*}
\text{Pred}_k & \leq -\frac{\kappa_{fcd}}{2}\|Z_k^T\barpsi_k\|\frac{\min\left\{\Delta_k,\epsilon_s\right\}}{\|R_k\|} - \frac{\kappa_{fcd}}{2}\|\vartheta_k\|\frac{\min\left\{\Delta_k,\epsilon_s\right\}}{\|R_k\|}\\&\quad  + \frac{\kappa_{fcd}}{2}\|\vartheta_k\|\frac{\min\left\{\Delta_k,\epsilon_s\right\}}{\|R_k\|}
+\Upsilon_{3,j}\bargamma_k\|\vartheta_k\| -\barmu_k\bargamma_k\|\vartheta_k\|\\
& \leq -\frac{\kappa_{fcd}}{2}\|\bar{Q}_k\|\frac{\min\left\{\Delta_k,\epsilon_s\right\}}{\|R_k\|} + \frac{\kappa_{fcd}}{2}\|\vartheta_k\|\frac{\min\left\{\Delta_k,\epsilon_s\right\}}{\|R_k\|}
+\Upsilon_{3,j}\bargamma_k\|\vartheta_k\| -\barmu_k\bargamma_k\|\vartheta_k\|\\
& \leq -\frac{\kappa_{fcd}}{2}\|\bar{Q}_k\|\min\left\{\frac{\Delta_k}{\|R_k\|},\frac{\epsilon_s}{\|R_k\|},\frac{\|\bar{Q}_k\|}{\|\bar{W}_k\|}\right\} + \frac{\|\vartheta_k\|}{2}\frac{\min\left\{\Delta_k,\epsilon_s\right\}}{\|R_k\|}\\
& \quad +\Upsilon_{3,j}\bargamma_k\|\vartheta_k\| -\barmu_k\bargamma_k\|\vartheta_k\|,    
\end{align*}
since $\|Z_k^T\barpsi_k\| + \|\vartheta_k\| \geq \|\bar{Q}_k\|$ and $\kappa_{fcd}\leq 1$. Therefore, \eqref{eq:threshold_Predk} holds provided $\barmu_k\bargamma_k\|\vartheta_k\|\geq 0.5\|\vartheta_k\| \min\left\{\Delta_k,\epsilon_s\right\}/\|R_k\|
+\Upsilon_{3,j}\bargamma_k\|\vartheta_k\|$, equivalently, $\barmu_k\geq \frac{\min\left\{\Delta_k,\epsilon_s\right\}}{2\bargamma_k\|R_k\|} +\Upsilon_{3,j}$. 

By the definition of $\bargamma_k$ in \eqref{eq:Sto_gamma_k}, if $\bargamma_k=1$, it suffices that $\barmu_k \geq0.5\min\{\Delta_{\max},\epsilon_s\} + \Upsilon_{3,j}$, as $\|R_k\|\geq1$. While if  $\bargamma_k = \min\{\frac{\zeta\epsilon_s}{\|S_k^{-1}\bv_k^s\|},
\frac{\zeta\Delta_k}{\|R_k\bv_k\|}\}$, since $\|S_k^{-1}\bv_k^s\| \leq \|R_k\bv_k\|\leq\|R_k\|\|\bv_k\|$, we have $\bargamma_k \geq \zeta \frac{\min\{\epsilon_s,\Delta_k\}}{\|R_k\|\|\bv_k\|}$. Thus, it suffices that $\barmu_k \geq \frac{\|\bv_k\|}{2\zeta} +\Upsilon_{3,j}$. Noting that $\|\bv_k\|\leq \frac{1}{\sqrt{\kappa_{1,A}}}\|\vartheta_k\|\leq \frac{\kappa_c+\kappa_h+\Upsilon_s}{\sqrt{\kappa_{1,A}}}$, we further strengthen the condition to $\barmu_k \geq \frac{\kappa_c+\kappa_h+\Upsilon_s}{2\zeta\sqrt{\kappa_{1,A}}} +\Upsilon_{3,j}$.~Combining the above two cases, it is sufficient to have $\barmu_k \geq 0.5\min\{\Delta_{\max},\epsilon_s\}+\frac{\kappa_c+\kappa_h+\Upsilon_s}{2\zeta\sqrt{\kappa_{1,A}}}+\Upsilon_{3,j}$.

\noindent $\bullet$ \textbf{Case 2.} $\min\{\Delta_k,\epsilon_s\}/\|R_k\| > \|Z_k^T\barpsi_k\|/\|\bar{W}_k\|$. In this case, \eqref{eq:App1} leads to
\begin{align*}
\text{Pred}_k & \leq -\frac{\kappa_{fcd}}{2}\frac{\|Z_k^T\barpsi_k\|^2}{\|\bar{W}_k\|}+\Upsilon_{3,j}\bargamma_k\|\vartheta_k\| -\barmu_k\bargamma_k\|\vartheta_k\|\\
& = -\frac{\kappa_{fcd}\|Z_k^T\barpsi_k\|(\|Z_k^T\barpsi_k\| + \|\vartheta_k\|)}{2\|\bar{W}_k\|} + \frac{ \kappa_{fcd}\|Z_k^T\barpsi_k\| \|\vartheta_k\|}{2\|\bar{W}_k\|}+\Upsilon_{3,j}\bargamma_k\|\vartheta_k\| -\barmu_k\bargamma_k\|\vartheta_k\|.
\end{align*}
Rearranging the terms and using $\|Z_k^T\barpsi_k\| + \|\vartheta_k\| \geq \|\bar{Q}_k\|$ and $\kappa_{fcd}\leq 1$, we have
\begin{align*}
\text{Pred}_k    
& \leq -\frac{\kappa_{fcd}}{2}\frac{\|Z_k^T\barpsi_k\| \|\bar{Q}_k\|}{\|\bar{W}_k\|} + \frac{ \|Z_k^T\barpsi_k\| \|\vartheta_k\|}{2\|\bar{W}_k\|} -\frac{\kappa_{fcd}}{2}\|\vartheta_k\|\frac{\min\{\Delta_k,\epsilon_s\}}{\|R_k\|}\\
& \quad + \frac{\kappa_{fcd}}{2}\|\vartheta_k\|\frac{\min\{\Delta_k,\epsilon_s\}}{\|R_k\|}+\Upsilon_{3,j}\bargamma_k\|\vartheta_k\| -\barmu_k\bargamma_k\|\vartheta_k\|\\
& \leq -\frac{\kappa_{fcd}}{2}\|Z_k^T\barpsi_k\|\min\left\{\frac{\Delta_k}{\|R_k\|},\frac{\epsilon_s}{\|R_k\|},\frac{\|\bar{Q}_k\|}{\|\bar{W}_k\|}\right\} -\frac{\kappa_{fcd}}{2}\|\vartheta_k\|\min\left\{\frac{\Delta_k}{\|R_k\|},\frac{\epsilon_s}{\|R_k\|},\frac{\|\bar{Q}_k\|}{\|\bar{W}_k\|}\right\}\\
& \quad  + \frac{ \|Z_k^T\barpsi_k\|\|\vartheta_k\|}{2 \|\bar{W}_k\|} + \frac{1}{2}\|\vartheta_k\|\frac{\min\{\Delta_k,\epsilon_s\}}{\|R_k\|} + \Upsilon_{3,j}\bargamma_k\|\vartheta_k\| -\barmu_k\bargamma_k\|\vartheta_k\|\\
& \leq -\frac{\kappa_{fcd}}{2}\|\bar{Q}_k\|\min\left\{\frac{\Delta_k}{\|R_k\|},\frac{\epsilon_s}{\|R_k\|},\frac{\|\bar{Q}_k\|}{\|\bar{W}_k\|}\right\} \\
& \quad  + \frac{ \|Z_k^T\barpsi_k\| \|\vartheta_k\|}{2 \|\bar{W}_k\|} + \frac{1}{2}\|\vartheta_k\|\frac{\min\{\Delta_k,\epsilon_s\}}{\|R_k\|}+\Upsilon_{3,j}\bargamma_k\|\vartheta_k\| -\barmu_k\bargamma_k\|\vartheta_k\|.     
\end{align*}
Therefore, we only need $\barmu_k\bargamma_k\|\vartheta_k\| \geq \frac{ \|Z_k^T\barpsi_k\| \|\vartheta_k\|}{2 \|\bar{W}_k\|} + \frac{1}{2}\|\vartheta_k\|\frac{\min\{\Delta_k,\epsilon_s\}}{\|R_k\|}+\Upsilon_{3,j}\bargamma_k\|\vartheta_k\|$,
equivalently, $\barmu_k \geq \frac{ \|Z_k^T\barpsi_k\|}{2\bargamma_k\|\bar{W}_k\|} + \frac{\min\{\Delta_k,\epsilon_s\}}{2\bargamma_k\|R_k\|} + \Upsilon_{3,j}$.
Since $\min\{\Delta_k,\epsilon_s\}/\|R_k\| > \|Z_k^T\barpsi_k\|/\|\bar{W}_k\|$ in this case, it suffices to have
$\barmu_k \geq \frac{\min\{\Delta_k,\epsilon_s\}}{\bargamma_k\|R_k\|}+\Upsilon_{3,j}$. As in \textbf{Case 1}, this holds provided~$\barmu_k \geq \min\{\Delta_{\max},\epsilon_s\}+ \frac{\kappa_c+\kappa_h+\Upsilon_s}{\zeta\sqrt{\kappa_{1,A}}}+\Upsilon_{3,j}$.

Combining the conclusions in \textbf{Cases 1} and \textbf{2}, we know that there exists $\barmu_{\beta_j} = \min\{\Delta_{\max},\epsilon_s\}+ \frac{\kappa_c+\kappa_h+\Upsilon_s}{\zeta\sqrt{\kappa_{1,A}}}+\Upsilon_{3,j}$ such that \eqref{eq:threshold_Predk} holds for every $k$ in stage $j$ as long as $\barmu_k\geq \barmu_{\beta_j}$. By the updating scheme of the merit parameter, within stage $j$ the merit parameter is increased only from a value below $\barmu_{\beta_j}$, so every value produced in stage $j$ is at most $\rho\barmu_{\beta_j}$. Since $\barmu_k$ is nondecreasing, $N_1=0$, and $\barmu_{\beta_i}$ is nondecreasing in $i$ (as $\kappa_{W,\beta_i}$ increases when $\beta_i$ decreases), induction over stages gives $\barmu_k\leq \max\{\barmu_0,\rho\barmu_{\beta_j}\}=:\hatmu_j$ for all $k$ in stage $j$. Finally, since each update multiplies $\barmu_k\geq\barmu_0$ by $\rho>1$, at most $\log_\rho(\hatmu_j/\barmu_0)$ updates occur in stages $1,\dots,j$. This completes the proof.
\end{proof}

\end{document}